\documentclass[11pt]{article}
\usepackage{amsmath,amsthm,amssymb,amsfonts,amscd}
\usepackage{amsxtra}
\usepackage{eucal}
\usepackage{mathrsfs}
\usepackage{color}
\usepackage{comment}
\usepackage{hyperref}
\usepackage{enumitem}
\usepackage{blkarray}
\usepackage{tikz}
\usepackage[a4paper,margin=2cm]{geometry}
\definecolor{forestgreen(traditional)}{rgb}{0.0, 0.27, 0.13}
\definecolor{forestgreen(web)}{rgb}{0.13, 0.55, 0.13}
\definecolor{airforceblue}{rgb}{0.36, 0.54, 0.66}
\newcommand\OPLUS{\sideset{}{_2}\bigoplus\limits}
\newcommand\abs[1]{\lvert#1\rvert}

\newcommand\GF{\operatorname{GF}}
\hypersetup{
  colorlinks,
  citecolor=forestgreen(web),
  linkcolor=airforceblue,
  urlcolor=magenta,
  pdftitle={Binary matroid whose bases form a graphic delta-matroid},
  pdfauthor={Duksang Lee and Sang-il Oum}}
\newenvironment{subproof}[1][\proofname]{%
  \begin{proof}[#1]%
}{%
  \end{proof}%
}

\newtheorem{theorem}{Theorem}[section]
\newtheorem{lemma}[theorem]{Lemma}
\newtheorem{claim}{Claim}

\newtheorem{proposition}[theorem]{Proposition}

\theoremstyle{definition}

\theoremstyle{remark}

\usepackage{authblk}
\title{Characterizing matroids whose bases form graphic delta-matroids}
\author[2,1]{Duksang Lee\thanks{Supported by the Institute for Basic Science (IBS-R029-C1).}}
\author[$*$1,2]{Sang-il Oum}
\affil[1]{Discrete Mathematics Group,
Institute for Basic Science (IBS),
Daejeon,~South~Korea}
\affil[2]{Department of Mathematical Sciences, KAIST, Daejeon, South~Korea}
\affil[ ]{Email: \texttt{duksang@kaist.ac.kr}, \texttt{sangil@ibs.re.kr}}
\date{\today}							%

\begin{document}
\maketitle
\begin{abstract}
We introduce delta-graphic matroids, which are matroids whose bases form graphic delta-matroids. The class of delta-graphic matroids contains graphic matroids as well as cographic matroids and is a proper subclass of the class of regular matroids. 
We give a structural characterization of the class of delta-graphic matroids. We also show that every forbidden minor for the class of delta-graphic matroids has at most $48$ elements.
\end{abstract}
\section{Introduction}
\label{sec:intro}
Bouchet~\cite{Bouchet1987sym} introduced delta-matroids which are set systems admitting a certain exchange axiom,  generalizing matroids.
Oum~\cite{Oum2009} introduced graphic delta-matroids as minors of binary delta-matroids having line graphs as their fundamental graphs and proved that bases of graphic matroids form graphic delta-matroids.
We introduce  \emph{delta-graphic} matroids as matroids whose bases form a graphic delta-matroid. 
Since the class of delta-graphic matroids 
is closed under taking dual and minors, it contains  both graphic matroids and cographic matroids.
Alternatively one can define delta-graphic matroids as binary matroids whose fundamental graphs are pivot-minors of line graphs. See Section~\ref{sec:prelim} for the definition of pivot-minors.

Our first theorem provides a structural characterization of delta-graphic matroids.
A \emph{wheel} graph~$W$ is a graph having a vertex $s$ adjacent to all other vertices such that $W\setminus s$ is a cycle. In a wheel graph, a \emph{spoke} is an edge incident with $s$ and a \emph{rim edge} is an edge which is not a spoke. 
A \emph{generalized wheel} is a matroid obtained from the cycle matroid of a minor of a wheel graph by a sequence of $2$-sums
where the other part of each $2$-sum is graphic if the corresponding basepoint is a rim edge
and cographic otherwise.

\begin{theorem}
\label{thm:main}
A connected matroid is delta-graphic if and only if it is a generalized wheel.
\end{theorem}
It follows that the class of delta-graphic matroids includes graphic matroids, cographic matroids, and the $2$-sum of a graphic matroid and a cographic matroid.
So all $3$-connected delta-graphic matroids are graphic or cographic. Furthermore, delta-graphic matroids are regular. Later we will see that not all regular matroids are delta-graphic.

As a side remark, there is a class of delta-matroids called \emph{regular} delta-matroids, see~{\cite[Chapter 6]{Geelen1996}}. As regular matroids are precisely matroids whose bases form regular delta-matroids, there are many open problems regarding regular delta-matroids extending theorems on regular matroids~{\cite[page 66]{Geelen1996}}. 

Our second theorem discusses forbidden minors for the class of delta-graphic matroids.
Since delta-graphic matroids are binary, by the well-quasi-ordering theorem for binary matroids announced by Geelen, Gerards, and Whittle~\cite{Geelen2014}, there exists a finite list of binary forbidden minors for the class of delta-graphic matroids. Together with $U_{2,4}$, there exists a finite list of forbidden minors for the class of delta-graphic matroids. However, their theorem does not provide an explicit list
or an upper bound on the size of forbidden minors. 
We prove that each forbidden minor for the class of delta-graphic matroids has at most $48$ elements. 

\begin{theorem}
\label{thm:exclude}
If $M$ is a matroid which is minor-minimally not delta-graphic, then $|E(M)|\leq 48$.
\end{theorem}
We remark that it is unknown whether the list of forbidden minors for the class of graphic delta-matroids is finite. Oum~\cite{Oum2012} conjectured that graphs are well-quasi-ordered under the pivot-minor relation and if this conjecture holds, then the list of forbidden minors for the class of graphic delta-matroids is finite.

This paper is organized as follows. In Section~\ref{sec:prelim}, we review some terminologies and results about graphs, matroids, and delta-matroids.
In Section~\ref{sec:graft}, we characterize grafts defining delta-graphic matroids. In Section~\ref{sec:str}, we describe the structure of delta-graphic matroids. We describe the canonical tree decomposition of connected delta-graphic matroids and prove Theorem~\ref{thm:main} in Section~\ref{sec:td}. In Section~\ref{sec:exclude}, we prove Theorem~\ref{thm:exclude}.

\section{Preliminaries}\label{sec:prelim}
In this paper, all graphs are finite and may have parallel edges and loops. A graph is \emph{simple} if it has no loops and no parallel edges. 
For an integer $k$, a graph $G$ is \emph{$k$-connected} if
it has more than $k$ vertices
and $G\setminus X$ is connected for all subsets $X$ of $V(G)$ with $\abs{X}<k$.
Given a set $A$ of vertices, a path $P$ is an \emph{$A$-path} if $V(P)\cap A$ is the set of endpoints of $P$. 

For a graph $G$, \emph{contracting} an edge $e$ in $G$ is an operation to obtain a new graph, denoted by $G/e$, from $G$ 
by deleting $e$ and identifying ends of $e$. A \emph{minor} of a graph $G$ is a graph obtained from $G$ by contracting edges and deleting vertices and edges. For two sets $A$ and $B$, let $A\triangle B=(A- B)\cup(B- A)$. 
Observe that $(A\triangle B)\triangle C=A\triangle(B\triangle C)$ for sets $A$, $B$, and $C$ and therefore we can write $A\triangle B\triangle C$.
\subsection{Matroids}
\label{subsec:matroid}
A pair $M=(E,\mathcal{B})$ is a \emph{matroid} if $E=E(M)$ is a finite set and $\mathcal{B}$ is a nonempty set of subsets of $E$ 
such that if $B_{1},B_{2}\in\mathcal{B}$ and $x\in B_{1}-B_{2}$, then there is an element $y\in B_{2}-B_{1}$ such that $(B_{1}-\{x\})\cup\{y\}\in\mathcal{B}$.
An element of $\mathcal{B}$ is called a \emph{base} of $M$. We denote the set of bases of $M$ by $\mathcal{B}(M)$. 
A subset of $E$ is \emph{independent} if it is a subset of some base
and dependent if it is not independent.
A \emph{circuit} is a minimally dependent set. We denote the set of circuits of $M$ by $\mathcal{C}(M)$. An element $e$ of $M$ is a \emph{loop} if $\{e\}$ is a circuit in $M$. 
The \emph{rank} $r_{M}(X)$ of a subset $X$ of $E$  is the size of a maximal independent subset of $X$.

The \emph{dual} matroid $M^{*}$ of $M$ is a matroid on the same ground set $E$ such that $\mathcal{B}(M^{*})=\{E-B: B\in\mathcal{B}(M)\}$. A \emph{coloop} of $M$ is a loop of $M^{*}$.  A set $D$ is a \emph{cocircuit} of $M$ if $D$ is a circuit of $M^{*}$. A \emph{hyperplane} of $M$ is a complement of a cocircuit of $M$. 

Let $e$ be an element of $M$.
Then $M\setminus e$ is a matroid on $E-\{e\}$ whose bases are maximal independent subsets of $E-\{e\}$.
This operation of obtaining $M\setminus e$ is called the \emph{deletion} of $e$.
For a subset $X$ of $E(M)$, we denote $M\setminus(E(M)- X)$ by $M|X$.
We define $M/e=(M^*\setminus e)^*$, which is a matroid obtained by \emph{contracting} $e$.
A \emph{minor} of $M$ is a matroid obtained by a sequence of deletions and contractions from $M$. 
A minor $N$ of $M$ is \emph{proper} if $E(N)$ is a proper subset of $E(M)$.
A class $\mathcal{M}$ of matroids is \emph{minor-closed} if for all $M\in\mathcal{M}$,  all minors of $M$ belong to $\mathcal{M}$.
A matroid $N$ is a \emph{forbidden} minor for a minor-closed class $\mathcal{M}$ if $N$ is not in $\mathcal{M}$ but all proper minors of $N$ are in $\mathcal{M}$.

The \emph{cycle matroid} $M(G)$ of a graph $G$ is a matroid on $E(G)$ such that $F\subseteq E(G)$ is a base if and only if $F$ is a maximal edge set which does not contain a cycle. A matroid $M$ is \emph{graphic} if $M$ is isomorphic to $M(G)$ for a graph $G$. A matroid is \emph{cographic} if its dual matroid is graphic.

\begin{lemma}[Whitney~\cite{Whitney1933}; see Oxley~{\cite[Theorem 5.3.1]{Oxley2011}}]
\label{lem:whitney}
Let $G$ be a $3$-connected graph and $H$ be a loopless graph without isolated vertices. If $M(G)$ is isomorphic to $M(H)$, then $G$ is isomorphic to $H$.
\end{lemma}

A matroid $M$ is \emph{represented by} a matrix $A$ over $\mathbb{F}$
if there is a bijection from $E(M)$ to the columns of $A$ such that a set $X$ is a base of $M$ if and only if the set of column vectors corresponding to~$X$ is a basis of the column space of $A$. 
A matroid $M$ is \emph{representable over} a field $\mathbb{F}$ 
if there is a matrix $A$ over $\mathbb{F}$ such that $M$ is represented by $A$. 
A matroid $M$ is \emph{binary} if it is representable over $\GF(2)$ and is \emph{regular} if it is representable over every field. 
The \emph{Fano matroid} $F_{7}$ is the matroid represented by a matrix $A_{7}$ over $\GF(2)$ where
\[
A_{7}=
\begin{bmatrix}
1 & 0 & 0 & 1 & 0 & 1 & 1 \\
0 & 1 & 0 & 1 & 1 & 0 & 1 \\
0 & 0 & 1 & 0 & 1 & 1 & 1 
\end{bmatrix}.
\]
\begin{theorem}[Tutte~\cite{Tutte1959}]
\label{thm:tutte}
A matroid $M$ is regular if and only if $M$ has no minor isomorphic to $U_{2,4}$, $F_{7}$, or $F^{*}_{7}$.
\end{theorem}

Let $\simeq$ be an equivalence relation on $E(M)$ such that $x\simeq y$ if either $x=y$ or $M$ has a circuit containing both $x$ and $y$. A \emph{component} of $M$ is an equivalence class of $E(M)$ by $\simeq$. A matroid is \emph{connected} if $E(M)$ is a component or $E(M)=\emptyset$.

The \emph{connectivity function} $\lambda_{M}$ of a matroid $M$ maps a subset $X$ of $E$ to $\lambda_{M}(X)=r_{M}(X)+r_{M}(E-X)-r_{M}(E)$. A partition $(X,Y)$ of $E$ is a \emph{$k$-separation} of a matroid $M$ if $\lambda_{M}(X)<k$ and $\min\{|X|,|Y|\}\geq k$. For $n>1$, a matroid is \emph{n-connected} if it has no $k$-separation for all $k<n$. Let $R_{10}$, $R_{12}$ be matroids represented by matrices $A_{10}$, $A_{12}$ over $\GF(2)$, respectively, where
\begin{align*}
A_{10}&=
\left[
\begin{array}{c|c}
\begin{matrix}
 1 & 0 & 0 & 0 & 0 \\
 0 & 1 & 0 & 0 & 0 \\
 0 & 0 & 1 & 0 & 0 \\
 0 & 0 & 0 & 1 & 0 \\
 0 & 0 & 0 & 0 & 1 \\
\end{matrix} &
\begin{matrix}
 1 & 1 & 0 & 0 & 1 \\
 1 & 1 & 1 & 0 & 0 \\
 0 & 1 & 1 & 1 & 0 \\
 0 & 0 & 1 & 1 & 1 \\
 1 & 0 & 0 & 1 & 1 \\
\end{matrix} \\
\end{array}
\right],
\\
A_{12}&=
\left[
\begin{array}{c|c}
\begin{matrix}
1 & 0 & 0 & 0 & 0 & 0  \\
0 & 1 & 0 & 0 & 0 & 0 \\
0 & 0 & 1 & 0 & 0 & 0  \\
0 & 0 & 0 & 1 & 0 & 0  \\
0 & 0 & 0 & 0 & 1 & 0  \\
0 & 0 & 0 & 0 & 0 & 1 
\end{matrix} &
\begin{matrix}
1 & 1 & 1 & 0 & 0 & 0  \\
1 & 1 & 0 & 1 & 0 & 0 \\
1 & 0 & 0 & 0 & 1 & 0  \\
0 & 1 & 0 & 0 & 0 & 1  \\
0 & 0 & 1 & 0 & 1 & 1  \\
 0 & 0 & 0 & 1 & 1 & 1 
\end{matrix}
\end{array}
\right].
\end{align*}
The following lemma will be used for proving Theorem~\ref{thm:exclude}.

\begin{lemma}[Seymour~{\cite[(14.2)]{Seymour1980}}]
\label{regular}
Let $M$ be a $3$-connected regular matroid. Then $M$ is graphic or cographic, or $M$ has a minor isomorphic to $R_{10}$ or $R_{12}$.
\end{lemma}

We now define the $2$-sum of matroids as in \cite{Seymour1980}.
For two matroids $M_{1}$, $M_{2}$ with $E(M_{1})\cap E(M_{2})=\{e\}$, if $e$ is neither a loop nor a coloop of $M_{1}$ or $M_{2}$, then the \emph{$2$-sum} of $M_{1}$ and $M_{2}$ is a matroid $M_{1}\oplus_{2}M_{2}$  on $E(M_{1})\triangle E(M_{2})$ such that 
\[\mathcal{C}(M_{1}\oplus_{2}M_{2})=\mathcal{C}(M_{1}\setminus e)\cup\mathcal{C}(M_{2}\setminus e)\cup\{C_{1}\triangle C_{2}:C_{1}\in\mathcal{C}(M_{1}), C_{2}\in\mathcal{C}(M_{2}), e\in C_{1}\cap C_{2}\}.\]
Though some authors require both $M_1$ and $M_2$ have at least $3$ elements, we will not require that in this paper. 

\begin{lemma}
\label{lem:circuittobase}
Let $M_{1}$, $M_{2}$ be matroids such that $E(M_{1})\cap E(M_{2})=\{e\}$ and   $e$ is neither a loop nor a coloop of $M_{1}$ or $M_{2}$. Then $\mathcal{B}(M_{1}\oplus_{2}M_{2})=\{B_{1}\triangle B_{2}\triangle\{e\} : B_{1}\in\mathcal{B}(M_{1}), B_{2}\in\mathcal{B}(M_{2}), e\in B_{1}\triangle B_{2}\}$. 
\end{lemma}
\begin{proof}
Let $M:=M_{1}\oplus_{2}M_{2}$.
First we prove that 
if $B_1$, $B_2$ are bases of $M_1$, $M_2$ respectively and $e\in B_1\triangle B_2$, then $B_1\triangle B_2\triangle \{e\}$ is a base of $M$. Let $B=B_1\triangle B_2\triangle \{e\}$.
By symmetry, we may assume that $e\in B_{1}$ and $e\notin B_{2}$. 
Suppose that $B$ is dependent in $M$. 
Since $B_{1}-\{e\}$ contains no circuits of $M_{1}\setminus e$ and $B_{2}$ contains no circuits of $M_{2}\setminus e$,  
$B$ contains $C_1\triangle C_2$ as a subset for a circuit $C_1$ of $M_1$ and a circuit $C_2$ of $M_2$ with $e\in C_1\cap C_2$.
It follows that $C_1$ is a subset of $B_1$, contradicting the assumption that $B_1$ is independent.
So $B$ is an independent set of $M$. 

Suppose that $B\cup \{f\}$ is independent in $M$ for some $f\in E(M)-B$.
If $f\in E(M_{2})$, then $B_{2}\cup\{f\}$ contains a circuit $C$ of $M_{2}\setminus e$, which is a circuit of $M$, contradicting the assumption that $B\cup \{f\}$ is independent in $M$.
So $f\in E(M_{1})$. Then $B_{1}\cup\{f\}$ contains a circuit $C'$ of $M_{1}$. If $e\notin C'$, then $C'$ is a circuit of $M_{1}\setminus e$, which is also a circuit of $M$, contradicting the assumption on $B\cup \{f\}$. 
So $e\in C'$. Let $C''\subseteq B_{2}\cup\{e\}$ be a circuit of $M_2$. Then $B\cup\{f\}$ contains $C'\cup C''-\{e\}$ and so it is dependent in $M$, contradicting the assumption. This proves that $B$ is a base of $M$.

Conversely, let us show that if $B$ is a base of $M$, then 
$B=B_1'\triangle B_2'\triangle \{e\}$ for a base $B_1'$ of $M_1$ and a base $B_2'$ of $M_2$ with $e\in B_1'\triangle B_2'$.
For $i\in\{1,2\}$, let $B_{i}=B\cap E(M_{i})$. Then $B_{i}$ contains no circuit of $M_{i}\setminus e$ and therefore $B_i$ is independent in $M_i$.

If $B_i\cup\{e\}$ is dependent in $M_i$ for each $i=1,2$, then let $C_i$ be a circuit of $
M_i$ that is a subset of $B_i\cup\{e\}$ and contains $e$. Then $C_1\triangle C_2$ is a circuit of $M$  but a subset of $B$, contradicting the assumption on $B$. Hence $B_{1}\cup\{e\}$ is independent in $M_1$ or $B_{2}\cup\{e\}$ is independent in $M_2$.

We first claim that if $B_1\cup \{e\}$ is independent in $M_1$, then $B_1\cup\{e\}$ is a base of $M_1$. If not, then there is $f\in E(M_1)-B_1-\{e\}$ such that $B_1\cup\{e,f\}$ is independent in $M_1$. 
However, $B\cup \{f\}$ contains a circuit $C$ of $M$ containing $f$. 
Then $C$ is not a circuit of $M_1\setminus e$, because otherwise  $C\subseteq B_{1}\cup\{e,f\}$, contradicting the assumption that $B_1\cup\{e,f\}$ is independent in $M_1$.
Thus, $C=C_1\triangle C_2$ for a circuit $C_1$ of $M_1$ and a circuit $C_2$ of $M_2$ with $C_1\cap C_2=\{e\}$. 
But still, $C_1\subseteq B_1\cup\{e,f\}$, contradicting the assumption that $B_1\cup \{e,f\}$ is independent in $M_1$.
This proves the first claim. 
By symmetry, if $B_2\cup \{e\}$ is independent in $M_2$, then $B_2\cup\{e\}$ is a base of $M_2$. 

We are now going to show that $B_1$ or $B_1\cup \{e\}$ is a base of $M_1$. 
Suppose not. 
Then 
$B_1\cup\{e\}$ is dependent in~$M_1$,
$B_2\cup\{e\}$ is a base of $M_2$,  and 
$B_1\cup \{g\}$ is independent in~$M_1$ for some $g\in E(M_1)-B_1$
because $B_1$ is independent in~$M_1$.
Since $B_{1}\cup\{e\}$ is dependent in $M_1$,
we have $g\neq e$ and so $B\cup\{g\}$ contains a circuit $D$ of $M$ containing~$g$.
Since $B_1\cup\{g\}$ is independent in~$M_1$, $D$ is not a circuit of $M_1\setminus e$. 
Thus, $D=D_1\triangle D_2$ for a circuit $D_1$ of $M_1$ and a circuit $D_2$ of $M_2$ with $D_1\cap D_2=\{e\}$. Then $D_2$ is a subset of $B_2\cup \{e\}$, contradicting the fact that $B_2\cup\{e\}$ is a base of~$M_2$. 
Therefore $B_1$ or $B_1\cup \{e\}$ is a base of $M_1$.
Similarly, $B_2$ or $B_2\cup \{e\}$ is a base of $M_2$.

Now we claim that either $B_1\cup\{e\}$ is dependent in $M_1$ or $B_2\cup\{e\}$ is dependent in $M_2$. Suppose not. Then $B_1\cup\{e\}$ is independent $M_1$ and $B_2\cup\{e\}$ is independent in $M_2$. Let $\tilde{B}$ be a base of $M_{1}$ which contains $B_{1}\cup\{e\}$.
Since $E(M_1)-\tilde{B}$ is a cobase of $M_1$, there exists a cocircuit $X$ of $M_1$ that is a subset of $(E(M_1)-\tilde{B})\cup\{e\}$ and contains $e$.
As $e$ is not a coloop of $M_1$, $X$ has an element $f\neq e$. 
Since $\tilde{B}$ is a base of $M_1$ and $f\notin \tilde{B}$, there exists a circuit $Y$ of $M_1$ that is a subset of $\tilde{B}\cup\{f\}$.
Then $Y$ must contain $e$ because $\abs{X\cap Y}\neq 1$ \cite[Proposition 2.1.11]{Oxley2011}.
Since such a circuit $Y$ is unique by the circuit exchange axiom, 
${B}_1\cup \{f\}$ is independent in $M_1$. 
As $B\cup\{f\}$ is dependent in~$M$, 
$B\cup\{f\}$ contains a circuit $C$ of $M$ containing $f$.
We can observe that $C$ is not a circuit of $M_1\setminus e$ because ${B}_1\cup\{f\}$ is independent in $M_1\setminus e$. 
Thus, $C=C_1\triangle C_2$ for a circuit $C_1$ of $M_1$ and a circuit $C_2$ of $M_2$ with $C_1\cap C_2=\{e\}$.
This contradicts the assumption that $B_2\cup\{e\}$ is independent in $M_2$, because $C_2\subseteq B_2\cup\{e\}$. 

Thus, either $B_1$ is a base of $M_1$ and $B_2\cup\{e\}$ is a base of $M_2$
or $B_1\cup\{e\}$ is a base of $M_1$ and $B_2$ is a base of $M_2$.
This concludes the proof.
\end{proof}
We list a few useful lemmas on $2$-sums of matroids.

\begin{lemma}[Seymour~{\cite[(2.6)]{Seymour1980}}]
\label{twosumminor}
If $M$ is the $2$-sum of $M_1$ and $M_2$, then 
both $M_1$ and $M_2$ are isomorphic to minors of $M$.
\end{lemma}

\begin{lemma}[see Oxley~{\cite[Proposition 7.1.22~(ii)]{Oxley2011}}]
\label{2sumconnected}
Let $M$, $N$ be matroids with $|E(M)|,|E(N)|\geq 2$. Then $M\oplus_{2}N$ is connected if and only if both $M$ and $N$ are connected.
\end{lemma}

For matroids $M_{1},\ldots,M_{k}$, we denote $M_{1}\oplus_{2}M_{2}\oplus_{2}\cdots\oplus_{2}M_{k}$ by $\OPLUS_{i=1}^{k}M_{i}$.

\begin{lemma}[Oxley~{\cite[Proposition 7.1.22 (i)]{Oxley2011}}]
\label{lem:twosumdual}
Let $M$, $N$ be matroids. Then $(M\oplus_{2}N)^{*}=M^{*}\oplus_{2}N^{*}$.
\end{lemma}

\begin{lemma}
\label{lem:twosomeminor}
Let $M$, $N$ be matroids such that $E(M)\cap E(N)=\{e\}$ and $e$ is neither a loop nor a coloop of $M$ or $N$. Then, for $f\in E(N)-\{e\}$, the following hold:
\begin{enumerate}[label=\rm (\roman*)]
\item If $e$ is neither a loop nor a coloop of $N\setminus f$, then $(M\oplus_{2}N)\setminus f=M\oplus_{2}(N\setminus f)$.
\item If $e$ is neither a loop nor a coloop of $N/f$, then $(M\oplus_{2}N)/f=M\oplus_{2}(N/f)$.
\end{enumerate}
\end{lemma}

\begin{proof}
  By Lemma~\ref{lem:twosumdual}, it is enough to prove (i). A set $C$ is a circuit of $(M\oplus_{2}N)\setminus f$
  if and only if 
  $C$ is a circuit of $M\setminus e$, a circuit of $N\setminus e\setminus f$, or 
  the symmetric difference of a circuit $C_1$ of $M$ and a circuit $C_2$ of $N$ with $C_1\cap C_2=\{e\}$ and $f\notin C_2$.  
  It is easy to see that this is equivalent to the statement that $C$ is a circuit of $M\oplus_2(N\setminus f)$. 
\end{proof}

\begin{lemma}
\label{2summinor}
Let $M$, $N$, $N'$ be matroids such that $E(M)\cap E(N)=\{e\}$ 
and $e$ is neither a loop nor a coloop of $M$ or $N'$.
If $N'$ is a minor of $N$, then $M\oplus_{2}N'$ is a minor of $M\oplus_{2}N$. 
\end{lemma}

\begin{proof}
We prove by induction on $|E(N)- E(N')|$. We may assume that there is an element $f\in E(N)- E(N')$. Then $N'$ is a minor of $N\setminus f$ or $N/f$. 
By Lemma~\ref{lem:twosumdual}, we may assume that $N'$ is a minor of $N\setminus f$ by symmetry.
By the assumption, $e$ is neither a loop nor a coloop of $N\setminus f$. Lemma~\ref{lem:twosomeminor} implies that $M\oplus_{2}(N\setminus f)=(M\oplus_{2}N)\setminus f$. By the induction hypothesis, $M\oplus_{2}N'$ is a minor of $M\oplus_{2}(N\setminus f)$ and therefore $M\oplus_{2}N'$ is a minor of $M\oplus_{2}N$.
\end{proof}

\paragraph{Tree decomposition of matroids.}
We review the concept of a tree decomposition of matroids introduced by 
Cunningham and Edmonds~\cite{Cunningham1980}.
A \emph{matroid-labelled tree} is a pair $(T,\rho)$ of a tree $T$ and a map $\rho$ mapping each vertex of $T$ to a matroid such that, for all $u\neq v\in V(T)$,
\begin{itemize}
\item if $u$ and $v$ are joined by an edge $e$ of $T$, then $E(\rho(u))\cap E(\rho(v))=\{e\}$ and $e$ is neither a loop nor a coloop of $\rho(u)$ or $\rho(v)$, 
\item if $u$, $v$ are not adjacent, then $E(\rho(u))\cap E(\rho(v))=\emptyset$.
\end{itemize}

For a matroid-labelled tree $(T,\rho)$ and a subtree $S$ of $T$ with $V(S)=\{v_{1},\ldots, v_{n}\}$, let $\rho(S)=\rho(v_{1})\oplus_{2}\rho(v_{2})\oplus_{2}\cdots\oplus_{2}\rho(v_{n})$. We require that a subtree has at least one vertex.

\medskip

A pair  $\{u, v\}$ of vertices of a matroid-labelled tree $(T,\rho)$ is \emph{bad} if $u$ is adjacent to $v$ in $T$, and both $\rho(u)$ and $\rho(v)$ are uniform matroids of rank $1$ or are uniform matroids of corank $1$. 

Let $(T,\rho)$ be a matroid-labelled tree and $e=uv$ be an edge of $T$. We define $(T,\rho)/e$ as a matroid-labelled tree $(T/e,\rho^{*})$ such that 
\[
\rho^{*}(x)=
\begin{cases}
\rho(x) & \text{if $x\in V(T)-\{u,v\}$,} \\
\rho(u)\oplus_{2}\rho(v) & \text{otherwise.}
\end{cases}
\]

\medskip

A \emph{tree decomposition} of a connected matroid $M$ is a matroid-labelled tree $(T,\rho)$ such that the following hold:
\begin{itemize}
\item $E(M)=\displaystyle\bigcup_{v\in V(T)} E(\rho(v))- E(T)$.
\item For $v\in V(T)$, $|E(\rho(v))|\geq 3$ unless $|E(M)|<3$, in which case $|V(T)|=1$.
\item $M=\rho(T)$.
\end{itemize}

A \emph{canonical tree decomposition} of a connected matroid $M$ is a tree decomposition $(T,\rho)$ which satisfies the following:
\begin{itemize}
\item For each $v\in V(T)$, $\rho(v)$ is $3$-connected or $\rho(v)$ is isomorphic to $U_{1,n}$ or $U_{n-1,n}$ for some $n\geq 3$.
\item There is no bad pair of vertices of $T$.
\end{itemize}

\begin{theorem}[Cunningham and Edmonds~{\cite[Theorem 18]{Cunningham1980}}]
\label{thm:canonical}
Every connected matroid has a canonical tree decomposition $(T,\rho)$. Moreover, $T$ is unique up to relabelling edges.
\end{theorem}

The following lemma is immediate from Lemma~\ref{lem:twosumdual}.

\begin{lemma}[see Oxley~{\cite[Lemma 8.3.8]{Oxley2011}}]
\label{dualtree}
Let $(T,\rho)$ be a canonical tree decomposition of a connected matroid $M$ and $\rho^{*}$ be a map from $V(T)$ to the class of all matroids such that $\rho^{*}(v)=\rho(v)^{*}$. Then $(T,\rho^{*})$ is a canonical tree decomposition of $M^{*}$.
\end{lemma}

\begin{lemma}
\label{subtreeminor}
Let $M$ be a connected matroid with a canonical tree decomposition $(T,\rho)$. If $S$ is a subtree of $T$, then $\rho(S)$ is isomorphic to a minor of $M$ with a canonical tree decomposition $(S,\rho|_{V(S)})$.
\end{lemma}

\begin{proof}
We proceed by induction on $|V(T)|-|V(S)|$. It is trivial when $V(T)=V(S)$. If $|V(T)|-|V(S)|=1$, let $v$ be a vertex in $V(T)- V(S)$. Since $S$ is connected, $v$ is a leaf of $T$.
Let $N:=\rho(v)$. Then $M=\rho(T)=\rho(T\setminus v)\oplus_{2}N$ and, by Lemma~\ref{twosumminor}, $\rho(T\setminus v)$ is isomorphic to a minor of $M$.

Suppose that $|V(T)|-|V(S)|>1$. Let $v$ be a leaf of $T$ which is not in $V(S)$. By the induction hypothesis, $\rho(S)$ is isomorphic to a minor of $\rho(T\setminus v)$ and $\rho(T\setminus v)$ is isomorphic to a minor of $M$. Therefore, $\rho(S)$ is isomorphic to a minor of $M$. It is easy to check that $(S,\rho|_{V(S)})$ is a canonical tree decomposition of $\rho(S)$.
\end{proof}

\paragraph{Graph minors and matroid minors.}
\label{par:minors}
We introduce some useful facts about graph minors and matroid minors.

\begin{lemma}[Hall~\cite{Hall1943}]
\label{rm1}
Every $3$-connected simple nonplanar graph is isomorphic to $K_{5}$ or has a minor isomorphic to $K_{3,3}$.
\end{lemma}

The following lemma is due to Brylawski, who proved it for matroids. We state its version for graphs, as in \cite[Theorem 5.1]{Ding2016}.

\begin{lemma}[Brylawski~\cite{Brylawski1972}]
\label{lem:Brylawski}
Let $G$ be a $2$-connected loopless graph and $H$ be a $2$-connected loopless minor of $G$. If $e\in E(G)- E(H)$, then $G/e$ or $G\setminus e$ is loopless, $2$-connected, and contains $H$ as a minor.
\end{lemma}

\begin{lemma}[folklore; see Oxley~{\cite[Proposition 4.3.6]{Oxley2011}}]
\label{3circuit}
Let $M$ be a matroid with $|E(M)|\geq 3$. Then $M$ is connected if and only if $M$ has a circuit or a cocircuit containing $X$ for each subset $X$ of $E(M)$ with $|X|=3$.
\end{lemma}

\begin{lemma}[Bixby and Coullard~\cite{Bixby1987}]
\label{lem: bixby}
Let $M$ be a $3$-connected matroid and $N$ be a $3$-connected minor of $M$ with $|E(N)|\geq 4$. If $e\in E(M)- E(N)$, then there exists a $3$-connected minor $N'$ of $M$ such that $e\in E(N')$, $|E(N')- E(N)|\leq 4$, and $N$ is a minor of $N'$.
\end{lemma}

\begin{lemma}[Oxley~{\cite[Corollary 3.7]{Oxley1987}}]
\label{MK4minor}
Let $M$ be a $3$-connected binary matroid such that $r(M), r(M^{*})\geq 3$. If $x,y,z$ are distinct elements of $E(M)$, then $M$ has a minor isomorphic to $M(K_{4})$ using $x,y,z$. 
\end{lemma}

\begin{lemma}[Truemper~\cite{Truemper1984}]
\label{lem:Truemper}
Let $M$ be a $3$-connected matroid and $N$ be a $3$-connected minor of $M$. Then there are $3$-connected matroids $M_{0},M_{1},\ldots,M_{n}$ such that $M_{0}=N$, $M_{n}=M$, and $M_{i}$ is a minor of $M_{i+1}$ with $|E(M_{i+1})|-|E(M_{i})|\leq 3$ for each $i\in\{0,\ldots,n-1\}$.
\end{lemma}

Lemma~\ref{lem:Truemper} implies the following corollary for graphs.

\begin{lemma}
\label{lem:Truemper2}
Let $G$ be a simple $3$-connected graph and $H$ be a simple $3$-connected minor of $G$. Then there are simple $3$-connected graphs $G_{0},G_{1},\ldots,G_{n}$ such that $G_{0}=G$, $G_{n}=H$, and $G_{i}$ is a minor of $G_{i+1}$ with $|E(G_{i+1})|-|E(G_{i})|\leq 3$ for each $i\in\{0,\ldots,n-1\}$.
\end{lemma}

See the book of Oxley~\cite{Oxley2011} for more about matroids.
\subsection{Delta-matroids} 
\label{subsec:deltamatroid}
Bouchet~\cite{Bouchet1987sym} introduced delta-matroids.
A \emph{delta-matroid} is a pair $(V,\mathcal{F})$ of a finite set $V$ and a nonempty set $\mathcal{F}$ of subsets of $V$ such that if $X,Y\in\mathcal{F}$ and $x\in X\triangle Y$, then there is $y\in X\triangle Y$ such that $X\triangle\{x,y\}\in\mathcal{F}$.
An element of $\mathcal{F}$ is called a \emph{feasible} set. 

\paragraph{Minors.}
For a delta-matroid $M=(V,\mathcal{F})$ and a subset $X$ of $V$, we can obtain a new delta-matroid $M\triangle X=(V,\mathcal{F}\triangle X)$ from $M$ 
where $\mathcal{F}\triangle X=\{F\triangle X : F\in\mathcal{F}\}$. This operation is called \emph{twisting}. A delta-matroid $N$ is \emph{equivalent} to $M$ if $N$ can be obtained from $M$ by twisting.

If there is a feasible subset of $V- X$, then $M\setminus X=(V- X,\mathcal{F}\setminus X)$ is a delta-matroid 
where $\mathcal{F}\setminus X=\{F\in\mathcal{F} : F\cap X=\emptyset\}$. This operation of obtaining $M\setminus X$ is called the \emph{deletion} of $X$ in $M$. A delta-matroid $N$ is a \emph{minor} of a delta-matroid $M$ if $N=M\triangle X\setminus Y$ for some subsets $X,Y$ of $V$. 
We remark that $Y$ could be empty. 

A delta-matroid is \emph{even} if $|X\triangle Y|$ is even for all feasible sets $X$ and $Y$. 
It is easy to see that all minors of even delta-matroids are even.

\paragraph{Fundamental graphs and connectedness.}
For a delta-matroid $M=(V,\mathcal{F})$ and a feasible set $F$ of~$M$, 
the \emph{fundamental graph} of $M$ with respect to $F$ is a graph $G_{M,F}$ without parallel edges on $V$ 
such that two vertices $u$ and $v$ are adjacent if and only if 
$F\triangle \{u,v\}$ is feasible. 
Note that if $F\triangle \{u\}$ is feasible, then in $G_{M,F}$, the vertex $u$ is incident with a loop and if $M$ is even, then $G_{M,F}$ is a simple graph.
Let $A_{G}$ be the adjacency matrix of $G$ over $\GF(2)$.

Let $M=(V,\mathcal{F})$ be a delta-matroid. A partition $(V_{1},V_{2})$ of $V$ is a \emph{separation} of $M$ if $V_{1},V_{2}\neq\emptyset$ and there exist delta-matroids $M_{1}=(V_{1},\mathcal{F}_{1})$ and $M_{2}=(V_{2},\mathcal{F}_{2})$ such that $\mathcal{F}=\{F_{1}\cup F_{2}: F_{1}\in\mathcal{F}_{1}, F_{2}\in\mathcal{F}_{2}\}$.

A delta-matroid $M$ is \emph{connected} if it has no separation. If $(V_{1},V_{2})$ is a separation of $M$, then it is a separation of every delta-matroid equivalent to $M$~{\cite[Proposition 3]{Bouchet1991}}. Hence, if $M$ is connected, then every delta-matroid equivalent to $M$ is connected. It is easy to observe that a matroid is connected if and only if its bases form a connected delta-matroid. 

\begin{lemma}[Geelen~{\cite[Lemma 3.5]{Geelen1996}}]
\label{evenfundgraph}
Let $M=(V,\mathcal{F})$ be an even delta-matroid, let $F\in\mathcal{F}$, and let $uv\in E(G_{M,F})$. Then, for $x,y\in V-\{u,v\}$, the following hold:
\begin{enumerate}[label=\rm (\arabic*)]
\item $ux\in E(G_{M,F})$ if and only if $vx\in E(G_{M,F\triangle\{u,v\}})$.
\item If $x,y\notin N_{G_{M,F}}(v)$, then $xy\in E(G_{M,F})$ if and only if $xy\in E(G_{M,F\triangle\{u,v\}})$.
\end{enumerate}
\end{lemma}

\begin{lemma}[Bouchet~\cite{Bouchet1991}]
\label{bouchetconnected}
For an even delta-matroid $M=(V,\mathcal{F})$ and a nonempty proper subset $X$ of $V$, the following are equivalent:
\begin{enumerate}[label=\rm (\roman*)]
\item\label{item:bc1} $(X, V- X)$ is a separation of $M$.
\item\label{item:bc2} There is no edge between $X$ and $V- X$ in $G_{M,F}$ for some $F\in\mathcal{F}$.
\item\label{item:bc3} There is no edge between $X$ and $V- X$ in $G_{M,F}$ for all $F\in\mathcal{F}$.
\end{enumerate}
\end{lemma}
We include a proof motivated by Cunningham in \cite[Theorem 3.7]{Geelen1996}.

\begin{proof}
We first prove that \ref{item:bc1} implies \ref{item:bc2}. By twisting, we may assume that $\emptyset\in\mathcal{F}$. 
By \ref{item:bc1}, there are delta-matroids $M_{1}=(X,\mathcal{F}_{1})$ and $M_{2}=(V- X,\mathcal{F}_{2})$ such that $\mathcal{F}=\{F_{1}\cup F_{2}:F_{1}\in\mathcal{F}_{1}, F_{2}\in\mathcal{F}_{2}\}$. 
Suppose that $G_{M,\emptyset}$ has an edge $xy$ for some $x\in X$ and $y\in V- X$. Then $\{x,y\}\in\mathcal{F}$.  By the construction of $\mathcal{F}_{1}$ and $\mathcal{F}_{2}$, we have $\{x\}\in\mathcal{F}_{1}$ and $\emptyset\in\mathcal{F}_{2}$. Hence, $\{x\}\in\mathcal{F}$, contradicting our assumption that $M$ is even. So \ref{item:bc2} holds.

We now show that \ref{item:bc2} implies \ref{item:bc3}. Let $\mathcal S$
be the set of all feasible sets $F$ such that there is no edge between $X$ and $V- X$ in $G_{M,F}$. By \ref{item:bc2}, there exists $F\in \mathcal S$. Suppose that \ref{item:bc3} does not hold. Then there exists $ F'\in\mathcal{F}- \mathcal S$. Among all such sets, we choose $F'$ with minimum $|F\triangle F'|$. Choose $u\in F\triangle F'$. 
Then there is an element $v\in F\triangle F'$ such that $u\neq v$ and $F'\triangle\{u,v\}\in\mathcal{F}$. 
Since $|F\triangle(F'\triangle\{u,v\})|<|F\triangle F'|$, $F'\triangle\{u,v\}\in \mathcal S$ and therefore $uv\in E(G_{M,F'\triangle\{u,v\}})$. 
By symmetry, we may assume that $u,v\in X$. 
Since $F'\notin \mathcal S$, there exist $x\in X$, $y\in V- X$ such that $xy\in E(G_{M,F'})$.  If $x\in\{u,v\}$, then assume $u=x$ without loss of generality. Then, by Lemma~\ref{evenfundgraph}, $vy\in E(G_{M,F'\triangle\{u,v\}})$, contradicting our assumption. Therefore, we may assume that in $G_{M,F'}$, neither $u$ nor $v$ has neighbors in $V- X$. Then $u,v\notin N_{G_{M,F'}}(y)$ and, by Lemma~\ref{evenfundgraph},  $uv\in E(G_{M,F'\triangle\{x,y\}})$. It implies that $xy\in E(G_{M,F'\triangle\{u,v\}})$, contradicting our assumption. So \ref{item:bc3} holds.

So it remains to prove that \ref{item:bc3} implies \ref{item:bc1}. By twisting, we may assume that $\emptyset\in\mathcal F$. Let $\mathcal{F}_{1}=\{F_{1}\in\mathcal{F}:F_{1}\subseteq X\}$ and $\mathcal{F}_{2}=\{F_{2}\in\mathcal{F}:F_{2}\subseteq V- X\}$. Suppose that there is no edge between $X$ and $V- X$ in $G_{M,F}$ for all $F\in\mathcal{F}$. Suppose that there exist $F_{1},F_{2}\in\mathcal{F}$ such that $(F_{1}\cap X)\cup (F_{2}\cap (V- X))\notin\mathcal{F}$. Among such pairs, we choose $F_{1}$, $F_{2}$ such that $|(F_{1}\triangle F_{2})\cap X|$ is minimum. Since $(F_{1}\triangle F_{2})\cap X\neq\emptyset$, there exists $x\in (F_{1}\triangle F_{2})\cap X$. So there is $y\in F_{1}\triangle F_{2}$ such that $F_{2}\triangle\{x,y\}\in\mathcal{F}$. So $xy\in E(G_{M,F_{2}})$ and $y\in X$ by \ref{item:bc3}. However, $|F_{1}\triangle(F_{2}\triangle\{x,y\})|<|F_{1}\triangle F_{2}|$ and $(F_{1}\cap X)\cup(F_{2}\cap(V- X))=(F_{1}\cap X)\cup((F_{2}\triangle\{x,y\})\cap(V- X))\in\mathcal{F}$, which is a contradiction. So $\{F_{1}\cup F_{2} : F_{1}\in\mathcal{F}_{1}, F_{2}\in\mathcal{F}_{2}\}\subseteq\mathcal{F}$.

If there exists $F\in\mathcal{F}$ such that $F\cap X\notin\mathcal{F}$ or $F\cap (V- X)\notin\mathcal{F}$, we choose a minimal set $F$ among all such sets. Observe that $F\cap X, F\cap(V- X)\neq\emptyset$. We choose $x_{1}\in F\cap X$. Then there exists $y_{1}\in F$ such that $F':=F\triangle\{x_{1},y_{1}\}\in\mathcal{F}$. Since $M$ is even, $x_{1}\neq y_{1}$ and $x_{1}y_{1}\in E(G_{M,F})$. By \ref{item:bc3}, $y_{1}\in F\cap X$. Since $F'\subsetneq F$, we have $F'\cap (V- X)\in\mathcal{F}$ by the minimality assumption of $F$. Therefore, $F\cap (V- X)\in\mathcal{F}$ and, by symmetry, $F\cap X\in\mathcal{F}$, contradicting our assumption. Hence, $\mathcal{F}\subseteq\{F_{1}\cup F_{2} : F_{1}\in\mathcal{F}_{1}, F_{2}\in\mathcal{F}_{2}\}$.

Let $M_{1}=(V,\mathcal{F}_{1})$ and $M_{2}=(V,\mathcal{F}_{2})$. Then $M_{1}$ and $M_{2}$ are delta-matroids and $(V,V- X)$ is a separation of $M$.
\end{proof}

\paragraph{Representable delta-matroids.}
Let $V$ be a finite set. For a $V\times V$ symmetric or skew-symmetric matrix $A$ over a field $\mathbb{F}$ and a subset $X$ of $V$, let $A[X]$ be an $X\times X$ submatrix of $A$. Let \(\mathcal{F}(A)=\left\{ X\subseteq V : A[X] \text{ is nonsingular} \right\}\). 
We assume that $A[\emptyset]$ is nonsingular and so $\emptyset\in\mathcal{F}(A)$.
Bouchet~\cite{Bouchet1988} proved that $(V,\mathcal{F}(A))$ is a delta-matroid. 
A delta-matroid is \emph{representable over} a field~$\mathbb{F}$ if it is a delta-matroid whose set of feasible sets is $\mathcal{F}(A)\triangle X$ for a skew-symmetric or symmetric matrix $A$ over $\mathbb{F}$ and a subset $X$ of $V$.  
Since $\emptyset\in\mathcal{F}(A)$ for all symmetric or skew-symmetric matrices $A$, it is natural to define representable delta-matroids with twisting so that the empty set is not necessarily feasible in representable delta-matroids.

A delta-matroid is \emph{binary} if it is representable over $\GF(2)$. 

\paragraph{Pivoting.} For a finite set $V$ and a symmetric or skew-symmetric $V\times V$ matrix 
\[
A=
\begin{blockarray}{ccc}
  & X & Y \\
 \begin{block}{c(cc)}
 X& \alpha &\beta \\
 Y & \gamma & \delta \\
\end{block}
\end{blockarray}~,
\]
if $A[X]=\alpha$  is nonsingular, then let
\[
A*X=
\begin{blockarray}{ccc}
  & X & Y \\
 \begin{block}{c(cc)}
 X& \alpha^{-1} &\alpha^{-1}\beta \\
 Y & -\gamma \alpha^{-1} & \delta-\gamma\alpha^{-1}\beta \\
\end{block}
\end{blockarray}~.
\]
This operation is called \emph{pivoting}. Tucker \cite{Tucker1960} proved that when $A[X]$ is nonsingular, $A*X[Y]$ is nonsingular if and only if $A[X\triangle Y]$ is nonsingular for each subset $Y$ of $V$. Hence, if $M=(V,\mathcal{F}(A))$ and $X$ is a feasible set of $M$, then $M\triangle X=(V,\mathcal{F}(A*X))$. 

For a simple graph $G$ and an edge $uv\in E(G)$, a graph $G\wedge uv$ is obtained from $G$ by \emph{pivoting} an edge $uv$ if $A_{G\wedge uv}=A_G*\{u,v\}$. A combinatorial description can be found in \cite{Oum2004}.
A \emph{pivot-minor} of a simple graph $G$ is a graph obtained from $G$ by pivoting and deleting vertices repeatedly. The following result gives a relation between minors of delta-matroids and pivot-minors of graphs.

\begin{lemma}[Bouchet \cite{Bouchet1988}]
\label{pminor}
Let $M$ be an even binary delta-matroid. Then, a simple graph $G$ is a pivot-minor of a fundamental graph of $M$ if and only if $G$ is a fundamental graph of a minor of $M$.
\end{lemma}

\paragraph{Twisted matroids.} A \emph{twisted matroid} $N$ is a delta-matroid such that $N=(E,\mathcal{B}\triangle X)$ for some matroid $M=(E,\mathcal{B})$ and some set $X\subseteq E$.  

Let us first see that every (delta-matroid) minor of a twisted matroid is a twisted matroid. It is enough to prove that, if $N'=N\triangle X\setminus\{e\}$ for a matroid $N=(E,\mathcal{B})$ and a set $X\subseteq E$, then $N'$ is a twisted matroid. We may assume that $e\notin X$ by replacing $N$ with $N^{*}$ and $X$ with $E- X$. Then $N'=N\triangle X\setminus\{e\}=(N\setminus e)\triangle X$.

The following lemma characterizes twisted matroids. We remark that Bouchet~\cite{Bouchet1987matching} proved a weaker version for even delta-matroids.

\begin{lemma}[Geelen {\cite[Theorem 3.11]{Geelen1996}}]
\label{g1} 
Let $M$ be a delta-matroid and $F$ be a feasible set of $M$. 
Then $G_{M,F}$ is bipartite if and only if $M$ is a twisted matroid.
\end{lemma}

Duchamp implicitly characterizes the class of twisted matroids 
in terms of forbidden (delta-matroid) minors.

\begin{lemma}[Duchamp~\cite{Duchamp1992}]
\label{MK3}
A delta-matroid $M$ is a twisted matroid if and only if it has no minor isomorphic to $D_{1}=(\{1\},\{\emptyset,\{1\}\})$ or 
$MK_{3}=(\{1,2,3\},\{\emptyset,\{1,2\},\{1,3\},\{2,3\}\})$.
\end{lemma}

The following lemma appeared in Chun, Moffatt, Noble, and Rueckriemen~\cite{Chun2019} with a short proof and is also implied by an old theorem of Holzmann, Norton, and Tobey~\cite{Holzmann1973} and was stated in the Ph.D.~thesis of Geelen~\cite[line 18 of page 15]{Geelen1996} without a proof.

\begin{lemma}[Chun, Moffatt, Noble, and Rueckriemen~{\cite[Theorem 3.11]{Chun2019}}]
\label{tmatroid}
Let $M_{1}$, $M_{2}$ be matroids on $E$. If $M_{1}$ is connected and $\mathcal{B}(M_{1})\triangle X=\mathcal{B}(M_{2})$ for some $X\subseteq E$, then $X=\emptyset$ or $X=E$. In other words, $M_{1}=M_{2}$ or $M_{1}=M_{2}^{*}$.
\end{lemma}  

\paragraph{Graphic delta-matroids.} Oum~\cite{Oum2009} introduced graphic delta-matroids. Let $G=(V,E)$ be a graph and $T$ be a subset of the set of vertices of~$G$. A subgraph~$H$ of~$G$ is called a \emph{T-spanning} subgraph of~$G$ if $V(G)=V(H)$ and for each component $C$ of~$H$, either
\begin{enumerate}
\item[(i)] $|V(C)\cap T|$ is odd, or
\item[(ii)] $V(C)\cap T=\emptyset$ and $G[V(C)]$ is a component of $G$.
\end{enumerate}
A \emph{graft} is a pair $(G,T)$ of a graph $G$ and a subset $T$ of the set of vertices of $G$. A set $F$ of edges of $G$ is \emph{feasible} in $(G,T)$ if it is the edge set of a $T$-spanning forest of $G$. Let $\mathcal{G}(G,T)=(E(G),\mathcal{F})$ where $\mathcal{F}$ is the set of all feasible sets of $(G,T)$. Oum \cite{Oum2009} proved that $\mathcal{G}(G,T)$ is an even binary delta-matroid. A delta-matroid $M$ is \emph{graphic} if there exist a graft $(G,T)$ and a subset $X$ of $E(G)$ such that $M=\mathcal{G}(G,T)\triangle X$. 
We remark that deleting isolated vertices from $G$ does not change $\mathcal G(G,T)$.

The following proposition is one of the major motivations to introduce graphic delta-matroids.
The \emph{binary line graph} $\operatorname{BL}(G)$ of a graph $G$ is a simple graph on $E(G)$ such that two edges $e$, $f$ of $G$ are adjacent in $\operatorname{BL}(G)$
if and only if $e$ and $f$ are non-loops that share exactly one end.
Note that if $G$ is simple, then the binary line graph of $G$ is equal to the line graph of $G$.
\begin{proposition}[Oum~\cite{Oum2009}]
  A binary delta-matroid is graphic if and only if 
  its fundamental graph is a pivot-minor of a binary line graph of some graph.
\end{proposition}

For a graft $(G,T)$ and an edge $e$ of $G$, let $(G,T)\setminus e=(G\setminus e,T)$. This operation is called a \emph{deletion} of an edge $e$ in a graft $(G,T)$. For an isolated vertex $v$ of $G$, let $(G,T)\setminus v$ be $(G\setminus v, T-\{v\})$. This is a deletion of a vertex $v$. 
For an edge $e=uv$ of $G$, we define $(G,T)/e$ to be a graft $(G/e,T')$ such that 
\[
T'=
\begin{cases}
(T-\{u,v\})\cup\{e^{*}\} & \text{if }\lvert T\cap \{u,v\} \rvert=1, \\
T-\{u,v\} & \text{otherwise,}
\end{cases}
\]
where $e^*$ is the vertex corresponding to both ends of $e$ in $G/e$. 
This operation is called a \emph{contraction} of an edge $e$ in a graft $(G,T)$. 
Note that if $e$ is a loop,  then $(G,T)/e=(G,T)\setminus e$.
A graft $(G',T')$ is a \emph{minor} of a graft $(G,T)$ if $(G',T')$ is obtained from $(G,T)$ by a sequence of deletions and contractions. Oum~\cite{Oum2009} proved minors of graphic delta-matroids are graphic.
Let $\kappa(G,T)$ be the number of components of $G$ having no vertices in $T$. An edge $e$ is a \emph{T-bridge} of a graft $(G,T)$ if $\kappa(G\setminus e,T)>\kappa(G,T)$. An edge $e=uv$ is a \emph{$T$-tunnel} of a graft $(G,T)$ if $V(C)\cap T=\{u,v\}$ for the component $C$ of $G$ containing $e$.

\begin{lemma}[Oum {\cite[Proposition 8]{Oum2009}}]
\label{lem:Tbridge}
All feasible sets of $\mathcal{G}(G,T)$ contain an edge $e$ if and only if $e$ is a $T$-bridge.
\end{lemma}

\begin{lemma}[Oum {\cite[Proposition 9]{Oum2009}}] 
\label{lem:graftdeletion}
Let $(G,T)$ be a graft. For an edge $e$ of $G$, %
\[
\mathcal{G}((G,T)\setminus e)=
\begin{cases}
\mathcal{G}(G,T)\setminus\{e\} & \text{if $e$ is not a $T$-bridge,} \\
\mathcal{G}(G,T)\triangle\{e\}\setminus\{e\} & \text{otherwise.}
\end{cases}
\]
\end{lemma}

\begin{lemma}[Oum {\cite[Proposition 10]{Oum2009}}]
\label{lem:Ttunnel}
No feasible set of $\mathcal{G}(G,T)$ contains an edge $e$ if and only if $e$ is a loop or a $T$-tunnel.
\end{lemma}

\begin{lemma}[Oum {\cite[Proposition 11]{Oum2009}}] 
\label{lem:graftcontraction}
Let $(G,T)$ be a graft. For an edge $e$ of $G$, 
\[
\mathcal{G}((G,T)/e)=
\begin{cases}
\mathcal{G}(G,T)\triangle\{e\}\setminus\{e\} & \text{if $e$ is neither a $T$-tunnel nor a loop,} \\
\mathcal{G}(G,T)\setminus\{e\} & \text{otherwise.}
\end{cases}
\]
\end{lemma}

For a good coverage of delta-matroids, we refer the reader to the survey by Moffatt~\cite{Moffatt2019}.

\section{Delta-graphic matroids and grafts}
\label{sec:graft}
A matroid $M$ is \emph{delta-graphic} if there exist a graft $(G,T)$ and a subset $X$ of $E(G)$ such that the set of bases of $M$ is equal to $\mathcal{F}(\mathcal{G}(G,T))\triangle X$. 
The following two lemmas imply that the class of delta-graphic matroids is 
clsoed under taking dual and minor.

\begin{lemma}
\label{lem:dualclosed}
If a matroid $M$ is delta-graphic, then so is $M^{*}$.
\end{lemma}
\begin{proof}
If $\mathcal{B}(M)=\mathcal{F}(\mathcal{G}(G,T))\triangle X$
for a graft $(G,T)$ and a subset $X$ of $E(G)$, 
then $\mathcal{B}(M^{*})=\mathcal{F}(\mathcal{G}(G,T))\triangle (E(G)- X)$.
\end{proof}

\begin{lemma}
\label{tgraphicminorclosed}
If a matroid $M$ is delta-graphic, then every (matroid) minor of $M$ is delta-graphic.
\end{lemma}
\begin{proof}
Let $M$ be a delta-graphic matroid. Then there exist a graft $(G,T)$ and a subset $X$ of $E(G)$ such that $\mathcal{B}(M)=\mathcal{F}(\mathcal{G}(G,T))\triangle X$. 
We claim that every minor $N$ is delta-graphic. By Lemma~\ref{lem:dualclosed}, we may assume that $N=M\setminus e$. 
If $e$ is a coloop of $M$, then $M\setminus e= M/e=(M^*\setminus e)^*$ and so we replace $M$ by $M^*$ to assume that $e$ is not a coloop of $M$.
Then, we have 
\[
\mathcal{B}(M\setminus e)=\mathcal{F}(\mathcal{G}(G,T))\triangle X\setminus\{e\}=
\begin{cases}
(\mathcal{F}(\mathcal{G}(G,T))\setminus\{e\})\triangle X & \text{if $e\notin X$,} \\
(\mathcal{F}(\mathcal{G}(G,T))\triangle\{e\}\setminus\{e\})\triangle (X\triangle\{e\}) & \text{if $e\in X$.} 
\end{cases}
\]
Let us consider the case when $e\notin X$ first. Since $e$ is not a coloop of $M$, $e$ is not a $T$-bridge by Lemma~\ref{lem:Tbridge}. Then, $\mathcal{B}(M\setminus e)=\mathcal{F}(\mathcal{G}((G,T)\setminus e))\triangle X$ by Lemma~\ref{lem:graftdeletion}. 

If $e\in X$, since $e$ is not a coloop of $M$, $e$ is neither a $T$-tunnel nor a loop by Lemma~\ref{lem:Ttunnel}. So $\mathcal{B}(M\setminus e)=\mathcal{F}(\mathcal{G}((G,T)/e))\triangle(X\triangle\{e\})$ by Lemma~\ref{lem:graftcontraction}. 
\end{proof}

\begin{lemma}
\label{Tlessthan2}
Let $G$ be a graph and $T$ be a subset of $V(G)$.
If $|T|\leq 2$, then $\mathcal{G}(G,T)$ is the cycle matroid of a graph that is obtained from $G$ by identifying all vertices in $T$.
\end{lemma}

\begin{proof}
  Let $G'$ be the graph obtained from $G$ by identifying all vertices in $T$.
The lemma is trivial when $|T|\leq 1$. If $|T|=2$, then $F$ is the edge set of a $T$-spanning forest of $G$ if and only if $F$ is the edge set of a maximal spanning forest of $G'$. The conclusion follows easily.
\end{proof}
  
\begin{lemma}
\label{graftconnected}
Let $G$ be a graph with no isolated vertices and $T$ be a subset of $V(G)$.
If $\mathcal{G}(G,T)$ is connected, then $G$ is connected.
\end{lemma}
\begin{proof}
Suppose that $G$ is disconnected. Then there is a partition $(X,Y)$ of $V(G)$ such that $X,Y\neq\emptyset$ and there is no edge between $X$ and $Y$. Let $G_{1}=G[X],~G_{2}=G[Y], ~T_{1}=T\cap X, ~T_{2}=T\cap Y$.
Then $E(G)=E(G_{1})\cup E(G_{2})$ and $E(G_{1})\cap E(G_{2})=\emptyset$. Since $G$ has no isolated vertices, we have $E(G_{1}), E(G_{2})\neq\emptyset$. We show that $\mathcal{F}(\mathcal{G}(G,T))=\{F_{1}\cup F_{2} : F_{1}\in\mathcal{F}(\mathcal{G}(G_{1},T_{1})), F_{2}\in\mathcal{F}(\mathcal{G}(G_{2},T_{2}))\}$. 

Let $F_{1}$ be an edge set of a $T_{1}$-spanning forest and $F_{2}$ be an edge set of a $T_{2}$-spanning forest. Then, it is obvious that $F_{1}\cup F_{2}$ is the set of edges of a $T$-spanning forest. 
Conversely, let $F$ be an edge set of a $T$-spanning forest. Then $F_{1}=F\cap E(G_{1})$ is the set of edges of a $T_{1}$-spanning forest and $F_{2}=F\cap E(G_{2})$ is the set of edges of a $T_{2}$-spanning forest. 
Then, $(E(G_{1}), E(G_{2}))$ is a separation of $\mathcal{G}(G,T)$, contradicting our assumption.
\end{proof}

Recall that, in Lemma~\ref{MK3}, $MK_{3}$ is a delta-matroid defined by $(\{1,2,3\},\{\emptyset,\{1,2\},\{1,3\},\{2,3\}\})$. Let $\Delta_{1}=(K_{3},V(K_{3}))$, $\Delta_{2}=(K_{1,3},V(K_{1,3}))$, $\Delta_{3}=(K_{1,3},S)$ where $S$ is the set of vertices of degree~$1$ in $K_{1,3}$.

\begin{lemma}
\label{MK3graft}
For a graft $(G,T)$ with no isolated vertices, $\mathcal{G}(G,T)$ is equivalent to $MK_{3}$ up to isomorphism if and only if $(G,T)$ is isomorphic to one of $\Delta_{1}$, $\Delta_{2}$, $\Delta_{3}$.
\end{lemma}

\begin{proof}
  Let $\mathcal F$ be the set of feasbile sets in $\mathcal G(G,T)$.
Let us prove the forward direction because the backward direction is trivial. Suppose that $\mathcal{G}(G,T)\triangle X=MK_{3}$ for a set $X\subseteq\{1,2,3\}$.
We deduce that $|E(G)|=3$ and $\mathcal{F}\triangle X=\left\{\emptyset,\left\{1,2\right\},\left\{2,3\right\},\left\{1,3\right\}\right\}$. If $|X|$ is  even, then $\mathcal{F}=\left\{\emptyset,\left\{1,2\right\},\left\{2,3\right\},\left\{1,3\right\}\right\}$ and $\mathcal{F}=\left\{\left\{1\right\},\left\{2\right\},\left\{3\right\},\left\{1,2,3\right\}\right\}$ otherwise. 
Since $MK_{3}$ is connected, $\mathcal{G}(G,T)$ is connected and by Lemma~\ref{graftconnected}, $G$ is connected. Since $\{1,2,3\}\in\mathcal{F}$ or $\{\{1,2\},\{1,3\},\{2,3\}\}\subseteq\mathcal{F}$, $G$ cannot contain loops or parallel edges. So $G$ is simple and is isomorphic to one of $K_{3}$, $P_{4}$, $K_{1,3}$. 

If $\mathcal{F}=\left\{\emptyset,\left\{1,2\right\},\left\{2,3\right\},\left\{1,3\right\}\right\}$, then $T=V(G)$ because $\emptyset$ is feasible. It is easy to check that $\mathcal{G}(P_{4},V(P_{4}))$ is not isomorphic to $\mathcal G(G,T)$. So $(G,T)$ is not isomorphic to $(P_{4},V(P_{4}))$ and therefore $(G,T)$ is isomorphic to $(K_{3},V(K_{3}))$ or $(K_{1,3},V(K_{1,3}))$.

If $\mathcal{F}=\left\{\left\{1\right\},\left\{2\right\},\left\{3\right\},\left\{1,2,3\right\}\right\}$, then $G$ cannot be isomorphic to $K_{3}$ since $\left\{1,2,3\right\}$ is a feasible set.
So $G$ is isomorphic to $P_{4}$ or $K_{1,3}$. If $|T|\leq 2$, then by Lemma \ref{Tlessthan2}, $\mathcal{G}(G,T)$ is a matroid, contradicting the fact that $\mathcal{F}$ has feasible sets of different size. Observe that $|T|$ is odd because $\{1,2,3\}$ is feasible. Hence, $|T|=3$ and we can check that $G$ is isomorphic to $K_{1,3}$ and $T=S$ where $S$ is the set of vertices of degree $1$ in $K_{1,3}$. Therefore, $(G,T)$ is isomorphic to $(K_{1,3},S)$.
\end{proof}

\begin{lemma}
\label{inter}
Let $(G,T)$ be a graft and $u$ be a vertex of $G$. If there exist distinct vertices $u_{1},u_{2},u_{3}\in T-\{u\}$ and three internally disjoint paths $P_{1},P_{2},P_{3}$ from $u$ to $u_{1},u_{2},u_{3}$ respectively, then $(G,T)$ has a minor isomorphic to $\Delta_{2}$ or $\Delta_{3}$.
\end{lemma}

\begin{proof}
For each $i\in \{1,2,3\}$, choose a vertex $v_{i}$ in $(V(P_{i})\cap T)-\{u\}$ such that $d_{P_{i}}(u,v_{i})$ is minimum and let $P_{i}'$ be a subpath of $P_{i}$ from $u$ to $v_{i}$.
Let $E_{i}:=E(P_{i}')$ and $e_{i}$ be an edge of $E(P_{i}')$ which is incident with $u$ for $i\in \{1,2,3\}$. 
By deleting all edges in $E(G)-(E_{1}\cup E_{2}\cup E_{3})$ and contracting all edges in $(E_{1}\cup E_{2}\cup E_{3})- \{e_{1},e_{2},e_{3}\}$, we get a minor isomorphic to $\Delta_{2}$ if $u\in T$ and $\Delta_{3}$ otherwise.
\end{proof}

\begin{lemma}
\label{lem:3connectedgraft}
Let $G$ be a graph and $T$ be a subset of $V(G)$.
If $G$ is $3$-connected and $|T|\geq 3$ then $(G,T)$ has a minor isomorphic to  
$\Delta_{2}$ or $\Delta_{3}$.
\end{lemma}

\begin{proof}
Let $w$ be a vertex of $G$ such that $|T- \left\{w\right\}|\geq 3$. For a set $A=\left\{v_{1},v_{2},v_{3}\right\}\subseteq T-\{w\}$, there exist 3 internally disjoint paths $P_{1},P_{2},P_{3}$ from $w$ to $v_{1},v_{2},v_{3}$ by the theorem of Menger. 
The conclusion follows by Lemma~\ref{inter}.
\end{proof}

To describe the structure of grafts resulting delta-matroids with bipartite fundamental graphs, we introduce cyclic decompositions of grafts.
Let $G=(V,E)$ be a graph and $T$ be a subset of $V$. A \emph{cyclic decomposition} of $(G,T)$ is a pair $(H, \mathcal{B})$, where $H$ is a bipartite graph whose maximum degree is at most 2 and $\mathcal{B}=\left\{B_{x} : x\in V(H)\right\}$ is a collection of subsets of $V$ satisfying the following:
\begin{enumerate}[label=\rm (C\arabic*)]
\item\label{item:c1} $\displaystyle\bigcup_{x\in V(H)} B_{x}=V$.
\item\label{item:c2} If $u,v\in V$ are adjacent in $G$, then there is a vertex $x$ of $H$ such that $\left\{u,v\right\}\subseteq B_{x}$.
\item\label{item:c3} For distinct $x,y\in V(H)$, $B_{x}\cap B_{y}\subseteq T$ and $|B_{x}\cap B_{y}|$ is equal to the number of edges joining $x,y$ in $H$. 
\item\label{item:c4} For all $x\in V(H)$, $|T\cap B_{x}|\leq 2$.
\end{enumerate}

For a cyclic decomposition $(H,\mathcal{B})$ of $(G,T)$ and $B\in\mathcal{B}$, $G[B]$ is a \emph{bag} of $(H,\mathcal{B})$.

\begin{lemma}
\label{lem:minorcyclic}
If a graft $(G,T)$ has a cyclic decomposition, then every minor of $(G,T)$ also has a cyclic decomposition.
\end{lemma}
\begin{proof}
Suppose that $(G,T)$ has a cyclic decomposition $(H,\mathcal{B})$. It is enough to show that both $(G,T)\setminus e$ and $(G,T)/e$ have cyclic decompositions for every $e\in E(G)$. Given $e=uv\in E(G)$, it is obvious that $(G,T)\setminus e$ has a cyclic decomposition $(H,\mathcal{B})$. 

So it remains to show that $(G,T)/e=(G/e,T^{*})$ has a cyclic decomposition. 
We may assume that $e$ is not a loop.
Let $e^{*}$ be the vertex obtained by contracting $e$. Let $z$ be a vertex of $H$ such that $\{u,v\}\subseteq B_{z}$. For all $x\in V(H)$, let
\[
B_{x}'=
\begin{cases}
(B_{x}-\{u,v\})\cup\{e^{*}\} & \text{if $\{u,v\}\cap B_{x}\neq\emptyset$,} \\
B_{x} & \text{otherwise.}
\end{cases}
\]
and $\mathcal{B}'$ be the set $\{B_{x}' : x\in V(H)\}$. 

If $\{u,v\}\nsubseteq T$, then it is easy to check that $(H,\mathcal{B}')$ is a cyclic decomposition of $(G,T)/e=(G/e,T^{*})$. So we can assume that $u,v\in T$ and $e^{*}\notin T^{*}$. Now we do some case analysis.

If $z$ is an isolated vertex in $H$, then $(H,\mathcal{B}')$ is a cyclic decomposition of $(G,T)/e$ obviously. 

Suppose that $z$ has degree $1$ in $H$. Let $y$ be the neighbor of $z$ in $H$. 
We may assume that $u\in B_{y}\cap B_{z}$ by symmetry. 
Since $H$ is bipartite, by \ref{item:c3}, $u\notin B_x$ for all $x\in V(H)-\{y,z\}$.
Let $f=yz\in E(H)$ and $f^{*}$ be the vertex obtained by contracting $f$ in $H$. Let $B_{f^{*}}=(B_{y}\cup B_{z}\cup\{e^{*}\})-\{u,v\}$. 

We claim that $(H/f, \mathcal{B}\cup\left\{B_{f^{*}}\right\}-\left\{B_{y},B_{z}\right\})$ is a cyclic decomposition of $(G,T)/e$. It is trivial to check \ref{item:c1}. Since $N_{G/e}(e^{*})\subseteq B_{f^{*}}$, \ref{item:c2} holds. \ref{item:c3} holds since $B_{x}\cap B_{y}=B_{x}\cap B_{f^{*}}$ for all $x\in V(H)-\{y,z\}$. Also \ref{item:c4} holds trivially.

It remains to consider the case when $z$ is a vertex of degree $2$ in $H$. If $z$ is incident with parallel edges $e_{1},e_{2}$, then let $y$ be the unique  neighbor of $z$ in $H$. Then $H[\{y,z\}]$ is a component of $H$ which is a cycle of length $2$. Let $H'$ be a graph $((V(H)-\{y,z\})\cup\{z'\},E(H)-\{e_{1},e_{2}\})$ and $B_{z'}=(B_{y}\cup B_{z}\cup\{e^{*}\})-\{u,v\}$. Then $(H',(\mathcal{B}-\{B_{y},B_{z}\})\cup\{B_{z'}\})$ is a cyclic decomposition of $(G,T)/e$.

Now if $z$ has two distinct neighbors $x,y$, then $B_{x}\cap B_{y}\cap B_{z}$ is empty because $H$ has no cycle of length $3$. 
We may assume that $u\in B_{x}\cap B_{z}$ and $v\in B_{y}\cap B_{z}$ by symmetry. Let $f_{1}=xz\in E(H)$ and $f_{2}=yz\in E(H)$. Also let $f^{*}$ be a vertex obtained by contracting edges $f_{1},f_{2}$ in $H$ and $B_{f^{*}}=(B_{x}\cup B_{y}\cup B_{z}\cup\{e^{*}\})-\left\{u,v\right\}$. Then $(H/\{f_{1},f_{2}\},(\mathcal{B}\cup\left\{B_{f^{*}}\right\})-\left\{B_{x},B_{y},B_{z}\right\})$ is a cyclic decomposition of $(G,T)/e$. So we prove this lemma.
\end{proof}

A cyclic decomposition $(H,\mathcal{B})$ of a graft (G,T) is \emph{nice} if the following hold.
\begin{enumerate}[label=\rm (N\arabic*)]
\item\label{item:n1} $G[B_{u}]$ is connected for all $u\in V(H)$.
\item\label{item:n2} If $|B_{v}\cap T|\leq 1$ for a vertex $v\in V(H)$, then $v$ is an isolated vertex of $H$.
\end{enumerate}

\begin{lemma}
\label{nice}
Let $G$ be a graph and $T$ be a subset of $V(G)$.
If a graft $(G,T)$ has a cyclic decomposition, then $(G,T)$ has a nice cyclic decomposition.
\end{lemma}

\begin{proof}

Among all cyclic decompositions of $(G,T)$ with the minimum number of disconnected bags, we choose a cyclic decomposition $(H,\mathcal{B})$ with minimum $|\mathcal{B}|=|V(H)|$. We will show that $(H,\mathcal{B})$ is a nice cyclic decomposition of $(G,T)$. Let $t$ be the number of disconnected bags of $(H,\mathcal{B})$.

Suppose that $t\geq 1$. Let $u$ be a vertex of $H$ such that $G[B_{u}]$ is disconnected. Then, for a positive integer $m$, let $C_{0},\ldots,C_{m}$ be components of $G[B_{u}]$. Let $k$ be the number of components of $G[B_{u}]$ which intersects with $\bigcup_{v\neq u}B_{v}$. By relabelling components, we may assume that $V(C_{i})\cap \bigcup_{v\neq u}B_{v}\neq\emptyset$ for all $0\leq i<k$ and $V(C_{i})\cap \bigcup_{v\neq u}B_{v}=\emptyset$ if $i\geq k$. Let $H'$ be a graph obtained from $H$ by adding  isolated vertices $u_{1},\ldots,u_{m}$ and $\mathcal{B}'$ be a set $\{B_{x}' : x\in V(H)\}$ such that, for all $x\in V(H')$,
\[
B_{x}'=
\begin{cases}
V(C_{0}) & \text{if $x=u$,} \\
V(C_{i}) & \text{if $x=u_{i}$ for $1\leq i\leq m$,} \\
B_{x} & \text{otherwise.}
\end{cases}
\] 
If $k\leq 1$, then $(H',\mathcal{B}')$ is a cyclic decomposition of $(G,T)$ with the number of disconnected bags less than $t$, contradicting our assumption.
Therefore, we may assume that $k\geq 2$ and $\deg_{H}(u)=2$ by \ref{item:c4}.
This implies that $k=2$ and $|V(C_{0})\cap\bigcup_{v\neq u}B_{v}|=|V(C_{1})\cap\bigcup_{v\neq u}B_{v}|=1$. Let $x\in V(C_{0})\cap\bigcup_{v\neq u}B_{v}$ and $y\in V(C_{1})\cap\bigcup_{v\neq u}B_{v}$. Since $H$ is bipartite, there is a unique vertex $w$ of $H$ such that $y\in V(C_{1})\cap B_{w}$. By \ref{item:c3}, there is an edge $e=uw$ in $H$. Let $H''=(H'\backslash e)+g$ where $g$ is an edge joining $u_{1}$ and $w$. Then $(H'',\mathcal{B}')$ is a cyclic decomposition of $(G,T)$ with less than $t$ disconnected bags, contradicting our assumption.
So $t=0$ and $(H,\mathcal{B})$ satisfies \ref{item:n1}. 

Now to show \ref{item:n2}, we claim that if $v$ is a non-isolated vertex of $H$, then $\abs{B_v\cap T}\ge 2$.
If $v$ is incident with two parallel edges, then for the unique neighbor $u$ of $v$, by \ref{item:c3}, $|T\cap B_{v}|\geq |B_{u}\cap B_{v}|=2$.
If $v$ has two distinct neighbors $u$, $u'$ in $H$, then 
since $H$ is bipartite, $B_{u}\cap B_{u'}=\emptyset$ by \ref{item:c3}
and therefore $\abs{B_{v}\cap T}\ge \abs{B_{v}\cap B_u}+\abs{B_v\cap B_{u'}}\ge 2$ by \ref{item:c3}.
If $\deg_{H}(v)=1$, then let $u\in N_{H}(v)$. Let $H'=H\setminus v$ and, for all $z\in V(H')$, let 
\[
B''_{z}=
\begin{cases}
B_{u}\cup B_{v} & \text{if $z=u$,} \\
B_{z} & \text{otherwise.}
\end{cases}
\] and let $\mathcal{B}'':=\{B_{x}'' : x\in V(H')\}$. 
If $\abs{B_v\cap T}\le 1$, then $(H',\mathcal{B}'')$ is a cyclic decomposition of $(G,T)$ with no disconnected bags, contradicting the assumption that $\abs{\mathcal B}$ is chosen to be minimum.
Therefore, $\abs{B_v\cap T}\ge 2$.
\end{proof}

\begin{lemma}
\label{evencycle}
Let $G$ be a connected graph and $T$ be a subset of $V(G)$ such that $|T|\geq 3$. If G has a cycle containing at least $3$ vertices of $T$ and $(G,T)$ has no minor isomorphic to $\Delta_{1}, \Delta_{2}$,or $\Delta_{3}$, then $(G,T)$ has a nice cyclic decomposition $(H,\mathcal{B})$ such that $H$ is an even cycle of length at least 4.
\end{lemma}

\begin{proof}

Let $C$ be a cycle of $G$ containing at least $3$ vertices of $T$. 

\begin{claim}\label{claim:1}
  $C$ contains all vertices in $T$. 
\end{claim}
\begin{subproof}
  Suppose that $x\in T- V(C)$. Let $y$ be a vertex of $C$ such that $d_{G}(x,y)$ is minimum and let $P_{1}$ be a shortest path from $x$ to $y$. Since $|V(C)\cap T|\geq 3$, there exist two vertices $a,b\in (V(C)\cap T)-\{y\}$. Let $P_{2},P_{3}$ be internally disjoint subpaths of $C$ from $y$ to $a,b$, respectively. Since $P_{1},P_{2},P_{3}$ are internally disjoint paths from $y$ to $x,a,b\in T$, by Lemma \ref{inter}, $(G,T)$ has a minor isomorphic to $\Delta_{2}$ or $\Delta_{3}$, contradicting our assumption.
\end{subproof}

We also know that $|V(C)\cap T|$ is even, because otherwise $(G,T)$ has a minor isomorphic to $\Delta_{1}$. For $k\geq 2$, let $v_{1},v_{2},\ldots,v_{2k}$ be all vertices of $V(C)\cap T$ in the cyclic order of $C$. By Claim~\ref{claim:1}, $T=\{v_{1},\ldots, v_{2k}\}$. Let $v_{2k+1}:=v_{1}$ and, for $i\in \{1,2,\ldots,2k\}$, let $Y_{i}$ be a path from $v_{i}$ to $v_{i+1}$ in $C$ such that no internal vertex is in $T$. 

\begin{claim}\label{claim:2}
   For every $V(C)$-path $P$, $V(C)\cap V(P)\subseteq V(Y_{i})$ for some $1\leq i\leq 2k$. 
\end{claim}
\begin{subproof}
Suppose that there is a $V(C)$-path $P$ such that $V(C)\cap V(P)\nsubseteq V(Y_{i})$ for all $1\leq i\leq 2k$.  Let $x$, $y$ be ends of $P$ and $u_1$, $u_2$ be the two vertices of $V(C)\cap (T-\{x\})$ such that $C$ has a path $Q_i$ from $u_i$  to $x$ without internal vertices in $T$ for $i=1,2$.
Since $\{x,y\}\nsubseteq V(Y_{i})$ for all $1\leq i\leq 2k$, $y\notin V(Q_{1})$ and $y\notin V(Q_{2})$. There exists $w\in T-(V(Q_{1})\cup V(Q_{2}))$ because $|T|\geq 4$. Let $Q_{3}$ be a subpath of $C$ from $y$ to $w$ which does not intersect with $Q_{1}$ or $Q_{2}$. Let $Q_{3}'$ be a path $wQ_{3}yPx$. Then, by applying Lemma~\ref{inter} to $Q_{1},Q_{2},Q_{3}'$, we get a minor isomorphic to $\Delta_{2}$ or $\Delta_{3}$, contradicting our assumption.
\end{subproof}

For a positive integer $m$, let $X_{1},X_{2},\ldots,X_{m}$ be components of $G\setminus V(C)$. Since $G$ is connected, $N_{G}(V(X_{i}))\subseteq V(C)$ is nonempty for $1\leq i\leq m$. By Claim~\ref{claim:2}, for every $1\leq i\leq m$, there is $1\leq j\leq 2k$ such that $N_{G}(V(X_{i}))\subseteq V(Y_{j})$.

For $1\leq i\leq 2k$, let 
\[
A_{i}=
\begin{cases}
\{j\in \{1,2,\ldots,m\} : N_{G}(V(X_{j}))\subseteq V(Y_{1})\} & \text{if $i=1$,} \\
\{j\in \{1,2,\ldots,m\} : N_{G}(V(X_{j}))\subseteq V(Y_{i})\}-\bigcup_{p=1}^{i-1}A_{p} & \text{otherwise.}
\end{cases}
\]
Let $H=x_{1}\cdots x_{2k}x_{1}$ be a cycle of length $2k$. Let $B_{x_{i}}:=V(Y_{i})\cup \bigcup_{j\in A_{i}} V(X_{j})$ for $1\leq i\leq 2k$ 
and $\mathcal B=\{B_{x_i}:1\le i\le 2k\}$.

\smallskip
We claim that $(H,\mathcal{B})$ is a nice cyclic decomposition of $(G,T)$.
By Claim~\ref{claim:2}, \ref{item:c1} holds. Claim~\ref{claim:2} implies that $\{A_{i}:1\leq i\leq 2k\}$ is a partition of $\{1,\ldots,m\}$. So each component of $G\setminus V(C)$ is contained in a unique $B_{x_{i}}$ for $1\leq i\leq 2k$. So if $i\neq j$, then $B_{x_{i}}\cap B_{x_{j}}\subseteq V(Y_{i})\cap V(Y_{j})\subseteq T$ and \ref{item:c3} holds.
Since $|B_{x_i}\cap T|=\{v_{i},v_{i+1}\}$ for all $1\leq i\leq 2k$, we can check easily that \ref{item:c4} holds. To show \ref{item:c2}, let $u$, $v$ be adjacent vertices of $G$. If $u,v\in V(C)$, it is obvious that $u,v\in V(Y_{i})\subseteq B_{x_{i}}$ for some $1\leq i\leq 2k$.
So we may assume that $u\notin V(C)$. Then $u\in V(X_{i})$ for some $1\leq i\leq m$. Then there is $1\leq j\leq 2k$ such that $i\in A_{j}$. So $u$ and all neighbors of $u$ are in $B_{x_{j}}$. Hence, $\{u,v\}\subseteq B_{x_{j}}$. So \ref{item:c2} holds. By the construction of $(H,\mathcal{B})$, \ref{item:n1} and \ref{item:n2} hold obviously. So $(H,\mathcal{B})$ is a nice cyclic decomposition.
\end{proof}

\begin{lemma}
\label{path}
Let $G$ be a connected graph and $T$ be a subset of $V(G)$ such that $|T|\geq 3$. If $(G,T)$ has no minor isomorphic to $\Delta_{1}$, $\Delta_{2}$,or $\Delta_{3}$ and every cycle of $G$ contains at most $2$ vertices of $T$, then $(G,T)$ has a nice cyclic decomposition $(H,\mathcal{B})$ such that $H$ is a path of length at least $1$.
\end{lemma}

\begin{proof} 
First, we prove the following claim.

\begin{claim}\label{claim:3}
There is a path $P$ such that $T\subseteq V(P)$.
\end{claim}
\begin{subproof}
Let $P$ be a path with both ends $w_{1}$, $w_{2}$ in $T$ and maximum $|V(P)\cap T|$. If there is a vertex~$x$ in $T- V(P)$, then choose a vertex~$y$ of $P$ whose $d_{G}(x,y)$ is minimum. Let $Q$ be a shortest path from $x$ to~$y$. By the choice of $P$, $y\neq w_{1}$ and $y\neq w_{2}$. Let $Q_{1}$, $Q_{2}$ be paths of $P$ from $y$ to $w_{1}$, $w_{2}$, respectively. Then, by applying Lemma \ref{inter} to $Q$, $Q_{1}$, $Q_{2}$, we get a minor isomorphic to $\Delta_{2}$ or $\Delta_{3}$, contradicting our assumption. 
\end{subproof}

Let $P$ be a path containing $V(T)=\{x_{1},\ldots, x_{n}\}$, such that $x_1$, 
$\ldots$, $x_n$ occur in order on $P$ for $n\geq 3$ and $x_1$, $x_{n}$ are the ends of $P$. For $1\leq i\leq n-1$, let $Y_{i}$ be the subpath of $P$ from $x_{i}$ to $x_{i+1}$.

\medskip

\begin{claim}\label{claim:4}
For every $V(P)$-path $R$, $V(P)\cap V(R)\subseteq V(Y_{i})$ for some $1\leq i\leq n-1$.
\end{claim}
\begin{subproof}
Let $y_{1}$, $y_{2}$ be ends of $V(R)$ and $Q$ be the subpath of $P$ from $y_{1}$ to $y_{2}$. Suppose that there is a vertex $v$ in $(V(Q)-\{y_{1},y_{2}\})\cap T$. We may assume that $d_{P}(x_{1},y_{1})<d_{P}(x_{1},y_{2})$. If $y_{1}=x_{1}$ and $y_{2}=x_{n}$, then $Q=P$ and a cycle $C=x_{1}Px_{n}Rx_{1}$ contains all vertices of $T$, contradicting the assumption that no cycle contains more than $2$ vertices in $T$. So we can assume $y_{1}\neq x_{1}$.
Then by Lemma~\ref{inter} applied to paths $Q_{1}=x_{1}Py_{1}$, $Q_{2}=y_{1}Pv$, $Q_{3}=y_{1}Ry_{2}Px_{n}$, we obtain a minor isomorphic to $\Delta_{2}$ or $\Delta_{3}$, contradicting our assumption. Hence, $Q$ has no internal vertex in $T$. Therefore, $y_{1},y_{2}\in V(Y_{i})$ for some $1\leq i\leq n-1$.
\end{subproof}

For an integer $m\geq 1$, let $X_{1},\cdots, X_{m}$ be components of $G\setminus V(P)$. Since $G$ is connected, we know that $\emptyset\neq N_{G}(V(X_{i}))\subseteq V(P)$ for every $i\in\{1,\ldots,m\}$. By Claim~\ref{claim:4}, for every $1\leq i\leq m$, there is $1\leq j\leq n-1$ such that $N_{G}(V(X_{i}))\subseteq V(Y_{j})$.

For $1\leq i\leq n-1$, let 
\[
A_{i}=
\begin{cases}
\{j\in \{1,2,\ldots,m\} : N_{G}(V(X_{j}))\subseteq V(Y_{1})\} & \text{if $i=1$,} \\
\{j\in \{1,2,\ldots,m\} : N_{G}(V(X_{j}))\subseteq V(Y_{i})\}-\bigcup_{p=1}^{i-1}A_{p} & \text{otherwise.}
\end{cases}
\]

Let $H=z_{1}\cdots z_{n-1}$ be a path of length $n-2$.
Let $B_{z_{i}}:=V(Y_{i})\cup\bigcup_{j\in A_{i}}V(X_{j})$ for $1\leq i\leq n-1$ and let 
$\mathcal{B}=\{B_{z_{i}}: 1\leq i\leq n-1\}$.
\smallskip

We claim that $(H,\mathcal{B})$ is a nice cyclic decomposition of $(G,T)$.
By Claim~\ref{claim:4}, \ref{item:c1} holds. 
Claim~\ref{claim:4} implies that $\{A_{i}:1\leq i\leq n-1\}$ is a partition of $\{1,\ldots,m\}$. So each component of $G\setminus V(P)$ is contained in a unique $B_{z_{i}}$ for $1\leq i\leq n-1$. 
So if $i\neq j$, then $B_{z_{i}}\cap B_{z_{j}}\subseteq V(Y_{i})\cap V(Y_{j})\subseteq T$ and \ref{item:c3} holds. 
Since $|B_{z_i}\cap T|=\{x_{i},x_{i+1}\}$ for all $1\leq i\leq n-1$, we can check easily that \ref{item:c4} holds. 
To see \ref{item:c2}, let $u$, $v$ be adjacent vertices of $G$. If $u,v\in V(P)$, it is obvious that $u,v\in V(Y_{i})\subseteq B_{z_{i}}$ for some $1\leq i\leq n-1$. 
So we may assume that $u\notin V(P)$. Then $u\in V(X_{i})$ for some $1\leq i\leq m$. 
So $u$ and all neighbors of $u$ are contained in $B_{z_{j}}$ for some $j\in \{1,2,\ldots,n-1\}$. 
So $\{u,v\}\subseteq B_{z_{j}}$ and \ref{item:c2} holds. By the construction of $(H,\mathcal{B})$, \ref{item:n1} and \ref{item:n2} hold obviously. So $(H,\mathcal{B})$ is a nice cyclic decomposition.
\end{proof}

\begin{lemma}
\label{lem:D123}
Let $(G,T)$ be a graft admitting a cyclic decomposition $(H,\mathcal{B})$.
Then $(G,T)$ does not contain a minor isomorphic to $\Delta_{1}$, $\Delta_{2}$, or $\Delta_{3}$.
\end{lemma}
\begin{proof}
By Lemma~\ref{lem:minorcyclic}, it is enough to show that $\Delta_{1}$, $\Delta_{2}$, and $\Delta_{3}$ do not have cyclic decompositions. Suppose that $\Delta_{1}=(K_{3},V(K_{3}))$ has a cyclic decomposition $(H,\mathcal{B})$. Let $V(K_{3})=\{u,v,w\}$. By \ref{item:c2}, there exist vertices $x$, $y$, $z$ of $H$ such that $\{u,v\}\subseteq B_{x}$, $\{v,w\}\subseteq B_{y}$, $\{u,w\}\subseteq B_{z}$. By \ref{item:c4}, $\{u,v\}=B_{x}$, $\{v,w\}=B_{y}$, $\{u,w\}=B_{z}$. By \ref{item:c3}, $xyzx$ is a cycle of length $3$ in $H$, contradicting the fact that $H$ is bipartite. So $\Delta_{1}$ has no cyclic decomposition.

Let $S=\{u,v,w\}$ be the set of leaves of $K_{1,3}$ and $s$ be the internal vertex of $K_{1,3}$.

Suppose that $\Delta_{2}=(K_{1,3},V(K_{1,3}))$ has a cyclic decomposition $(H,\mathcal{B})$. By \ref{item:c2}, there exist vertices $x$, $y$, $z$ of $H$ such that $\{s,u\}\subseteq B_{x}$, $\{s,v\}\subseteq B_{y}$, $\{s,w\}\subseteq B_{z}$. By \ref{item:c4}, $\{s,u\}=B_{x}$, $\{s,v\}=B_{y}$, $\{s,w\}=B_{z}$. By \ref{item:c3}, $xyzx$ is a cycle of length $3$ in $H$, contradicting the fact that $H$ is bipartite. So $\Delta_{2}$ has no cyclic decomposition. 

Suppose that $\Delta_{3}=(K_{1,3},S)$ has a cyclic decomposition $(H,\mathcal{B})$. By \ref{item:c2}, there exist vertices $x$, $y$, $z$ of $H$ such that $\{s,u\}\subseteq B_{x}$, $\{s,v\}\subseteq B_{y}$, $\{s,w\}\subseteq B_{z}$. Since $B_{x}\cap B_{y}=B_{x}\cap B_{z}=B_{y}\cap B_{z}\nsubseteq T$, by \ref{item:c3}, $x=y=z$. So $V(K_{1,3})\subseteq B_{x}$, contradicting \ref{item:c4}. So $\Delta_{3}$ has no cyclic decomposition. 
\end{proof}

\begin{proposition}
\label{prop:graftcondition}
Let $M$ be a connected graphic delta-matroid such that $M=\mathcal{G}(G,T)\triangle X$ for a graft $(G,T)$, 
where $X\subseteq E(G)$ and $G$ has no isolated vertices.
Then, $M$ is a twisted matroid if and only if at least one of the following holds:
\begin{enumerate}[label=\rm (G\arabic*)]
\item\label{item:g1} $|T|\leq 2$.
\item\label{item:g2} $(G,T)$ admits a nice cyclic decomposition $(H,\mathcal{B})$ with an even cycle $H$ of length at least 4.
\item\label{item:g3} $(G,T)$ admits a nice cyclic decomposition $(H,\mathcal{B})$ with a path $H$ of length at least 1.
\end{enumerate}
\end{proposition}

\begin{proof}
First, we prove the backward direction. If $|T|\leq 2$, then by Lemma~\ref{Tlessthan2}, $\mathcal{G}(G,T)=M(G')$ for some graph $G'$ and therefore $M$ is a twisted matroid. If \ref{item:g2} or \ref{item:g3} holds, then $(G,T)$ has a cyclic decomposition. So by Lemma~\ref{lem:D123}, $(G,T)$ has no minor isomorphic to $\Delta_{1}$, $\Delta_{2}$, or $\Delta_{3}$. By Lemma~ \ref{MK3graft}, $\mathcal{G}(G,T)$ has no (delta-matroid) minor equivalent to $MK_{3}$. Since $\mathcal{G}(G,T)$ is even~\cite{Oum2009}, it has no minor isomorphic to $D_{1}$. Therefore, by Lemma~\ref{MK3}, $\mathcal{G}(G,T)$ is a twisted matroid and so is $M$. 

Now let us prove the forward direction. Since $M$ is a twisted matroid, $\mathcal{G}(G,T)$ is a twisted matroid. As $M$ is connected, $\mathcal{G}(G,T)$ is connected and by Lemma~\ref{graftconnected}, $G$ is connected. Since $\mathcal{G}(G,T)$ is a twisted matroid, by Lemma~\ref{MK3}, $\mathcal G(G,T)$ has no (delta-matroid) minor equivalent to $MK_{3}$. Hence, by Lemma~\ref{MK3graft}, $(G,T)$ has no minor isomorphic to $\Delta_{1}$, $\Delta_{2}$, $\Delta_{3}$. 

Suppose that $\abs{T}\geq 3$. By Lemmas~\ref{evencycle} and \ref{path}, \ref{item:g2} holds if $G$ has a cycle containing at least 3 vertices of $T$ and \ref{item:g3} holds otherwise.
\end{proof}

Let $\mathcal{C}$ be the class of all matroids $M$ such that 
$\mathcal{B}(M)=\mathcal{F}(\mathcal{G}(G,T))\triangle X$ for a graft $(G,T)$ and $X\subseteq E(G)$ where $(G,T)$ admits a nice cyclic decomposition $(H,\mathcal{B})$ with an even cycle $H$ of length at least 4.
Similarly, let $\mathcal{P}$ be the class of all matroids $M$ such that $\mathcal{B}(M)=\mathcal{F}(\mathcal{G}(G,T))\triangle X$ for a graft $(G,T)$ and $X\subseteq E(G)$ where $(G,T)$ admits a nice cyclic decomposition $(H,\mathcal{B})$ with a path $H$ of length at least 1.

\begin{proposition}
\label{prop:tgraphicgraft}
A matroid $M$ is delta-graphic if and only if every component $C$ of $M$ satisfies at least one of the following conditions:
\begin{enumerate}[label=\rm (\arabic*)]
\item $M|C$ is graphic or cographic.
\item\label{item:tg2} $M|C\in\mathcal{C}$.
\item\label{item:tg3} $M|C\in\mathcal{P}$.
\end{enumerate}
\end{proposition}
\begin{proof}
Since the backward direction is obvious, we prove the forward direction.
Let $M$ be a delta-graphic matroid. We may assume that $M$ is connected. There is a graft $(G,T)$ such that $\mathcal{B}(M)=\mathcal{G}(G,T)\triangle X$ for some $X\subseteq E(G)$. We may assume that $G$ has no isolated vertices. Since $M$ is a twisted matroid and a graphic delta-matroid, \ref{item:g1}, \ref{item:g2}, or \ref{item:g3} holds by Proposition~\ref{prop:graftcondition}. Since \ref{item:g2} or \ref{item:g3} implies \ref{item:tg2} or \ref{item:tg3}, we may assume that $|T|\leq 2$. By Lemma~\ref{Tlessthan2}, $\mathcal{G}(G,T)=M(G')$ for a graph~$G'$. Since $\mathcal{B}(M)=\mathcal{G}(G,T)\triangle X=\mathcal{B}(M(G'))\triangle X$ and $M$ is connected, $M=M(G')$ or $M=M^{*}(G')$ by Lemma~\ref{tmatroid}. Hence, $M$ is graphic or cographic. 
\end{proof}

\section{Structure of delta-graphic matroids}
\label{sec:str}
Now we aim to describe the structure of delta-graphic matroids.
\subsection{Structure of matroids in $\mathcal{C}$}

We will describe the structure of all matroids in $\mathcal{C}$ from the cycle matroid of wheel graphs by gluing graphic or cographic matroids with $2$-sum operation. The \emph{wheel graph} of order $k+1$ is a graph $W_{k}$ on the vertex set $\{s,t_{1},\ldots,t_{k}\}$ with an edge set $\{e_{i}:1\leq i\leq 2k\}$ where, for $1\leq i\leq k$, $e_{2i-1}:=t_{i}t_{i+1}, e_{2i}:=st_{i+1}$ and $t_{k+1}:=t_{1}$. The vertex $s$ is a \emph{center} of $W_{k}$.
An edge of $W_{k}$ is a \emph{spoke} if it is incident with a center and is a \emph{rim edge} otherwise. See Figure~\ref{fig:W6}.

\begin{figure}[t]
\centering
\tikzstyle{v}=[circle, draw, solid, fill=black, inner sep=0pt, minimum width=3pt]
  \begin{tikzpicture}
    \draw(0,0) node[v,label=left:$s$](c){};
	\foreach \x in {1,2,6} {
    \draw (\x*60+30:1) node [v,label=$t_\x$](v\x){};
    \draw (c)--(v\x); 
    }
    \foreach \x in {3,4,5} {
      \draw (\x*60+30:1) node [v,label=below:$t_\x$](v\x){};
      \draw (c)--(v\x); 
      }
    \draw (v1)--(v2)--(v3)--(v4)--(v5)--(v6)--(v1);
  \end{tikzpicture}
\caption{The graph $W_{6}$.}
\label{fig:W6}
\end{figure}
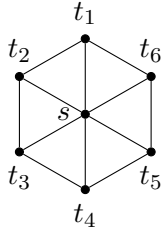

\begin{proposition}
\label{prop:main}
Let $k\geq 2$ be an integer and $(G,T)$ be a graft with a nice cyclic decomposition $(H,\mathcal{B})$ such that $H$ is an even cycle $1,2,3,\ldots ,2k,1$ and $\mathcal{B}=\{B_{i}:1\leq i\leq 2k\}$. 
Let $G_1,G_2,\ldots,G_{2k}$ be subgraphs of $G$ such that $V(G_i)=B_i$ for all $i\in\{1,2,\ldots,2k\}$
and every edge of $G$ is in exactly one of them.
Let $u_{i}$ be the vertex in $B_{i}\cap B_{i+1}$ where $B_{2k+1}:=B_{1}$.
Let $\tilde{G}_{i}$ be a graph obtained from $G_{i}$ by adding an edge $e_{i}=u_{i-1}u_{i}$ where $u_{0}:=u_{2k}$. 

Let $M^{H,\mathcal{B}}:=M(W_{k})\oplus_{2}M(\tilde{G}_{1})\oplus_{2}M^{*}(\tilde{G}_{2})\oplus_{2}M(\tilde{G}_{3})\oplus_{2}M^{*}(\tilde{G}_{4})\oplus_{2}\cdots\oplus_{2}M(\tilde{G}_{2k-1})\oplus_{2}M^{*}(\tilde{G}_{2k})$ where the edge $e_{i}\in E(W_{k})$ is identified with the edge $e_{i}$ in $E(\tilde{G}_{i})$ for $1\leq i\leq 2k$.
Then,
\[
\mathcal{F}(\mathcal{G}(G,T))=\mathcal{B}(M^{H,\mathcal{B}})\triangle\displaystyle\bigcup_{i=1}^{k}E(G_{2i}).
\]
\end{proposition}
Note that $G_i$ contains all non-loop edges of $G[B_i]$ and all loops not incident with vertices in $T$. It may be possible that a loop belongs to both $G[B_i]$ and $G[B_{i+1}]$ and that is why we define $G_i$ as above, instead of defining it as an induced subgraph of $G$.
Unfortunately this means that if $G$ has a loop $e$ incident with $u_i$ for some $i$, then depending on whether $e$ belongs to $G_i$ or $G_{i+1}$, the matroid $M^{H,\mathcal B}$ has $e$ as a loop or a coloop, meaning that $M^{H,\mathcal B}$ is not uniquely determined. However, later when we apply this construction, we will focus on connected matroids and so $G$ has no loops, implying that $M^{H,\mathcal B}$ is uniquely defined in such cases.

Since $(H,\mathcal{B})$ is nice, $G_{i}$ is connected for each $i$.
For each $i\in\{1,\ldots,2k\}$, a set $X\subseteq E({G}_{i})$ is \emph{attached} if $X$ is the set of edges of a spanning tree of $G_{i}$ and $X$ is \emph{detached} if $X$  is the set of edges of a spanning forest of $G_i$ with $2$ components, each having exactly one vertex in $\{u_{i-1},u_{i}\}$. 

\begin{lemma}
\label{o}
Let $k$, $(G,T)$, $G_{i}$, $\tilde{G}_{i}$ be given as in Proposition~\ref{prop:main}.
Let $1\leq i\leq 2k$ be an integer. For a subset $X$ of $E(G_{i})$, the following hold:
\begin{enumerate}
\item[$(1)$] $X$ is attached if and only if $X$ is a base of $M(\tilde{G}_{i})$.
\item[$(2)$] $X$ is detached if and only if $X\cup\{e_{i}\}$ is a base of $M(\tilde{G}_{i})$.
\end{enumerate}
\end{lemma}

\begin{proof}
The conclusion follows easily from the definitions of ``attached'' and ``detached''.
\end{proof}

\begin{lemma}
\label{B}
Let $k$, $(G,T)$, $G_{i}$, $\tilde{G}_{i}$, $M^{H,\mathcal{B}}$ be given as in Proposition~\ref{prop:main}.
A set $B$ is a base of $M^{H,\mathcal{B}}$ if and only if there exist sets $F_{1},F_{2},\ldots,F_{2k}$ such that
\begin{enumerate}[label=\rm(\arabic*)]
\item\label{item:4.3(1)} $F_{i}$ is a base of $M(\tilde{G}_{i})$ for each $1\leq i\leq 2k$,
\item\label{item:4.3(2)} $B=\bigcup_{i=1}^{2k} (F_{i}-\{e_{i}\})\triangle\bigcup_{i=1}^{k}E(G_{2i})$,
\item\label{item:4.3(3)} $\{e_{i} : i\in\{1,2,\ldots,2k\},~e_{i}\notin F_{i}\}\triangle\{e_{2},e_{4},\ldots,e_{2k}\}$ is a base of $M(W_{k})$.
\end{enumerate}
\end{lemma}

\begin{proof}
For $1\leq i\leq 2k$, let $\mathcal{B}_{i}:=\mathcal{B}(M(\tilde{G}_{i}))$ if $i$ is odd and $\mathcal{B}_{i}:=\mathcal{B}(M^{*}(\tilde{G}_{i}))$ otherwise. If $B$ is a base of $M^{H,\mathcal{B}}$, then there exist $X_{i}\in \mathcal{B}_{i}$ for each $i\in\{1,\ldots,2k\}$ and $X\in\mathcal{B}(M(W_{k}))$ such that $B=X\triangle X_{1}\triangle\cdots\triangle X_{2k}\triangle E(W_{k})$ and $X=\{e_{i}:e_{i}\notin X_{i}\}$. So $B=\bigcup_{i=1}^{2k}(X_{i}-\{e_{i}\})$. 
For each $1\leq i\leq 2k$, let
\[
F_{i}=
\begin{cases}
X_{i} & \text{if $i$ is odd,} \\
E(\tilde{G}_{i})- X_{i} & \text{if $i$ is even.}
\end{cases}
\] Then \ref{item:4.3(1)} holds obviously. Since $X_{2i}- \{e_{2i}\}=(F_{2i}- \{e_{2i}\})\triangle E(G_{2i})$ for all $1\leq i\leq k$, \ref{item:4.3(2)} holds. We know that \ref{item:4.3(3)} holds because $X=\{e_{i} : e_{i}\notin X_{i}\}=\{e_{i} : e_{i}\notin F_{i}\}\triangle\{e_{2},e_{4},\ldots,e_{2k}\}$. 

Conversely, suppose that there are $F_{1},F_{2},\ldots,F_{2k}$ satisfying \ref{item:4.3(1)}, \ref{item:4.3(2)}, and \ref{item:4.3(3)}. For each $1\leq i\leq 2k$, let 
\[
X_{i}=
\begin{cases}
F_{i} & \text{if $i$ is odd,} \\
E(\tilde{G}_{i})- F_{i} & \text{if $i$ is even.}
\end{cases}
\] 
Then by \ref{item:4.3(1)}, $X_{i}\in \mathcal{B}_{i}$ for each $i$. By \ref{item:4.3(3)}, we have $X:=\{e_{i} : e_{i}\notin F_{i}\}\triangle\{e_{2},e_{4},\ldots,e_{2k}\}=\{e_{i} : e_{i}\notin X_{i}\}\in \mathcal{B}(M(W_{k}))$. So $B=(X_{1}- \{e_{1}\})\cup\cdots\cup(X_{2k}- \{e_{2k}\})=X\triangle X_{1}\triangle\cdots\triangle X_{2k}\triangle E(W_{k})$ is a base of $M^{H,\mathcal{B}}$.
\end{proof}

\begin{lemma}
\label{evencycle1}
Let $(G,T)$ be a graft with a cyclic decomposition $(H,\mathcal{B})$ such that $H$ is connected. If $G'$ is a $2$-connected subgraph of $G$, 
\begin{enumerate}[label=\rm (\arabic*)]
\item\label{item:ec1} $V(G')\subseteq B_{v}$ for some $v\in V(H)$, or
\item\label{item:ec2} $H$ is an even cycle and $G'$ contains a path of $G[B_{v}]$ joining two vertices of $B_{v}\cap T$ for every $v\in V(H)$.
\end{enumerate}
\end{lemma}
\begin{proof}
Since deleting edges does not change a cyclic decomposition, we may assume that $E(G)=E(G')$.
We first claim that \ref{item:ec1} holds if $H$ is a path. Let $H=v_{1}v_{2}\cdots v_{\ell}$ be a path. 
Since $G'$ is $2$-connected, for each $1\leq k\leq \ell-1$, either $V(G')\subseteq\bigcup_{1\leq i\leq k}B_{v_{i}}$ or $V(G')\subseteq\bigcup_{k<i\leq \ell} B_{v_{i}}$. For $1\leq k\leq \ell-1$, we orient an edge $v_{k}v_{k+1}$ from $v_{k}$ to $v_{k+1}$ if $V(G')\subseteq\bigcup_{k<i\leq \ell} B_{v_{i}}$ and from $v_{k+1}$ to $v_{k}$ otherwise. Then there is a vertex $v_{i}$ of $H$ with no outgoing edges. This implies that $V(G')\subseteq B_{v_{i}}$ and \ref{item:ec1} holds.

Now suppose that $H$ is an even cycle and \ref{item:ec2} does not hold. Then there is a vertex $v$ of $H$ such that $G'$ does not contain a path of $G[B_{v}]$ joining two vertices of $B_{v}\cap T$. Then there is a partition $(X,Y)$ of $B_{v}$, each part having a vertex of $T$ such that there is no edge between $X$ and $Y$. Let $N_{H}(v)=\{x,y\}$ and $H'$ be a path obtained from $H\setminus v$ by adding vertices $x',y'$ such that $N_{H'}(x')=\{x\}$, $N_{H'}(y')=\{y\}$. Let $\mathcal{B}'=(\mathcal{B}-\{B_{v}\})\cup\{B_{x'},B_{y'}\}$ where $B_{x'}=X$ and $B_{y'}=Y$. Then $(H',\mathcal{B}')$ is a cyclic decomposition of $(G,T)$. Since $H'$ is a path and $B_{x'}, B_{y'}\subseteq B_{v}$, by the previous claim, \ref{item:ec1} holds.
\end{proof}

\begin{lemma}
\label{componentlemma}
Let $(G,T)$ be a graft with a cyclic decomposition $(H,\mathcal{B})$. Let $x\in V(H)$ and $u$,$v$ be vertices in $\bigcup_{y\neq x}B_{y}$. If $P$ is a path from $u$ to $v$ such that $E(P)\cap E(G[B_{x}])\neq\emptyset$, then $|B_{x}\cap T|=2$ and $P$ contains a path of $G[B_{x}]$ joining two vertices of $B_{x}\cap T$.
\end{lemma}

\begin{proof}
Let $P:=v_{1}\cdots v_{\ell}$ for $\ell\geq 2$ where $v_{1}:=u$, $v_{\ell}:=v$. Then $V(P)\cap B_{x}\neq\emptyset$ and we may let $i=\min\{k\in\{1,\ldots,\ell\} : v_{k}\in B_{x}\}$ and $j=\max\{k\in\{1,\ldots,\ell\} : v_{k}\in B_{x}\}$. Since $E(P)\cap E(G[B_{x}])\neq\emptyset$, we have $i\neq j$. We first claim that $v_{i}\in T$. It is trivial from \ref{item:c3} when $i=1$. So we may assume that $i>1$. Since $v_{i-1}\notin B_{x}$, by \ref{item:c2}, $v_{i-1},v_{i}\in B_{y}$ for some $y\in V(H)-\{x\}$. So $v_{i}\in B_{x}\cap B_{y}\subseteq T$ by \ref{item:c3}. Similarly, $v_{j}\in T$. By~\ref{item:c4}, $B_{x}\cap T=\{v_{i}, v_{j}\}$ and by~\ref{item:c3}, $Q=v_{i}Pv_{j}$ is a subpath of $P$ joining two vertices of $B_{x}\cap T$ in $G[B_{x}]$.
\end{proof}
\begin{lemma}
\label{connectedcomponent}
Let $(G,T)$ be a graft with a cyclic decomposition $(H,\mathcal{B})$ where $H$ is connected. For a subgraph $G'$ of $G$, let
\[
S_{G'}=\{v\in V(H): |B_{v}\cap T|=2 \text{~and $G'$ contains a path of $G[B_{v}]$ joining two vertices of $B_{v}\cap T$}\}
\]
and $T_{G'}=\{v\in V(H): V(G')\cap B_{v}\neq\emptyset\}$.
If $G'$ is connected, then
\begin{enumerate}[label=\rm (\arabic*)]
\item $H[S_{G'}]$ is connected, and
\item if $S_{G'}=\emptyset$, then $T_{G'}$ is a clique of size at most $2$ in $H$.
\end{enumerate}
\end{lemma}
\begin{proof}
Suppose that $H[S_{G'}]$ is disconnected. Let $x$, $y$ be vertices of $H$ which are contained in different components of $H[S_{G'}]$. Choose vertices $u$, $v$ of $G'$ and a path $P=w_{1}w_{2}\cdots w_{n}$ from $u=w_{1}$ to $v=w_{n}$ such that $u\in B_{x}$, $v\in B_{y}$ and $|V(P)|$ is minimum. By \ref{item:c3}, $B_{x}\cap B_{y}=\emptyset$ and $u\neq v$. For $1\leq i\leq n-1$, by \ref{item:c2}, there is a bag $B_{z_{i}}$ containing $w_{i}$, $w_{i+1}$ for some $z_{i}\in V(H)$. 
Since $|V(P)|$ is minimum, we have $w_{2}\notin B_{x}$ and $w_{n-1}\notin B_{y}$. So $z_{1}\neq x$, $z_{n-1}\neq y$, and $u\in B_{x}\cap B_{z_{1}}$, $v\in B_{y}\cap B_{z_{n-1}}$. By \ref{item:c3}, $x$ is adjacent to $z_{1}$ and $y$ is adjacent to $z_{n-1}$. For $1\leq i\leq n-1$, either $z_{i}=z_{i+1}$ or $z_{i}$, $z_{i+1}$ are adjacent in $H$ by \ref{item:c3}. By Lemma~\ref{componentlemma}, $z_{i}\in S_{G'}$ for all $1\leq i\leq n-1$. So there is a path from $x$ to $y$ in $H[S_{G'}]$, contradicting our assumption.

To show (2), suppose that $T_{G'}$ is not a clique of $H$. Then there are $x,y\in T_{G'}$ such that $x$, $y$ are not adjacent in $H$. Let $u\in V(G')\cap B_{x}$, $v\in V(G')\cap B_{y}$. Since $x$, $y$ are not adjacent in $H$, by \ref{item:c3}, $B_{x}\cap B_{y}=\emptyset$. Let $P=v_{1}v_{2}\cdots v_{n}$ be a path $G'$ from $u=v_{1}$ to $v=v_{n}$. There is an edge $v_{j}v_{j+1}$ which is not contained in $E(G[B_{x}])\cup E(G[B_{y}])$ because $B_{x}\cap B_{y}=\emptyset$. By \ref{item:c2}, there is a vertex $w\in V(H)-\{x,y\}$ such that $\{v_{j}, v_{j+1}\}\subseteq B_{w}$. By Lemma~\ref{componentlemma}, $w\in S_{G'}$, contradicting our assumption that $S_{G'}=\emptyset$. Since $H$ is bipartite, $T_{G'}$ is a clique of size at most $2$.
\end{proof}

\begin{lemma}
\label{F}
Let $k$, $(G,T)$, $G_{i}$, $\tilde{G}_{i}$ be given as in Proposition~\ref{prop:main}.
A set $F$ is a feasible set of $\mathcal{G}(G,T)$ if and only if there exist sets $F_{1},\ldots,F_{2k}$ such that
\begin{enumerate}[label=\rm(\arabic*)]
\item\label{item:4.7(1)} $F_{i}$ is a base of $M(\tilde{G}_{i})$ for each $1\leq i\leq 2k$,
\item\label{item:4.7(2)} $F=\bigcup_{i=1}^{k}(F_{i}- \{e_{i}\})$,
\item\label{item:4.7(3)} $\{e_{i} : i\in\{1,2,\ldots,2k\},~ e_{i}\notin F_{i}\}\triangle\{e_{2},e_{4},\ldots,e_{2k}\}$ is a base of $M(W_{k})$.
\end{enumerate}
\end{lemma}

\begin{proof}
Suppose that there are $F_{1},\ldots, F_{2k}$ satisfying \ref{item:4.7(1)}, \ref{item:4.7(2)}, and \ref{item:4.7(3)}. We show that $F=\bigcup_{i=1}^{2k}(F_{i}-\{e_{i}\})$ is a feasible set of $\mathcal{G}(G,T)$. For $i\in\{1,\ldots,2k\}$, let $F_{i}':=F_{i}-\{e_{i}\}$ and $H_{i}$ be a graph $(V(G_{i}),F_{i}')$. Let $D:=\{e_{i} : e_{i}\notin F_{i}\}\triangle\{e_{2},e_{4},\ldots,e_{2k}\}$ and $\Gamma=(V(G),F)$ be a spanning subgraph of $G$. We claim that $\Gamma$ is a $T$-spanning forest. 
By \ref{item:4.7(1)} and Lemma~\ref{o}, for each $i\in\{1,\ldots,2k\}$, $F_{i}'$ is attached or detached. For $i\in\{1,\ldots,2k\}$, if $F_{i}'$ is detached, then let $X_{i}$, $Y_{i}$ be edge sets of components of $H_{i}$ containing $u_{i-1}$, $u_{i}$ respectively. 

Suppose that $\Gamma$ contains a cycle $C$. If $F_{i}'$ is attached for all $i\in\{1,\ldots,2k\}$, then, by Lemma~\ref{o}, $e_{i}\notin F_{i}'=F_{i}$ for all $i\in\{1,\ldots,2k\}$ and $D=\{e_{1},e_{3},\ldots,e_{2k-1}\}$ is an edge set of a cycle of $W_{k}$, contradicting \ref{item:4.7(3)}. Therefore, by Lemma~\ref{evencycle1}, $V(C)\subseteq B_{i}$ for some $i\in\{1,\ldots,2k\}$. By \ref{item:4.7(1)}, $E(C)\nsubseteq F_{i}$. So there is an edge $e=xy\in E(C)$ such that $e\notin F_{i}$. So there is $j\neq i$ such that $e\in F_{j}$ and $x,y\in B_{i}\cap B_{j}$. By \ref{item:c3}, $H$ is an even cycle of length $2$, contradicting our assumption that $k\geq 2$. Therefore, $\Gamma$ is a forest.

It remains to show that $|V(G')\cap T|$ is odd for each component $G'$ of $\Gamma$. 
By Lemma~\ref{connectedcomponent}, $H[S_{G'}]$ is connected and if $S_{G'}=\emptyset$, then $T_{G'}$ is a clique of size at most $2$ in $H$. Suppose that $S_{G'}=\emptyset$. Since $F_{i}'$ is attached or detached for all $i\in\{1,\ldots,2k\}$, every component of $\Gamma$ contains a vertex of $T$. So $|T_{G'}|=2$ and $G'=X_{i+1}\cup Y_{i}$ for some $i\in\{1,\ldots,2k\}$. Therefore, we deduce that $|V(G')\cap T|=1$ if $S_{G'}=\emptyset$.
If $S_{G'}\neq\emptyset$, then $P:=H[S_{G'}]$ is a path because $\Gamma$ is a forest. By rotational symmetry, we can assume that $P$ is a path from $1$ or $2$. If $P$ is a path $1,2,\ldots,n$ for some $1\le n<2k$, then $F_{i}'$ is attached and $E(G')\cap E(H_{i})=F_{i}'$ for all $i\in\{1,\ldots,n\}$ and $F_{n+1}', F_{2k}'$ are detached and $E(G')\cap E(H_{2k})=Y_{2k}$,  $E(G')\cap E(H_{n+1})=X_{n+1}$. We show that $|V(G')\cap T|=n+1$ is odd. Suppose that $n$ is odd. Then, by Lemma~\ref{o}, $e_{i}\notin F_{i}$ for all $i\in\{1,\ldots,n\}$ and $e_{n+1}\in F_{n+1}$, $e_{2k}\in F_{2k}$. Therefore, $D$ contains $\{e_{1},e_{3},\ldots, e_{n},e_{n+1},e_{2k}\}$ which is an edge set of a cycle, contradicting \ref{item:4.7(3)}.

If $P$ is a path $2,3,\ldots,n$ for $2\leq n\leq 2k$, then $F_{i}'$ is attached and $E(G')\cap E(H_{i})=F_{i}'$ for all $i\in\{2,\ldots,n\}$ and both $F_{1}'$ and $F_{n+1}'$ are detached and $E(G')\cap E(H_{1})=Y_{1}$,  $E(G')\cap E(H_{n+1})=X_{n+1}$. We show that $|V(G')\cap T|=n$ is odd. Suppose that $n$ is even. Then, by Lemma~\ref{o}, $e_{i}\notin F_{i}$ for all $i\in\{2,\ldots,n\}$, $e_{1}\in F_{1}$, and $e_{n+1}\in F_{n+1}$. So $D\cap\{e_{1},e_{2},e_{4},\ldots, e_{n},e_{n+1}\}=\emptyset$. So $D$ is an edge set of a disconnected subgraph of $W_{k}$, contradicting \ref{item:4.7(3)}.

Conversely, suppose that $F$ is an edge set of a $T$-spanning forest $H$ of $G$. For $1\leq i\leq 2k$, let $F_{i}'=F\cap E(G_{i})$ and $H_{i}=(V(G_{i}),F_{i}')$. 
Since $G$ is connected and $H$ is a $T$-spanning forest, every component of $H$ contains a vertex in $T$. By \ref{item:c2}, there is no component of $H_{i}$ avoiding $T$. Since $|B_{i}\cap T|=2$, $F_{i}'$ should be either attached or detached. Let $F_{i}:=F_{i}'$ if $F_{i}'$ is attached and $F_{i}:=F_{i}'\cup\{e_{i}\}$ if $F_{i}'$ is detached. If $F_{i}'$ is detached, then let $X_{i}$,  $Y_{i}$ be edge sets of components of $H_{i}$ containing $u_{i-1}$, $u_{i}$ respectively. 

Trivially, \ref{item:4.7(2)} holds. By Lemma~\ref{o}, \ref{item:4.7(1)} holds. So it remains to prove \ref{item:4.7(3)}. Let $D:=\{e_{i} : e_{i}\notin F_{i}\}\triangle\{e_{2},e_{4},\ldots,e_{2k}\}$ and $Z=(V(W_{k}),D)$ be a subgraph of $W_{k}$. We wish to show that $Z$ is a spanning tree of $W_{k}$. Suppose that $Z$ contains a cycle.

If $Z$ has a cycle containing the center, then let $C$ be a shortest cycle among all cycles containing the center. Then $C$ is an induced cycle of $Z$. By rotational symmetry, we may assume that $E(C)=\{e_{1},e_{3},\ldots,e_{2n-1}\}\cup\{e_{2n},e_{2k}\}$ for some $1\leq n< k$. Since $E(C)\subseteq D$ and $D\cap\{e_{2},\ldots, e_{2(n-1)}\}=\emptyset$, by Lemma~\ref{o}, both $F_{2n}'$ and $F_{2k}'$ are detached and $F_{i}'$ is attached for $1\leq i\leq 2n-1$. So $Y_{2k}\cup F_{1}'\cup F_{2}'\cup\cdots\cup F_{2n-1}'\cup X_{2n}$ is an edge set of a component $X$ of $H$. However, $|V(X)\cap T|=2n\equiv 0 \pmod{2}$, contradicting our assumption. 

So $E(C)=\{e_{1},e_{3},\ldots,e_{2k-1}\}$ if $Z$ contains a cycle $C$. Then $F_{2j-1}'$ is attached for $1\leq j\leq k$. Since $Z$ has no cycles containing the center, $|D\cap \{e_{2},e_{4},\ldots,e_{2k}\}|\leq 1$. So all but at most one of $F_{1}',\ldots, F_{2k}'$ is attached and $H$ has a component containing all vertices of $T$, contradicting our assumption. Therefore, $Z$ is a forest. 

It remains to prove that $Z$ is connected. If $Z$ is disconnected, then there exists a component $P$ not containing the center $s$ of $W_{k}$. Then $P$ is a connected subgraph of a cycle $W_{k}\backslash s$. Since $Z$ is a forest and $P$ is connected, $P$ is a path. By rotational symmetry, let $E(P)=\{e_{1},e_{3},\ldots, e_{2n-1}\}$ for some $n<k$. By Lemma~\ref{o}, both $F_{2n+1}'$ and $F_{2k-1}'$ are detached and $F_{i}'$ is attached for all $i\in\{1,\ldots,2n\}\cup\{2k\}$. Therefore, $Y_{2k-1}\cup F_{2k}'\cup F_{1}'\cup\cdots\cup F_{2n}'\cup X_{2n+1}$ is a component $Y$ of $H$ and $|V(Y)\cap T|=2n+2\equiv 0 \pmod{2}$, contradicting our assumption. Therefore, $Z$ is connected.
\end{proof}

\begin{proof}[Proof of Proposition~\ref{prop:main}]
Obvious from Lemma~\ref{B} and Lemma~\ref{F}.
\end{proof}

\subsection{Structure of matroids in $\mathcal{P}$}
For a positive integer $k\geq 2$, let

\[
\Pi_{k}=
\begin{cases}
W_{\lfloor\frac{k}{2}\rfloor+1}/e_{k+1} & \text{if $k$ is odd,} \\
W_{\lfloor\frac{k}{2}\rfloor+1}\setminus e_{k+1}/e_{k+2} & \text{if $k$ is even.} \\
\end{cases}
\]

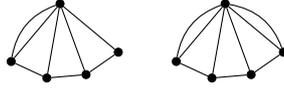
\begin{figure}[t]
\centering
\tikzstyle{v}=[circle, draw, solid, fill=black, inner sep=0pt, minimum width=3pt]
  \begin{tikzpicture}
    \draw(0,0) node[v](c){};
	\foreach \x in {0,1,2,3} {
    \draw (230+\x*30:1) node [v](v\x){};
    \draw (c)--(v\x); 
    }
    \draw (v0)--(v1)--(v2)--(v3);
    \draw (c) [bend right] to (v0);
  \end{tikzpicture}
  $\quad$
  \begin{tikzpicture}
    \draw(0,0) node[v](c){};
	\foreach \x in {0,1,2,3} {
    \draw (230+\x*30:1) node [v](v\x){};
    \draw (c)--(v\x); 
    }
    \draw (v0)--(v1)--(v2)--(v3);
    \draw (c) [bend right] to (v0);
    \draw (c)[bend left] to (v3);
  \end{tikzpicture}
\caption{Two graphs $\Pi_{8}$ and $\Pi_{9}$.}
\label{fig:Pi89}
\end{figure}

\begin{proposition}
\label{prop:main2}
Let $\ell\geq 2$ be an integer and $(G,T)$ be a graft with a nice cyclic decomposition $(H,\mathcal{B})$ such that $H$ is a path $1,2,\ldots,\ell$ and $\mathcal{B}=\{B_{i}:1\leq i\leq \ell\}$. 
Let $G_1,G_2,\ldots,G_\ell$ be subgraphs of $G$ such that $V(G_i)=B_i$ for all $i\in\{1,2,\ldots,\ell\}$ and every edge of $G$ is in exactly one of them.
For $i\in\{1,\ldots,\ell-1\}$, let $u_{i}$ be the vertex in $B_{i}\cap B_{i+1}$ and $u_{0}$ be the vertex in $(T\cap B_{1})- B_{2}$ and $u_{\ell}$ be the vertex in $(T\cap B_{\ell})- B_{\ell-1}$. Let $\tilde{G}_{i}$ be a graph obtained from $G_{i}$ be adding an edge $e_{i}=u_{i-1}u_{i}$.
Let  
\[
M^{H,\mathcal{B}}=
\begin{cases}
M(\Pi_{\ell})\oplus_{2}M(\tilde{G}_{1})\oplus_{2}M^{*}(\tilde{G}_{2})\oplus_{2}M(\tilde{G}_{3})\oplus_{2}M^{*}(\tilde{G}_{4})\oplus_{2}\cdots\oplus_{2}M(\tilde{G}_{\ell}) & \text{if $\ell$ is odd,} \\
M(\Pi_{\ell})\oplus_{2}M(\tilde{G}_{1})\oplus_{2}M^{*}(\tilde{G}_{2})\oplus_{2}M(\tilde{G}_{3})\oplus_{2}M^{*}(\tilde{G}_{4})\oplus_{2}\cdots\oplus_{2}M^{*}(\tilde{G}_{\ell}) & \text{if $\ell$ is even}
\end{cases}
\]
where $e_{i}\in E(\Pi_{\ell})$ is identified with $e_{i}$ in $E(\tilde{G}_{i})$ for $1\leq i\leq \ell$.
Then, 
\[
\mathcal{F}(\mathcal{G}(G,T))=\mathcal{B}(M^{H,\mathcal{B}})\triangle\displaystyle\bigcup_{i=1}^{\lfloor\frac{\ell}{2}\rfloor} E(G_{2i})
\]
\end{proposition}

\begin{proof}
Suppose that $\ell$ is odd. Let $\ell=2n-1$ for some $n\geq 2$. Let $\tilde{G}=G+f_{\ell+1}$ where $f_{\ell+1}$ is a new edge joining $u_{0}$ and $u_{\ell}$. 
Let $G_{\ell+1}$ be a graph with $V(G_{\ell+1})=\{u_{0},u_{\ell}\}$ and $E(G_{\ell+1})=\{f_{\ell+1}\}$. Then $(\tilde{G},T)$ has a nice cyclic decomposition $(\tilde{H},\tilde{\mathcal{B}})$ such that $\tilde{H}$ is an even cycle obtaining from $H$ by adding a vertex $\ell+1$ adjacent to $1$ and $\ell$
and $\tilde{\mathcal{B}}=\mathcal{B}\cup\{B_{\ell+1}\}$ where $B_{\ell+1}:=\{u_{0},u_{\ell}\}$. Let $\tilde{G}_{\ell+1}$ be a graph on $\{u_{0},u_{\ell}\}$ with $2$ parallel edges $e_{\ell+1}$, $f_{\ell+1}$. 

Since $f_{\ell+1}$ is not a $T$-bridge of $(\tilde{G},T)$, $\mathcal{G}(G,T)=\mathcal{G}(\tilde{G}\setminus f_{\ell+1},T)=\mathcal{G}(\tilde{G},T)\setminus f_{\ell+1}$. We observe that $(M(W_{n})\oplus_{2}M^{*}(\tilde{G}_{\ell+1}))/f_{\ell+1}=M(W_{n}/e_{\ell+1})$. Hence, by Proposition~\ref{prop:main}, we have 
\[
\mathcal{F}(\mathcal{G}(G,T))=\left(\mathcal{B}(M^{\tilde{H},\tilde{B}})\triangle\bigcup_{i=1}^{n} E(G_{2i})\right)
\setminus f_{\ell+1}=\mathcal{B}(M^{\tilde{H},\tilde{B}}/f_{\ell+1})\triangle\bigcup_{i=1}^{n-1} E(G_{2i}). 
\]
Note that 
\begin{align*}
M^{\tilde{H},\tilde{\mathcal{B}}}/f_{\ell+1}&=(M(W_{n})\oplus_{2}M(\tilde{G}_{1})\oplus_{2}\cdots\oplus_{2}M(\tilde{G}_{\ell})\oplus_{2}M^{*}(\tilde{G}_{\ell+1}))/f_{\ell+1} \\
&=M(W_{n}/e_{\ell+1})\oplus_{2}M(\tilde{G}_{1})\oplus_{2}\cdots\oplus_{2}M(\tilde{G}_{\ell}) \\
&=M(\Pi_{\ell})\oplus_{2}M(\tilde{G}_{1})\oplus_{2}\cdots\oplus_{2}M(\tilde{G}_{\ell})=M^{H,\mathcal{B}}.
\end{align*}
Suppose that $\ell$ is even. Then $\ell=2n$ for $n\geq 1$. Let $\tilde{G}$ be a graph obtained from $G$ by adding a vertex $u_{\ell+1}$ adjacent to $u_{\ell}$. 
Let $f_{\ell+1}=u_{\ell}u_{\ell+1}$, $\tilde{T}=T\cup\{u_{\ell+1}\}$, and $G_{\ell+1}$ be a graph with $V(G_{\ell+1})=\{u_{\ell},u_{\ell+1}\}$ and $E(G_{\ell++1})=\{f_{\ell+1}\}$. Then $(\tilde{G},\tilde{T})$ has a nice cyclic decomposition $(\tilde{H},\tilde{\mathcal{B}})$ such that $\tilde{H}$ is a path obtained from $H$ by adding a vertex $v_{\ell+1}$ adjacent to $v_{\ell}$ and $\tilde{\mathcal{B}}=\mathcal{B}\cup\{B_{\ell+1}\}$ where $B_{\ell+1}:=\{u_{\ell},u_{\ell+1}\}$. Let $\tilde{G}_{\ell+1}$ is a graph on $\{u_{\ell},u_{\ell+1}\}$ with $2$ parallel edges $e_{\ell+1}$, $f_{\ell+1}$. 

Since $f_{\ell+1}$ is not a $\tilde{T}$-bridge of $(\tilde{G},\tilde{T})$, $\mathcal{F}(\mathcal{G}(G,T))=\mathcal{F}(\mathcal{G}(\tilde{G}\setminus f_{\ell+1},\tilde{T}))=\mathcal{F}(\mathcal{G}(\tilde{G},\tilde{T}))\setminus f_{\ell+1}$. We observe that $(M(W_{n+1}/e_{\ell+2})\oplus_{2}M(\tilde{G}_{\ell+1}))\setminus f_{\ell+1}=M(W_{n+1}\setminus e_{\ell+1}/e_{\ell+2})$. Hence, by the first case, we have
\[
\mathcal{F}(\mathcal{G}(G,T))=\left(\mathcal{B}(M^{\tilde{H},\tilde{B}})\triangle\bigcup_{i=1}^{n} E(G_{2i})\right)\setminus f_{\ell+1}=\mathcal{B}(M^{\tilde{H},\tilde{B}}\setminus f_{\ell+1})\triangle\bigcup_{i=1}^{n} E(G_{2i})
\] where
\begin{align*}
M^{\tilde{H},\tilde{\mathcal{B}}}\setminus f_{\ell+1}&=(M(W_{n+1}/e_{\ell+2})\oplus_{2}M(\tilde{G}_{1})\oplus_{2}\cdots\oplus_{2}M^{*}(\tilde{G}_{\ell})\oplus_{2}M(\tilde{G}_{\ell+1}))\setminus f_{\ell+1} \\
&=M(W_{n+1}\setminus e_{\ell+1}/e_{\ell+2})\oplus_{2}M(\tilde{G}_{1})\oplus_{2}\cdots\oplus_{2}M^{*}(\tilde{G}_{\ell}) \\
&=M(\Pi_{\ell})\oplus_{2}M(\tilde{G}_{1})\oplus_{2}\cdots\oplus_{2}M^{*}(\tilde{G}_{\ell})=M^{H,\mathcal{B}}.
\end{align*}
So we finish the proof.
\end{proof}

For a graft $(G,T)$, a \emph{$T$-separation} is a pair $(G_{1},G_{2})$ of subgraphs of $G$ such that the following hold:
\begin{itemize}
\item $G_{1}\cup G_{2}=G$, $E(G_{1})\cap E(G_{2})=\emptyset$.
\item $\min\{|E(G_{1})|, |E(G_{2})|\}\geq 1$.
\item $T=V(G_{1})\cap V(G_{2})$ and $|T|=2$.
\end{itemize} A vertex $v$ of $G$ is a \emph{$T$-cutvertex} if $\kappa(G,T)<\kappa(G\setminus v, T-\{v\})$. 

The following proposition characterizes grafts giving connected graphic delta-matroids.
\begin{proposition}
\label{prop:connectedgraft}
Let $G$ be a graph with $|E(G)|\geq 2$ and no isolated vertices. Let $T$ be a subset of $V(G)$. Then $\mathcal{G}(G,T)$ is connected if and only if $G$ is connected and has no loops, no $T$-separations, and no $T$-cutvertices.
\end{proposition}

\begin{proof}
We first prove the forward direction. Suppose that $M:=\mathcal{G}(G,T)=(E(G),\mathcal{F})$ is connected. By Lemma~\ref{graftconnected}, $G$ is connected. If $G$ has a loop $e$, then there is no feasible set containing $e$ and therefore $(\{e\},E(G)-\{e\})$ is a separation of $M$, contradicting our assumption. So $G$ has no loops. 

Suppose that $G$ has a $T$-separation for $T=\{u,v\}$. Then, by Lemma~\ref{Tlessthan2}, $\mathcal{G}(G,T)=M(G')$ where $G'=(G+e)/e$ and $e=uv$. Since $G'$ has a cut-vertex or a loop, $\mathcal{G}(G,T)=M(G')$ is disconnected, contradicting our assumption. So $G$ has no $T$-separations.

If $G$ has a $T$-cutvertex $v$, then there is a partition $(X,Y)$ of $V(G)-\{v\}$ such that $X$, $Y\neq\emptyset$, $Y\cap T=\emptyset$, and there is no edge between $X$, $Y$. Let $G_{X}:=G[X\cup\{v\}]$, $G_{Y}:=G[Y\cup\{v\}]$, $T_{X}=T\cap(X\cup\{v\})$, $T_{Y}=T\cap(Y\cup\{v\})$. 
Let $\mathcal F_X$ be the set of all feasible sets of $\mathcal G(G_X,T_X)$
and $\mathcal F_Y$ be the set of all feasible sets of $\mathcal G(G_Y,T_Y)$.
We show that $(E(G_{X}),E(G_{Y}))$ is a separation of $\mathcal{G}(G,T)$.

For each $F\in\mathcal{F}$, let $H=(V(G),F)$. Then $H[Y\cup\{v\}]$ is a spanning tree of $G_{Y}$ since $G$ is connected and $|T_{Y}|\leq 1$. Obviously, $H[X\cup\{v\}]$ is a $T_{X}$-spanning forest of $G_{X}$. So $\mathcal{F}\subseteq\{F_{X}\cup F_{Y} : F_{X}\in \mathcal F_X, F_{Y}\in\mathcal{F}_Y\}$. It is easy to see that $\{F_{X}\cup F_{Y} : F_{X}\in \mathcal F_X, F_{Y}\in\mathcal{F}_Y\}\subseteq\mathcal{F}$. Therefore, $(E(G_{X}),E(G_{Y}))$ is a separation of $\mathcal{G}(G,T)$, contradicting our assumption. So $G$ has no $T$-cutvertices.

For the backward direction, suppose that $G$ is connected and has no loops, no $T$-separations, and no $T$-cutvertices. If $|T|\leq 1$, then $\mathcal{G}(G,T)=M(G)$ is connected because $G$ has no cut-vertices. If $T=\{u,v\}$, then by Lemma~\ref{Tlessthan2}, $\mathcal{G}(G,T)=M(G')$ where $G'=(G+e)/e$ and $e=uv$. Since $G$ has no $T$-separation, there is no edge in $G$ parallel to $e$. Let $e^{*}$ be a vertex of $G'$ obtained from $G$ by contracting $e$. 
Since $G$ has neither a loop nor an edge parallel to $e$, $G'$ has no loop.
Suppose that $G'$ has a cut-vertex $w$. Then $w=e^{*}$, because otherwise $w$ is a $T$-cutvertex of $G$. Let $(X,Y)$ be a partition of $V(G')-\{w\}$ such that $X,Y\neq\emptyset$ and there is no edge between $X$ and $Y$. Let $G_{X}:=G[X\cup T]$ and $G_{Y}=G[Y\cup T]$. Then $(G_{X},G_{Y})$ is a $T$-separation of $G$, contradicting our assumption. So $\mathcal{G}(G,T)=M(G')$ is connected. So we may assume that $|T|\geq 3$.

Let $M:=\mathcal{G}(G,T)=(E(G),\mathcal{F})$. By Lemma~\ref{bouchetconnected}, it is enough to prove that, whenever $e,f\in E(G)$ have a common end $v$, they are contained in same component of $G_{M,F}$ for some $F\in\mathcal{F}$. 

If there is a cycle $C$ containing both $e$ and $f$ in $G$, then let $P:=C\setminus f$ and $H$ be a spanning tree of $G$ containing $P$. Let $F_{1}=E(H)$. If $|T|$ is odd, then $F_{1}, F_{1}\triangle\{e,f\}\in\mathcal{F}$ and so $e$ and $f$ are adjacent in $G_{M,F_{1}}$. Therefore, we may assume that $|T|$ is even. 
If $T\cap(V(H)- V(P))\neq\emptyset$, then we can choose an edge $g\in E(H)- E(P)$ such that $H'=H\setminus g$ induces two components $T_{1}$, $T_{2}$ so that $|V(T_{1})\cap T|=1$ and $V(T_{1})\cap V(C)=\emptyset$. Let $F_{2}=E(H\setminus g)$. Then $F_{2}, F_{2}\triangle\{e,f\}\in\mathcal{F}$ and $e$ and $f$ are adjacent in $G_{M,F_{2}}$. 

So we may assume that $T\subseteq V(P)$. Let $P=x_{1}x_{2}\cdots x_{n}$ where $e=x_{n-1}x_{n}$. Let $i=\min\{k: x_{k}\in T\}$ and $g'=x_{i}x_{i+1}$. Since $|T|\geq 3$, $g'\neq e$. Let $F_{3}:=E(H\setminus g')$. Then we know that $F_{3}\in\mathcal{F}$. If $x_{n}\notin T$, then we can observe that $F_{3}\triangle\{e,f\}\in\mathcal{F}$. So $e$ and $f$ are in same component of $G_{M,F_{3}}$. Therefore, we assume that $x_{n}\in T$. Let $j=\max\{k:k<n, x_{k}\in T\}$ and $g''=x_{j-1}x_{j}$. Since $|T|\geq 3$ and $|T|$ is even, $|T|\geq 4$ and so $g''\neq g'$. Then we know that $F_{3}\triangle\{e,g''\}, F_{3}\triangle\{f,g''\}\in\mathcal{F}$. So $e$ and $f$ are in the same component of $G_{M,F_{3}}$.

Hence, we can assume that there is no cycle containing both $e$ and $f$. Then $v$ is a cut-vertex of~$G$. Let $C_{1},\ldots, C_{m}$ be all components of $G\setminus v$. Then, for each $i\in\{1,\ldots,m\}$, $V(C_{i})\cap T\neq\emptyset$ because otherwise $v$ is a $T$-cutvertex. Let $e=vx$ and $f=vy$. 
Without loss of generality, we may assume that $x\in V(C_{1})$ and $y\in V(C_{2})$. For each $i\in\{1,\ldots,m\}$, let $T_{i}:=V(C_{i})\cap T$ and $H_{i}$ be a $T_{i}$-spanning forest of $G_{i}$. Let $F_{i}:=E(H_{i})$ for $i\in\{1,\ldots,m\}$ and $F:=F_{1}\cup F_{2}\cup\cdots\cup F_{m}$. If $v\in T$, then $F\in\mathcal{F}$ and $F\triangle\{e,f\}=F\cup\{e,f\}\in\mathcal{F}$, implying that $e$ and $f$ are adjacent in $G_{M,F}$. If $v\notin T$, then $F'=F\cup\{e\}\in\mathcal{F}$ and $F'\triangle\{e,f\}=F\cup\{f\}\in\mathcal{F}$ and therefore $e$ and $f$ are adjacent in $G_{M,F'}$. 
\end{proof}

Now we are ready to characterize the structure of connected delta-graphic matroids.

\begin{proposition}
\label{prop:structure_conn}
A connected matroid $M$ is delta-graphic if and only if 
\begin{enumerate}[label=\rm (\arabic*)]
\item\label{item:str1} $M$ or $M^{*}$ is graphic, or
\item\label{item:str2} $M$ or $M^{*}$ is isomorphic to $M^{H,\mathcal{B}}$ for a graft $(G,T)$ such that $(G,T)$ admits a nice cyclic decomposition $(H,\mathcal{B})$ where $H$ is an even cycle of length at least $4$ or a path of length at least~$1$.
\end{enumerate}
\end{proposition}

\begin{proof}
The backward direction follows from Propositions~\ref{prop:tgraphicgraft}, \ref{prop:main}, and \ref{prop:main2}.
For the forward direction, suppose that \ref{item:str1} does not hold. Then by Proposition~\ref{prop:tgraphicgraft}, we may assume that $M\in\mathcal{C}$ or $M\in\mathcal{P}$. It is enough to show that \ref{item:str2} holds.

If $M\in\mathcal{C}$, then $\mathcal{B}(M)=\mathcal{F}(\mathcal{G}(G,T))\triangle X$ for a graft $(G,T)$ and $X\subseteq E(G)$ where $(G,T)$ has a nice cyclic decomposition $(H,\mathcal{B})$ with an even cycle $H=1,2,\ldots,2k,1$ and $k\geq 2$.
By Proposition~\ref{prop:main}, $\mathcal{F}(\mathcal{G}(G,T))=\mathcal{B}(M^{H,\mathcal{B}})\triangle\bigcup_{i=1}^{k}E(G[B_{2i}])$.
So $\mathcal{B}(M)=\mathcal{B}(M^{H,\mathcal{B}})\triangle X'$ for a set $X'=X\triangle\bigcup_{i=1}^{k}E(G[B_{2i}])$. Since $M$ is connected, $M$ is isomorphic to $M^{H,\mathcal{B}}$ or $(M^{H,\mathcal{B}})^{*}$ by Lemma~\ref{tmatroid}.

Now suppose that $M\in\mathcal{P}$. Then $\mathcal{B}(M)=\mathcal{F}(\mathcal{G}(G,T))\triangle X$ for a graft $(G,T)$ and $X\subseteq E(G)$ where $(G,T)$ has a nice cyclic decomposition $(H,\mathcal{B})$ with a path $H=1,2,\ldots,\ell$ and $\ell\geq 2$. By Proposition~\ref{prop:main2}, $\mathcal{F}(\mathcal{G}(G,T))=\mathcal{B}(M^{H,\mathcal{B}})\triangle\bigcup_{i=1}^{\lfloor\frac{\ell}{2}\rfloor}E(G[B_{2i}])$.
So $\mathcal{B}(M)=\mathcal{B}(M^{H,\mathcal{B}})\triangle X'$ for a set $X'=X\triangle\bigcup_{i=1}^{\lfloor\frac{\ell}{2}\rfloor}E(G[B_{2i}])$. Since $M$ is connected, $M$ is isomorphic to $M^{H,\mathcal{B}}$ or $(M^{H,\mathcal{B}})^{*}$ by Lemma~\ref{tmatroid}. 
\end{proof}

\section{Tree decomposition of delta-graphic matroids}
\label{sec:td}
In this section, we describe the structure of delta-graphic matroids which are not $3$-connected. For that, we will use the tree decomposition of matroids in Section~\ref{sec:prelim}. 

\begin{lemma}
\label{pathcontraction}
Let $M$ be a connected matroid with a canonical tree decomposition $(T,\rho)$ and $P$ be a path of $T$ from $u$ to $v$ with length at least $2$ such that every internal vertex of $P$ has degree $2$ in $T$. Let $e_{u}$, $e_{v}$ be edges of $P$ incident with $u$, $v$, respectively and $N$ be a matroid obtained from $\rho(v)$ by relabelling an element $e_{v}$ to $e_{u}$. If $(T',\rho')$ is a matroid-labelled tree  such that $T'$ is obtained from $T\setminus(V(P)-\{u,v\})$ by adding an edge $e_{u}$ joining $u$ and $v$ and
\[
\rho'(x)=
\begin{cases}
\rho(x) & \text{if $x\in V(T')-\{v\}$,} \\
N & \text{if $x=v$,}
\end{cases}
\] 
then $\rho'(T')$ is a minor of $M$. In addition, the canonical tree decomposition of $\rho'(T')$ is $(T',\rho')/e_{u}$ if the pair $\{u, v\}$ is bad  and $(T',\rho')$ otherwise.
\end{lemma}  

\begin{proof}
Let $T_{u}$, $T_{v}$ be components of $T\setminus(V(P)-\{u,v\})$ containing $u$, $v$, respectively. Let $M_{1}=\rho(T_{u})$, $M_{2}=\rho(P\setminus\{u,v\})$, and $M_{3}=\rho(T_{v})$. Then $M=M_{1}\oplus_{2}M_{2}\oplus_{2}M_{3}$.
Since $M_{2}$ is connected, there is a circuit $C$ in $M_{2}$ containing $e_{u}$ and $e_{v}$. Then $M_{2}'=M_{2}\setminus (E(M_{2})- C)/(C-\{e_{u},e_{v}\})$ is a minor of $M_{2}$. By Lemma~\ref{2summinor}, $M_{2}'\oplus_{2}M_{3}$ is a minor of $M_{2}\oplus_{2}M_{3}$ and so $M'=M_{1}\oplus_{2}M_{2}'\oplus_{2}M_{3}$ is a minor of $M_{1}\oplus_{2}M_{2}\oplus_{2}M_{3}$.

Let $M_{3}'$ be the matroid obtained from $M_{3}$ by relabelling an element $e_{v}$ to $e_{u}$. Then $M_{2}'\oplus_{2}M_{3}=M_{3}'$. So $M'=\rho'(T')$ is a minor of $M$. 

It is easy to check that $(T',\rho')$ is a canonical tree decomposition of $\rho'(T')$ if the pair $\{u, v\}$ is not bad. If the pair $\{u, v\}$ is bad, then $(T',\rho')/e_{u}$ is a canonical tree decomposition of $\rho'(T')$ because  
$U_{1,m}\oplus_{2}U_{1,n}=U_{1,m+n-2}$ and $U_{m-1,m}\oplus_{2}U_{n-1,n}=U_{m+n-3,m+n-2}$  for $m,n\geq 2$.
\end{proof}

Let $\mathcal{G}$ be the class of graphic matroids and $\mathcal{G}^{*}$ be the class of cographic matroids.
If $T$ is a tree, $e\in E(T)$, and $v\in V(T)$, then we define $T_{e}(v)$ to be the component of $T\setminus e$ not containing $v$. 

\begin{lemma}
\label{twosum}
Let $M$ be a connected matroid. Then the following are equivalent:
\begin{enumerate}[label=\rm(\arabic*)]
\item There exists a graft $(G,T)$ with a nice cyclic decomposition $(H,\mathcal{B})$ such that $M$ or $M^{*}$ is isomorphic to $M^{H,\mathcal{B}}$ and $H$ is isomorphic to $P_{2}$, $P_{3}$, or $C_{4}$.
\item $M$ is a $2$-sum of a graphic matroid and a cographic matroid.
\end{enumerate}
\end{lemma}

\begin{proof}
We will first prove that (1) implies (2). By taking dual, we may assume that $M$ is isomorphic to $M^{H,\mathcal{B}}$. Let us start with the case that $H$ is a path. Let $H=1,\ldots,\ell+1$ be a path of length $\ell$. For $i\in\{1,\ldots,\ell+1\}$, let $G_{i}$ be a graph obtained from $G[B_{i}]$ by adding an edge $e_{i}$ joining two vertices of $B_{i}\cap T$ and 
\[
M_{i}=
\begin{cases}
M(G_{i}) & \text{if $i$ is odd,} \\
M^{*}(G_{i}) & \text{if $i$ is even.}
\end{cases}
\]

If $\ell=1$, then we have $M^{H,\mathcal{B}}=M(\Pi_{2})\oplus_{2}M_{1}\oplus_{2}M_{2}=(M(\Pi_{2})\oplus_{2}M_{1})\oplus_{2}M_{2}$. Since $M(\Pi_{2})=U_{1,2}$, $M(\Pi_{2})\oplus_{2}M_{1}$ is graphic and $M^{H,\mathcal{B}}$ is a $2$-sum of a graphic matroid and a cographic matroid.

If $\ell=2$, then $M^{H,\mathcal{B}}=M(\Pi_{3})\oplus_{2}M_{1}\oplus_{2}M_{2}\oplus_{2}M_{3}=(M(\Pi_{3})\oplus_{2}M_{1}\oplus_{2}M_{3})\oplus_{2}M_{2}$. Since $M(\Pi_{3})=U_{1,3}$, $M(\Pi_{3})\oplus_{2}M_{1}\oplus_{2}M_{3}$ is graphic and $M^{H,\mathcal{B}}$ is a $2$-sum of a graphic matroid and a cographic matroid.

Now it remains to check the case that $H=1,2,3,4,1$ is a cycle of length $4$. For $i\in\{1,\ldots,4\}$, let $G_{i}$ be a graph obtained from $G[B_{i}]$ by adding an edge $e_{i}$ joining two vertices of $B_{i}\cap T$ and 
\[
M_{i}=
\begin{cases}
M(G_{i}) & \text{if $i$ is odd,} \\
M^{*}(G_{i}) & \text{if $i$ is even.}
\end{cases}
\]

Then $M^{H,\mathcal{B}}=M(W_{2})\oplus_{2}M_{1}\oplus_{2}M_{2}\oplus_{2}M_{3}\oplus_{2}M_{4}$. 
It is easy to observe that $M(W_{2})=N_{1}\oplus_{2}N_{2}$ where $N_{1}$ is isomorphic to $U_{1,3}$ with $\{e_{1}, e_{3}\}\subseteq E(N_{1})$ and $N_{2}$ is isomorphic to $U_{2,3}$ with $\{e_{2}, e_{4}\}\subseteq E(N_{2})$. So $M^{H,\mathcal{B}}=(N_{1}\oplus_{2}M_{1}\oplus_{2}M_{3})\oplus_{2}(N_{2}\oplus_{2}M_{2}\oplus_{2}M_{4})$ and $N_{1}\oplus_{2}M_{1}\oplus_{2}M_{3}$ is graphic and $N_{2}\oplus_{2}M_{2}\oplus_{2}M_{4}$ is cographic. So $M^{H,\mathcal{B}}$ is a $2$-sum of a graphic matroid and a cographic matroid.

Now we show (2) implies (1). Suppose that $M=M(G_{1})\oplus_{2}M^{*}(G_{2})$ for graphs $G_{1}$ and $G_{2}$. We may assume that $G_{1}$ and $G_{2}$ have no isolated vertices. We may assume that $E(G_{1})\cap E(G_{2})=\{e\}$. Since $M$ is connected, $M(G_{1})$ and $M^{*}(G_{2})$ are connected by Lemma~\ref{2sumconnected}. Then, for $i\in\{1,2\}$, $G_{i}$ is $2$-connected or is isomorphic to a graph with $2$ vertices and parallel edges. Hence, $G_{1}\setminus e$ and $G_{2}\setminus e$ are connected. Let $x_{1}$, $y_{1}$ be ends of $e$ in $G_{1}$ and $x_{2}$, $y_{2}$ be ends of $e$ in $G_{2}$. Let $G$ be a graph obtained from the union of $G_{1}\setminus e$ and $G_{2}\setminus e$ by identifying $x_{1}$ with $x_{2}$ and $T:=\{x_{1},y_{1},y_{2}\}$. Let $H$ be a path of length $1$ on the vertex set $\{1,2\}$. Let $B_{1}:=V(G_{1})$, $B_{2}:=V(G_{2})$, and $\mathcal{B}=\{B_{1},B_{2}\}$. Because $G_{1}\setminus e$ and $G_{2}\setminus e$ are connected, $(H,\mathcal{B})$ is a nice cyclic decomposition of $(G,T)$ and $M=M^{H,\mathcal{B}}$.
\end{proof}

\begin{lemma}
\label{Lemma6.9}
Let $M$ be a connected matroid. Then $M=M_{1}\oplus_{2}M_{2}$ such that $M_{1}\in\mathcal{G}$, $M_{2}\in\mathcal{G}^{*}$ if and only if $M\in\mathcal{G}\cup\mathcal{G}^{*}$ or $M$ has a canonical tree decomposition $(T,\rho)$ which satisfies one of the following:
\begin{enumerate}[label=\rm(\arabic*)]
\item There is an edge $e$ of $T$ such that for two components $T_{1}$, $T_{2}$ of $T\setminus e$, $\rho(T_{1})\in\mathcal{G}$ and $\rho(T_{2})\in\mathcal{G}^{*}$.
\item There is a vertex $v$ of $T$ such that 
\begin{enumerate}[label=\rm (\roman*)]
\item\label{item:gg*1} $\rho(v)$ is isomorphic to $U_{1,m}$ or $U_{m-1,m}$ for some $m\geq 3$, 
\item\label{item:gg*2} for each component $C$ of $T\setminus v$, $\rho(C)\in\mathcal{G}\cup\mathcal{G}^{*}$, and
\item\label{item:gg*3} there are distinct components $C_{1}$, $C_{2}$ of $T\setminus v$ such that $\rho(C_{1})\in\mathcal{G}$ and $\rho(C_{2})\in\mathcal{G}^{*}$. 
\end{enumerate} 
\end{enumerate}
\end{lemma}

\begin{proof}
To prove the forward direction, suppose that $M=M_{1}\oplus_{2}M_{2}\notin\mathcal{G}\cup\mathcal{G}^{*}$ for $M_{1}\in\mathcal{G}$ and $M_{2}\in\mathcal{G}^{*}$. Then $M_{1}\notin\mathcal{G}^{*}$, $M_{2}\notin\mathcal{G}$, and $|E(M_{1})|,|E(M_{2})|\geq 3$. Let $E(M_{1})\cap E(M_{2})=\{e\}$. Since $M$ is connected, both $M_{1}$ and $M_{2}$ are connected by Lemma~\ref{2sumconnected}. Hence, by Theorem~\ref{thm:canonical}, $M_{1}$ and $M_{2}$ have canonical tree decompositions. For $i\in\{1,2\}$, let $(T_{i},\rho_{i})$ be a canonical tree decomposition of $M_{i}$ and $x_{i}$ be a vertex of $T_{i}$ such that $e\in E(\rho(x_{i}))$. Then $\rho(T_{1})=M_{1}\in\mathcal{G}$ and $\rho(T_{2})=M_{2}\in\mathcal{G}^{*}$. Let $(T',\rho')$ be a matroid-labelled tree such that $T'$ is obtained from $T_{1}\cup T_{2}$ by adding an edge $e$ joining $x_{1}$ and $x_{2}$ and
\[
\rho'(x)=
\begin{cases}
\rho_{1}(x) & \text{if $x\in V(T_{1})$,} \\
\rho_{2}(x) & \text{if $x\in V(T_{2})$.}
\end{cases}
\]
If a pair $\{x_{1}, x_{2}\}$ is not bad, then $(T,\rho):=(T',\rho')$ is a canonical tree decomposition of $M$ and $(T,\rho)$ satisfies (1).
If $\{x_{1}, x_{2}\}$ is bad, then $(T,\rho):=(T',\rho')/e$ is a canonical tree decomposition. 

Now we show that $(T,\rho)$ satisfies (2). Let $v$ be a vertex of $T$ obtained by contracting $e$. Then $\rho(v)$ is isomorphic to $U_{1,m}$ or $U_{m-1,m}$ for $m\geq 3$ and so \ref{item:gg*1} holds. For each component $C$ of $T\setminus v$, since $C$ is a subtree of $T_{1}$ or $T_{2}$, $\rho(C)$ is isomorphic to a minor of $\rho(T_{1})$ or $\rho(T_{2})$ by Lemma~\ref{subtreeminor}. So $\rho(C)\in\mathcal{G}\cup\mathcal{G}^{*}$ and \ref{item:gg*2} holds. Moreover, for each $i\in\{1,2\}$, $T_{i}$ is not a single vertex $x_{i}$ otherwise $M\in\mathcal{G}\cup\mathcal{G}^{*}$. Therefore, for each $i\in\{1,2\}$, there is a component $C_{i}$ of $T\setminus v$ which is a subtree of $T_{i}\setminus x_{i}$. So \ref{item:gg*3} holds because $\rho(C_{1})\in\mathcal{G}$ and $\rho(C_{2})\in\mathcal{G}^{*}$.

Now we prove the converse. If $M\in\mathcal{G}\cup\mathcal{G}^{*}$, then $M$ is isomorphic to $M\oplus_{2}U_{1,2}$ and $U_{1,2}\in\mathcal{G}\cap\mathcal{G}^{*}$. If (1) holds, then $M=\rho(T)=\rho(T_{1})\oplus_{2}\rho(T_{2})$ and so $M$ is a $2$-sum of a graphic matroid and a cographic matroid. 

Suppose that (2) holds. By taking dual, we may assume that $\rho(v)$ is a uniform matroid with rank~$1$. By \ref{item:gg*2}, \ref{item:gg*3} of (2), there exists a partition $(X,Y)$ of $E(\rho(v))$ such that $|X|,|Y|\geq 1$, $\rho(T_{e}(v))\in\mathcal{G}$ for each $e\in X\cap E(T)$, and $\rho(T_{f}(v))\in\mathcal{G}^{*}$ for each $f\in Y\cap E(T)$. Let $X\cap E(T)=\{e_{1},e_{2},\ldots,e_{m}\}$ and let $Y\cap E(T)=\{f_{1},f_{2},\ldots,f_{n}\}$. Let $N_{1}$ be a uniform matroid of rank $1$ on $X\cup\{g\}$ and $N_{2}$ be a uniform matroid of rank $1$ on $Y\cup\{g\}$. Then $\rho(v)=N_{1}\oplus_{2}N_{2}$ and so $M=(N_{1}\oplus_{2}\rho(T_{e_{1}}(v))\oplus_{2}\cdots\oplus_{2}\rho(T_{e_{m}}(v)))\oplus_{2}(N_{2}\oplus_{2}\rho(T_{f_{1}}(v))\oplus_{2}\cdots\oplus_{2}\rho(T_{f_{n}}(v)))$, $N_{1}\oplus_{2}\rho(T_{e_{1}}(v))\oplus_{2}\cdots\oplus_{2}\rho(T_{e_{m}}(v))$ is graphic, and $N_{2}\oplus_{2}\rho(T_{f_{1}}(v))\oplus_{2}\cdots\oplus_{2}\rho(T_{f_{n}}(v))$ is cographic.
\end{proof}

For a matroid $M$, a subset $C$ of $E(M)$ is a \emph{circuit-hyperplane} of $M$ if $C$ is a circuit and a hyperplane of $M$.

A canonical tree decomposition $(T,\rho)$ of a connected matroid $M$ is a \emph{wheel decomposition} if there is a vertex $v$ of $T$, called a \emph{hub}, such that the following hold:
\begin{enumerate}[label=\rm(W\arabic*)]
\item\label{item:w1} $\rho(v)$ is isomorphic to $M(W_{k})$ for some $k\geq 3$.
\item\label{item:w2} There exists a circuit-hyperplane $C$ of $\rho(v)$ such that for every component $X$ of $T\setminus v$ and the edge $e_{X}$ joining $v$ and a vertex of $X$, 
\begin{enumerate}
\item $\rho(X)\in\mathcal{G}$ if $e_{X}\in C$, 
\item $\rho(X)\in\mathcal{G}^{*}$ if $e_{X}\notin C$.
\end{enumerate}
\end{enumerate}

We remark that $C$ is the set of rim edges.

\begin{lemma}
\label{Wtdual}
If a connected matroid $M$ admits a wheel decomposition, then so does $M^{*}$.
\end{lemma}

\begin{proof}
The conclusion follows from the fact that the complement of a circuit-hyperplane of $M(W_{k})$ is a circuit-hyperplane of $M^{*}(W_{k})$ which is isomorphic to $M(W_{k})$.
\end{proof}

\begin{lemma}
\label{Wtreedecomposition}
A connected matroid $M$ has a wheel decomposition if and only if there is a graft $(G,T)$ with a nice cyclic decomposition $(H,\mathcal{B})$ such that $M$ or $M^{*}$ is isomorphic to $M^{H,\mathcal{B}}$ and $H$ is isomorphic to $C_{2k}$ for some $k\geq 3$.
\end{lemma}

\begin{proof}
To prove the backward direction, suppose that $H=1,2,\ldots,2k,1$ is an even cycle for some $k\geq 3$. 
As $M$ is connected, $G$ has no loops.
For each $1\leq i\leq 2k$, let $G_{i}$ be a graph obtained from $G[B_{i}]$ by adding an edge $e_{i}$ joining two vertices of $T\cap B_{i}$ and let
\[
M_{i}=
\begin{cases}
M(G_{i}) & \text{if $i$ is odd,} \\
M^{*}(G_{i}) & \text{if $i$ is even.}
\end{cases}
\]

By Lemma \ref{Wtdual}, we can assume that $M=M^{H,\mathcal{B}}=M(W_{k})\oplus_{2}M_{1}\oplus_{2}M_{2}\oplus_{2}\cdots\oplus_{2}M_{2k}$ by taking dual. Let $E(M(W_{k}))=\{e_{1},e_{2},\ldots,e_{2k}\}$ and $C:=\{e_{1},e_{3},\ldots,e_{2k-1}\}$ be the set of rim edges of $M(W_{k})$. Since $M$ is connected, $M_{i}$ is connected for each $i\in\{1,2,\ldots,2k\}$ by Lemma~\ref{2sumconnected}. Let $A=\{1\leq i\leq 2k:|E(M_{i})|\geq 3\}$. For each $i\in A$, let $(T_{i},\rho_{i})$ be a canonical tree decomposition of $M_{i}$ and let $v_{i}$ be a vertex of $V(T_{i})$ such that $e_{i}\in \rho_{i}(v_{i})$. Let $(T,\rho)$ be a matroid-labelled tree such that $T$ is obtained from $\bigcup_{i\in A}T_{i}$ by adding a new vertex $u$ and edges $e_{i}=uv_{i}$ for all $i\in A$ and 

\[
\rho(v)=\begin{cases}
M(W_{k}) & \mbox{if $v=u$,} \\
\rho_{i}(v) & \mbox{if $v\in V(T_{i})$ for some $i\in A$.}
\end{cases}
\]
Since $\rho(u)=M(W_{k})$ is $3$-connected, $(T,\rho)$ is a canonical tree decomposition of $\rho(T)$ which is isomorphic to $M^{H,\mathcal{B}}$. We claim that $(T,\rho)$ is a wheel decomposition with a hub $u$. Obviously, \ref{item:w1} holds. Observe that $C$ is a circuit-hyperplane of $M(W_{k})$. For each component $X$ of $T\setminus u$, $X=T_{i}$ for some $i\in A$. So $\rho(X)=\rho(T_{i})$ for each $i$. If $e_{i}\in C$, then $i$ is odd and $\rho(T_{i})=M_{i}\in\mathcal{G}$. If $e_{i}\notin C$, then $i$ is even and $\rho(T_{i})=M_{i}\in\mathcal{G}^{*}$. So \ref{item:w2} holds.

To prove the forward direction, let $(T,\rho)$ be a wheel decomposition of $M$ with a hub $v$. Then $\rho(v)=M(W_{k})$ for some $k\geq 3$. Let $E(M(W_{k}))=\{e_{1},e_{2},\ldots,e_{2k}\}$ and $C:=\{e_{1},e_{3},\ldots,e_{2k-1}\}$ be the circuit-hyperplane of $M(W_{k})$.
Let $A=\{1\leq i\leq 2k : e_{i}\in E(T)\}$. For each $i\in\{1,2,\ldots,2k\}$, let $M_{i}$ be a matroid $\rho(T_{e_{i}}(v))$ if $i\in A$ and $U_{1,2}$ containing $e_{i}$ if $i\notin A$. By \ref{item:w2}, there is a connected graph $G_{i}$ such that 
\[
M(G_{i})=
\begin{cases}
M_{i} & \text{if $i$ is odd,} \\
M_{i}^{*} & \text{if $i$ is even.}
\end{cases}
\]

Then $M$ is isomorphic to $M(W_{k})\oplus_{2}M_{1}\oplus_{2}M_{2}\oplus_{2}\cdots\oplus_{2} M_{2k}$. For each $i\in\{1,2,\ldots,2k\}$, since $M_{i}$ is connected, $G_{i}$ is $2$-connected or isomorphic to a graph with $2$ vertices and parallel edges and therefore $G_{i}\setminus e_{i}$ is connected. For each $i\in\{1,2,\ldots,2k\}$, let $x_{i}$, $y_{i}$ be ends of $e_{i}$ in $G_{i}$. Let $G$ be a graph obtained from the graph $(G_{1}\setminus e_{1})\cup(G_{2}\setminus e_{2})\cup\cdots\cup(G_{2k}\setminus e_{2k})$ by identifying $y_{i}$, $x_{i+1}$ for each $1\leq i\leq 2k-1$, and identifying $x_{1}$, $y_{2k}$ and let $T=\{x_{1},x_{2},\ldots,x_{2k}\}$. 
Let $H=1,2,\ldots,2k,1$ be an even cycle of length $2k$ and let $B_{i}=V(G_{i})$ for each $i\in\{1,2,\ldots,2k\}$. Let $\mathcal{B}=\{B_{i}:1\leq i\leq 2k\}$. Then $(H,\mathcal{B})$ is a nice cyclic decomposition of $(G,T)$ and $M$ is isomorphic to $M^{H,\mathcal{B}}$.
\end{proof}

A canonical tree decomposition $(T,\rho)$ of a connected matroid $M$ is a \emph{fan decomposition} if $T$ has a path $P$ of length at least $1$, called a \emph{spine}, such that the following hold:
\begin{enumerate}[label=\rm(F\arabic*)]
\item\label{item:f1} For each $v\in V(P)$, $\rho(v)$ is a uniform matroid of rank $1$ or corank $1$.

\item\label{item:f2} For an edge $e\in E(T)-E(P)$ incident with $v\in V(P)$,
\begin{enumerate}
\item $\rho(T_{e}(v))\in\mathcal{G}$ if $v$ is an internal node of $P$ and $\rho(v)$ has corank $1$,
\item $\rho(T_{e}(v))\in\mathcal{G}^{*}$ if $v$ is an internal node of $P$ and $\rho(v)$ has rank $1$, and
\item $\rho(T_{e}(v))\in\mathcal{G}\cup\mathcal{G}^{*}$ otherwise.
\end{enumerate}

\item\label{item:f3} For an end vertex $v$ of $P$, if $\deg_{T}(v)=|E(\rho(v))|$, then there exist distinct edges $e_{1}, e_{2}$ incident with $v$ such that $e_{1}, e_{2}\notin E(P)$, $\rho(T_{e_{1}}(v))\in\mathcal{G}$, and $\rho(T_{e_{2}}(v))\in\mathcal{G}^{*}$.
\end{enumerate}

\begin{lemma}
\label{ftdual} 
If a connected matroid $M$ admits a fan decomposition, then so does $M^{*}$. 
\end{lemma}

\begin{proof}
The conclusion follows easily from Lemma~\ref{dualtree}.
\end{proof}

\begin{lemma}
\label{Ftreedecomposition}
A connected matroid $M$ has a fan decomposition if and only if $M$ or $M^{*}$ is isomorphic to $M^{H,\mathcal{B}}$ for a graft $(G,T)$ admitting a nice cyclic decomposition $(H,\mathcal{B})$ such that $H$ is a path of length at least $3$.
\end{lemma}

\begin{proof}
Let us prove the backward direction first. Let $H=1,2,\ldots,\ell$ be a path for $\ell\geq 4$. 
As $M$ is connected, $G$ has no loops.
For each $i\in\{1,2,\ldots,\ell\}$, let $G_{i}$ be a graph obtained from $G[B_{i}]$ by adding an edge $e_{i}$ joining two vertices of $B_{i}\cap T$ and let 
\[
M_{i}=
\begin{cases}
M(G_{i}) & \text{if $i$ is odd,} \\
M^{*}(G_{i}) & \text{if $i$ is even.}
\end{cases}
\]
Then, by Lemma~\ref{ftdual}, we can assume that $M=M^{H,\mathcal{B}}=M(\Pi_{\ell})\oplus_{2}M_{1}\oplus_{2}M_{2}\oplus_{2}\cdots\oplus_{2}M_{\ell}$ by taking dual. Let $T'$ be a path graph on the vertex set $\{u_{1},u_{2},\ldots,u_{\ell-2}\}$ and $f_{i}=u_{i}u_{i+1}$ for $1\leq i\leq\ell-3$. For $i\in\{1,\ldots,\ell-2\}$, let
\[
E_{i}=
\begin{cases}
\{e_{1},e_{2},f_{1}\} & \text{if $i=1$,} \\
\{f_{i-1},f_{i},e_{i+1}\} & \text{if $1<i<\ell-2$,} \\
\{f_{\ell-3},e_{\ell-1},e_{\ell}\} & \text{if $i=\ell-2$.}
\end{cases}
\]
and let $\rho'(u_{i})$ be the uniform matroid on $E_{i}$ of rank $1$ if $i$ is odd and of rank $2$ if $i$ is even.

Then $(T',\rho')$ is a canonical tree decomposition of $M(\Pi_{\ell})$.
By Lemma~\ref{2sumconnected}, $M_{i}$ is connected for each $i\in\{1,2,\ldots,\ell\}$ because $M$ is connected. For $i\in\{1,2,\ldots,\ell\}$, let $(T_{i},\rho_{i})$ be a canonical tree decomposition of $M_{i}$ and $v_{i}$ be a vertex of $T_{i}$ such that $e_{i}\in E(\rho_{i}(v_{i}))$. Let $(T'',\rho'')$ be a matroid-labelled tree such that $T''$ is obtained from $(\bigcup_{i=1}^{\ell}T_{i})\cup T'$ by adding edges 
\[
e_{i}=
\begin{cases}
u_{1}v_{1} & \text{if $i=1$,} \\
u_{i-1}v_{i} & \text{if $2\leq i\leq\ell-1$, and}\\
u_{\ell-2}v_{\ell} & \text{if $i=\ell$.}
\end{cases}
\]
and

\[
\rho''(v)=
\begin{cases}
\rho'(v) & \text{if $v\in V(T')$,} \\
\rho_{i}(v) & \text{if $v\in V(T_{i})$ for some $i\in\{1,2,\ldots,\ell\}$.}
\end{cases}
\]
Let $A=\{e_{i}\in E(T''): \text{a pair of ends of $e_{i}$ is bad in $(T'',\rho'')$}\}$. Then $(T,\rho):=(T'',\rho'')/A$ is a canonical tree decomposition of $M$. We claim that $(T,\rho)$ is a fan decomposition with a spine $T'=u_{1}u_{2}\cdots u_{\ell-2}$. For each $i\in\{1,\ldots,\ell-2\}$, $\rho(u_{i})$ is a uniform matroid of rank $1$ or corank $1$. So \ref{item:f1} holds. 

For an edge $e\in E(T)-E(T')$ incident with $u_{i}$,
if $e_i\notin A$, then  $T_{e}(u_{i})=T_j$ for some $j$
and 
if $e_i\in A$, then $T_e(u_i)$ is a component of $T_{j}\setminus v_{j}$ for some $j$. 
By Lemma~\ref{subtreeminor}, $\rho(T_{e}(u_{i}))$ is isomorphic to a minor of $\rho(T_{j})\in\mathcal{G}\cup\mathcal{G}^{*}$, which implies that $\rho(T_{e}(u_{i}))\in\mathcal{G}\cup\mathcal{G}^{*}$. 

Moreover, if $u_{i}$ is an internal node and $\rho(u_{i})$ is a uniform matroid of rank $1$, then $i$ is odd and $\rho(T_{e}(u_{i}))$ is a minor of $\rho(T_{i+1})\in\mathcal{G}^{*}$. Therefore, $\rho(T_{e}(u_{i}))\in\mathcal{G}^{*}$.
If $u_{i}$ is an internal node and $\rho(u_{i})$ is a uniform matroid of corank $1$, then $i$ is even and $\rho(T_{e}(u_{i}))$ is a minor of $\rho(T_{i+1})$. Since $\rho(T_{i+1})\in\mathcal{G}$, $\rho(T_{e}(u_{i}))\in\mathcal{G}$. Hence, \ref{item:f2} holds. 

So it remains to prove that \ref{item:f3} holds. By symmetry, it is enough to show for $u_{1}$. Suppose that $\deg_{T}(u_{1})=|E(\rho(u_{1}))|$. For $i\in\{1,2\}$, let $g_{i}=e_{i}$ if the pair $\{u_{1}, v_{i}\}$ is not bad in $T''$. If the pair $\{u_{1}, v_{i}\}$ is bad, then let $g_{i}$ be an edge in $T_{i}$ incident with $v_{i}$. Such an edge exists because otherwise $\deg_{T}(u_{1})<|E(\rho(u_{1}))|$. Both $g_{1}$ and $g_{2}$ are incident with $u_{1}$ in $T$.

For each $i\in\{1,2\}$, $T_{g_{i}}(u_{1})$ is a subtree of $T_{i}$, and by Lemma~\ref{subtreeminor}, $\rho(T_{g_{i}}(u_{1}))$ is a minor of $\rho(T_{i})$. Therefore, $\rho(T_{g_{1}}(u_{1}))\in\mathcal{G}$ and $\rho(T_{g_{2}}(u_{1}))\in\mathcal{G}^{*}$. So \ref{item:f3} holds.

To prove the forward direction, let $(T,\rho)$ be a fan decomposition of $M$. Then there is a path $P=u_{1}u_{2}\cdots u_{\ell}$ for $\ell\geq 2$ such that \ref{item:f1}, \ref{item:f2}, and \ref{item:f3} holds. Let $f_{i}:=u_{i}u_{i+1}$ for $i\in\{1,2,\ldots,\ell-1\}$. By taking dual, for each $i\in\{1,2,\ldots,\ell\}$, we may assume that $\rho(u_{i})$ is a rank-$1$ uniform matroid if $i$ is odd and a corank-$1$ uniform matroid if $i$ is even.

For $j\in\{1,\ell\}$, 
let $X_j'$ be the set of all elements $e\in E(\rho(u_j))-E(P)$ such that
either $e\notin E(T)$ or $\rho(T_e(u_j))\in \mathcal G$
and let $Y_j'$ be the set of all elements $e\in E(\rho(u_j))-E(P)$ such that
either $e\notin E(T)$ or $\rho(T_e(u_j))\in \mathcal G^*$.
By \ref{item:f2}, $X_j'\cup Y_j'=E(\rho(u_j))-E(P)$
and by \ref{item:f3}, both $X_j'$ and $Y_j'$ are nonempty.
Since $\abs{X_j'\cup Y_j'}\ge 2$, 
we can choose $X_j\subseteq X_j'$ and $Y_j\subseteq Y_j'$ such that 
$X_j\cap Y_j=\emptyset$, $X_j\cup Y_j=E(\rho(u_j))-E(P)$, and $\abs{X_j}, \abs{Y_j}\ge 1$.
Then  $\rho(T_{e}(u_{j}))\in\mathcal{G}$ for each $e\in X_{j}\cap E(T)$ and $\rho(T_{f}(u_{j}))\in\mathcal{G}^{*}$ for each $f\in Y_{j}\cap E(T)$. 

For $i\in\{1,\ldots,\ell+2\}$, let
\[
\Gamma_{i}=
\begin{cases}
X_{1} & \text{if $i=1$,} \\
Y_{1} & \text{if $i=2$,} \\
E(\rho(u_{i-1}))-\{f_{i-2},f_{i-1}\} & \text{if $2<i<\ell+1$,} \\
X_{\ell} & \text{if $i\geq\ell+1$ and $i$ is odd,} \\
Y_{\ell} & \text{if $i\geq\ell+1$ and $i$ is even.} 
\end{cases}
\]
and let $\Gamma_{i}'=\Gamma_{i}\cap E(T)$.
Let $e_{1}^{i},\ldots,e_{m_{i}}^{i}$ be the list of all elements of $\Gamma_i'$ with $m_i=\abs{\Gamma_i'}\ge0$.
For $i\in\{1,2,\ldots,\ell+2\}$, let $N_{i}$ be a uniform matroid on $\Gamma_{i}\cup\{e_{i}\}$ such that $N_{i}$ has corank $1$ if $\min\{i,\ell+1\}\in\{k:2\leq k\leq\ell+1 \text{~and $k$ is odd}\}$ and rank $1$ otherwise.
Let $M_{i}=N_{i}\oplus_{2}\rho(T_{e_{1}^{i}}(u_{i}))\oplus_{2}\cdots\oplus_{2}\rho(T_{e_{m_{i}}^{i}}(u_{i}))$. By \ref{item:f2} and definitions of $X_{1}$, $Y_{1}$, $X_{\ell}$, and $Y_{\ell}$, we have $M_{i}\in\mathcal{G}$ if $i$ is odd and $M_{i}\in\mathcal{G}^{*}$ if $i$ is even. For each $i\in\{1,2,\ldots,\ell+2\}$, there is a connected graph $G_{i}$ such that
\[
M(G_{i})=
\begin{cases}
M_{i} & \text{if $i$ is odd,} \\
M_{i}^{*} & \text{if $i$ is even.}
\end{cases}
\]

For $i\in\{1,\ldots,\ell\}$, let $N_{i}'$ be a uniform matroid such that 
\[
E(N_{i}')=
\begin{cases}
\{f_{1},e_{1},e_{2}\} & \text{if $i=1$,} \\
\{f_{i-1},e_{i+1},f_{i}\} & \text{if $2\leq i\leq\ell-1$, and} \\
\{f_{\ell-1},e_{\ell+1},e_{\ell+2}\} & \text{if $i=\ell$.}
\end{cases}
\]
and $N_{i}'$ has rank $1$ if $i$ is odd and corank $1$ otherwise. It is straightforward to check that $N_{1}'\oplus_{2}\cdots\oplus_{2}N_{\ell}'=M(\Pi_{\ell+2})$. Moreover, for $i\in\{1,\ldots,\ell\}$,
\[
\rho(u_{i})=
\begin{cases}
N_{1}'\oplus_{2}N_{1}\oplus_{2}N_{2} & \text{if $i=1$,} \\
N_{i}'\oplus_{2}N_{i+1} & \text{if $1<i<\ell$,} \\
N_{\ell}'\oplus_{2}N_{\ell+1}\oplus_{2}N_{\ell+2} & \text{if $i=\ell$}
\end{cases}
\]
and therefore,
\begin{align*}
M&=\rho(T)=\OPLUS_{i=1}^{\ell}\left(\rho(u_{i})\oplus_{2}\OPLUS_{j=1}^{m_{i}}\rho(T_{e_{j}^{i}}(u_{i}))\right) \\
&=(N_{1}'\oplus_{2}\cdots\oplus_{2}N_{\ell}')\oplus_{2}(N_{1}\oplus_{2}\cdots\oplus_{2}N_{\ell+2})\oplus_{2}\OPLUS_{i=1}^{\ell}\left(\OPLUS_{j=1}^{m_{i}}\rho(T_{e_{j}^{i}}(u_{i}))\right) \\
&=(N_{1}'\oplus_{2}\cdots\oplus_{2}N_{\ell}')\oplus_{2}M_{1}\oplus_{2}\cdots\oplus_{2}M_{\ell+2}=M(\Pi_{\ell+2})\oplus_{2}M_{1}\oplus_{2}\cdots\oplus_{2}M_{\ell+2}.
\end{align*}

Since $M_{i}$ is connected, $G_{i}$ is loopless $2$-connected or isomorphic to a graph with $2$ vertices and parallel edges and therefore $G_{i}\setminus e_{i}$ is connected. For each $i\in\{1,2,\ldots,\ell+2\}$, let $x_{i}$, $y_{i}$ be ends of $e_{i}$ in $G_{i}$. Let $G$ be a graph obtained from the graph $(G_{1}\setminus e_{1})\cup(G_{2}\setminus e_{2})\cup\cdots\cup(G_{\ell+2}\setminus e_{\ell+2})$ by identifying $y_{i}$ and $x_{i+1}$ for each $i\in\{2,\ldots,\ell+1\}$ and $T:=\{x_{1},x_{2},\ldots,x_{\ell+2}\}$. Let $H=1,2,\ldots,\ell+2$ be a path of length $\ell+1\geq 3$ and let $B_{i}=V(G_{i})$ for each $i\in\{1,2,\ldots,\ell+2\}$. Let $\mathcal{B}=\{B_{i}:1\leq i\leq \ell+2\}$. Then $(H,\mathcal{B})$ is a nice cyclic decomposition of $(G,T)$ and $M=M^{H,\mathcal{B}}$.
\end{proof}

\begin{proposition}
\label{prop:deltagraphictree}
 A connected matroid $M\notin\mathcal{G}\cup\mathcal{G}^{*}$ is delta-graphic if and only if its canonical tree decomposition $(T,\rho)$ satisfies at least one of the following:
\begin{enumerate}[label=\rm(T\arabic*)]
\item\label{item:t1} There is an edge $e$ of $T$ such that for two components $T_{1}$, $T_{2}$ of $T\setminus e$, $\rho(T_{1})\in\mathcal{G}$ and $\rho(T_{2})\in\mathcal{G}^{*}$.
\item\label{item:t2} There is a vertex $v$ of $T$ such that 
\begin{enumerate}[label=\rm (\roman*)]
\item $\rho(v)$ is isomorphic to $U_{1,m}$ or $U_{m-1,m}$ for some $m\geq 3$, 
\item for each component $C$ of $T\setminus v$, $\rho(C)\in\mathcal{G}\cup\mathcal{G}^{*}$, and
\item there are distinct components $C_{1}$, $C_{2}$ of $T\setminus v$ such that $\rho(C_{1})\in\mathcal{G}$ and $\rho(C_{2})\in\mathcal{G}^{*}$. 
\end{enumerate} 
\item\label{item:t3} $(T,\rho)$ is a wheel decomposition.
\item\label{item:t4} $(T,\rho)$ is a fan decomposition
\end{enumerate} 
\end{proposition}
\begin{proof}
The conclusion follows from Proposition~\ref{prop:structure_conn} and Lemmas~\ref{twosum}, \ref{Lemma6.9}, \ref{Wtreedecomposition}, and \ref{Ftreedecomposition}. 
\end{proof}

Now we are ready to prove Theorem~\ref{thm:main}.
\begin{proof}[Proof of Theorem~\ref{thm:main}]
The forward direction follows from Proposition~\ref{prop:deltagraphictree}. Let us prove the backward direction. Let $M$ be a connected generalized wheel obtained from the cycle matroid $N$ of a minor of a wheel graph. By relabelling edges, we may assume that $N=M(W_{k})\setminus X/Y$ for $k\geq 3$ and some $X,Y\subseteq E(W_{k})$. Let $M'$ be a matroid such that $M'\setminus X/Y=M$. Then $M'$ is obtained from $M(W_{k})$ by a sequence of $2$-sums where 
the other part of each $2$-sum is connected and graphic if the corresponding basepoint is a rim edge
and is connected and cographic otherwise.
So $M'$ is connected and let $(T,\rho)$ be a canonical tree decomposition of $M'$. Then $(T,\rho)$ is a wheel decomposition and therefore $M'$ is delta-graphic by Proposition~\ref{prop:deltagraphictree}. Since $M$ is a minor of $M'$, $M$ is delta-graphic by Lemma~\ref{tgraphicminorclosed}.
\end{proof}

From Theorem~\ref{thm:main} and the facts that every graphic or cographic matroid is regular~{\cite[Proposition 5.1.2]{Oxley2011}} and the class of regular matroids is closed under taking $2$-sum~{\cite[Corollary 7.1.26]{Oxley2011}}, we can observe that every delta-graphic matroid is regular.

\section{Forbidden minors for the class of delta-graphic matroids}   
\label{sec:exclude}

In this section, we will show Theorem~\ref{thm:exclude}, stating every forbidden minor for the class of delta-graphic matroids has at most $48$ elements.
\begin{lemma}
\label{lem:3c}
A $3$-connected matroid $M$ is delta-graphic if and only if $M$ is graphic or cographic. 
\end{lemma}

\begin{proof}
Obviously, it is enough to prove the forward direction. Suppose that $M\notin\mathcal{G}\cup\mathcal{G}^{*}$. By Proposition~\ref{prop:deltagraphictree}, the canonical tree decomposition $(T,\rho)$ of $M$ satisfies one of \ref{item:t1}, \ref{item:t2}, \ref{item:t3}, and \ref{item:t4}. Since $M$ is 3-connected, $|V(T)|=1$ and $(T,\rho)$ should satisfy \ref{item:t3}. Then $M$ should be isomorphic to $W_{m}$ for $m\geq 3$, contradicting our assumption that $M\notin\mathcal{G}\cup\mathcal{G}^{*}$.
\end{proof}

Since $R_{10}$ is $3$-connected and neither graphic nor cographic, by Lemma~\ref{lem:3c}, $R_{10}$ is not delta-graphic but is regular~\cite{Oxley2011}.

\begin{lemma}
\label{3connected}
A $3$-connected matroid $M$ is delta-graphic if and only if it does not have a minor isomorphic to one of $U_{2,4}$, $F_{7}$, $F_{7}^{*}$, $R_{10}$, $R_{12}$.
\end{lemma}

\begin{proof}
Since $U_{2,4}$, $F_{7}$, $F_{7}^{*}$, $R_{10}$, $R_{12}$ are $3$-connected and neither graphic nor cographic, they are not delta-graphic by Lemma~\ref{lem:3c}
and by Lemma~\ref{tgraphicminorclosed}, a matroid containing a minor isomorphic to $U_{2,4}$, $F_{7}$, $F_{7}^{*}$, $R_{10}$, or $R_{12}$  is not delta-graphic.

Suppose that $M$ is a $3$-connected matroid which is not delta-graphic. Then $M$ is neither graphic nor cographic by Lemma~\ref{lem:3c}. If $M$ is not regular, then $M$ has a minor isomorphic to $U_{2,4}$, $F_{7}$, or $F_{7}^{*}$ by Theorem~\ref{thm:tutte}. If $M$ is regular, then $M$ has a minor isomorphic to $R_{10}$ or $R_{12}$ by Lemma \ref{regular}. 
\end{proof}

\begin{lemma}
\label{GorCG}
Let $M$ be a connected matroid with a canonical tree decomposition $(T,\rho)$. If $M$ is not delta-graphic, then $\rho(v)\in \mathcal{G}\cup \mathcal{G}^{*}$ for every vertex $v\in V(T)$ or $M$ has a minor isomorphic to one of $U_{2,4}$, $F_{7}$, $F_{7}^{*}$, $R_{10}$, $R_{12}$.
\end{lemma}

\begin{proof}
Suppose that there is a vertex $v\in V(T)$ such that $\rho(v)\notin\mathcal{G}\cup\mathcal{G}^{*}$. Then $\rho(v)$ is 3-connected. By Lemma~\ref{lem:3c}, $\rho(v)$ is not delta-graphic. By Lemma~\ref{3connected}, $\rho(v)$ has a minor isomorphic to one of $U_{2,4}$, $F_{7}$, $F_{7}^{*}$, $R_{10}$, $R_{12}$. \end{proof}

\begin{lemma}
\label{planar}
Let $M$ be a connected matroid with a canonical tree decomposition $(T,\rho)$. If $M$ is minor-minimally not delta-graphic, then $\rho(v)\notin\mathcal{G}\cap\mathcal{G}^{*}$ for every leaf $v$ of $T$.
\end{lemma}

\begin{proof}
Suppose that $\rho(v)\in\mathcal{G}\cap\mathcal{G}^{*}$ for some leaf $v$ of $T$. Since $M$ is not delta-graphic, $|V(T\setminus v)|\geq 1$. Then $\rho(T\setminus v)$ is a delta-graphic matroid with a canonical tree decomposition 
$(T\setminus v,\rho|_{V(T\setminus v)})$. 
Let $T':=T\setminus v$ and $u$ be a vertex of $T$ which is adjacent to $v$. If $\rho(T')\in\mathcal{G}\cup\mathcal{G}^{*}$, then $M=\rho(T')\oplus_{2}\rho(v)\in\mathcal{G}\cup\mathcal{G}^{*}$, contradicting our assumption. Hence, by Proposition~\ref{prop:deltagraphictree}, $(T',\rho|_{V(T')})$ satisfies one of \ref{item:t1}, \ref{item:t2}, \ref{item:t3}, \ref{item:t4}.

If $(T',\rho|_{V(T')})$ satisfies \ref{item:t1}, then there is an edge $e\in E(T')$ such that, for two components $T_{1}', T_{2}'$ of $T'\setminus e$, $\rho(T_{1}')\in\mathcal{G}$ and  $\rho(T_{2}')\in\mathcal{G}^{*}$. Let $T_{1}$, $T_{2}$ be components of $T\setminus e$ such that $T_{1}$ contains $T_{1}'$. Then $T_{1}'=T_{1}$ or $T_{2}'=T_{2}$. If $T_{1}'=T_{1}$, then 
$\rho(T_{2})=\rho(v)\oplus_{2}\rho(T_{2}')\in \mathcal{G}^{*}$. If $T_{2}'=T_{2}$, then $\rho(T_{1})=\rho(v)\oplus_{2}\rho(T_{1}')\in \mathcal{G}$. So $(T,\rho)$ satisfies \ref{item:t1} and $M$ is delta-graphic, contradicting our assumption.

Suppose that $(T',\rho|_{V(T')})$ satisfies \ref{item:t2}. There is a vertex $w$ of $T'$ such that 
\begin{enumerate}[label=\rm (\roman*)]
\item $\rho(w)$ is isomorphic to $U_{1,m}$ or $U_{m-1,m}$ for some $m\geq 3$, 
\item for each component $C$ of $T'\setminus w$, $\rho(C)\in\mathcal{G}\cup\mathcal{G}^{*}$, and
\item there are distinct components $C_{1}$, $C_{2}$ of $T'\setminus w$ such that $\rho(C_{1})\in\mathcal{G}$ and $\rho(C_{2})\in\mathcal{G}^{*}$. 
\end{enumerate}

If $u=w$, then it is obvious that $(T,\rho)$ satisfies \ref{item:t2}. Suppose that $u\neq w$. Then there is a component $C'$ of $T'\setminus w$ containing $u$. Let $C=T[V(C')\cup\{v\}]$. Then $C$ is a component of $T\setminus w$ and $\rho(C)=\rho(C')\oplus_{2}\rho(v)$. So $\rho(C)\in\mathcal{G}$ if $\rho(C')\in\mathcal{G}$ and $\rho(C)\in\mathcal{G}^{*}$ if $\rho(C')\in\mathcal{G}^{*}$. Hence, $(T,\rho)$ satisfies \ref{item:t2}, contradicting our assumption.

Suppose that $(T',\rho|_{V(T\setminus v)})$ satisfies \ref{item:t3}. Let $x\in V(T')$ be the hub. 
If $u=x$, then $(T,\rho)$ is also a wheel decomposition with the hub $x$. So we may assume that $u\neq x$. Let $X'$ be a component of $T'\setminus x$ containing $u$ and $e_{X'}$ be an edge of $T'$ joining $x$ and a vertex of $X'$. By \ref{item:w2}, $\rho(x)$ has a set $C$ which is a circuit-hyperplane such that $\rho(X')\in\mathcal{G}$ if $e_{X'}\in C$ and $\rho(X')\in\mathcal{G}^{*}$ otherwise. Let $X=T[V(X')\cup\{v\}]$. Then $X$ is a component of $T\setminus x$ and $\rho(X)=\rho(X')\oplus_{2}\rho(v)$. Hence, $\rho(X)\in\mathcal{G}$ if $\rho(X')\in\mathcal{G}$ and $\rho(X)\in\mathcal{G}^{*}$ if $\rho(X')\in\mathcal{G}^{*}$. So $(T,\rho)$  is a wheel decomposition with the hub $x$ and \ref{item:t3} holds, contradicting our assumption.

Now it remains to consider when $(T',\rho|_{V(T\setminus v)})$ satisfies \ref{item:t4}. Let $P$ be the spine. Suppose that $u\in V(P)$. Let $e_{1}:=uv$. If $\deg_{T}(u)\neq|E(\rho(u))|$, obviously $(T,\rho)$ is also a fan decomposition with the spine $P$. If $\deg_{T}(u)=|E(\rho(u))|$, then $T'$ has an edge $e_{2}\notin E(P)$ incident with $u$. Since $\rho(T_{e_{1}}(u))\in\mathcal{G}\cap\mathcal{G}^{*}$ and $\rho(T_{e_{2}}(u))\in\mathcal{G}\cup\mathcal{G}^{*}$, \ref{item:f3} holds and $(T,\rho)$ is a fan decomposition with the spine~$P$. 

So we may assume that $u\notin V(P)$. Let $X'$ be a component of $T'\setminus V(P)$ containing $u$ and let $e$ be an edge joining $w\in V(P)$ and a vertex of $X'$. Let $X=T[X'\cup\{v\}]$. Then $X=T_{e}(w)$ and $\rho(X)=\rho(X')\oplus_{2}\rho(v)$. So $\rho(X)\in\mathcal{G}$ if $\rho(X')\in\mathcal{G}$ and $\rho(X)\in\mathcal{G}^{*}$ if $\rho(X')\in\mathcal{G}^{*}$. So $(T,\rho)$ is a fan decomposition with the spine $P$, contradicting our assumption.
\end{proof}

For matroid-labelled trees $(T,\rho)$, $(T,\rho')$, we say that $(T,\rho)$ is \emph{equivalent} to $(T,\rho')$ if the following hold for each vertex $v$ of $T$:

\begin{enumerate}[label=\rm (\roman*)]
\item $\rho(v)\in\mathcal{G}$ if and only if $\rho'(v)\in\mathcal{G}$.
\item $\rho(v)\in\mathcal{G}^{*}$ if and only if $\rho'(v)\in\mathcal{G}^{*}$.
\item $\rho(v)$ is a uniform matroid of rank $1$ if and only if $\rho'(v)$ is a uniform matroid of rank $1$. 
\item $\rho(v)$ is a uniform matroid of corank $1$ if and only if $\rho'(v)$ is a uniform matroid of corank $1$. 
\item For $k\geq 3$, $\rho(v)$ is isomorphic to $M(W_{k})$ if and only if $\rho'(v)$ is isomorphic to $M(W_{k})$.
\item $\rho(v)$ is $3$-connected if and only if $\rho'(v)$ is $3$-connected.
\end{enumerate}

\begin{lemma}
\label{lem:Tiso}
Let $M$ be a connected matroid with a canonical tree decomposition $(T,\rho)$ and, for each vertex $v\in V(T)$, let $N_{v}$ be a connected minor of $\rho(v)$ such that $E(N_{v})\cap E(T)=E(\rho(v))\cap E(T)$ and $E(N_{v})\cap E(T)$ has neither loops nor coloops in $N_{v}$. Let $(T,\rho')$ be a matroid-labelled tree such that $\rho'(v)=N_{v}$ for each $v\in V(T)$ and $(T,\rho')$ is equivalent to $(T,\rho)$. Then $\rho'(T)$ is delta-graphic if and only if $\rho(T)$ is delta-graphic.
\end{lemma}

\begin{proof}
The backward direction is obvious. So we prove the forward direction.
If $|E(N_{v})|<2$ for some $v\in V(T)$, then every element of $N_v$ is a loop or a coloop, and thus $E(N_v)\cap E(T)=\emptyset$ and $\abs{V(T)}=1$. 
Since $(T,\rho)$ is equivalent to $(T,\rho')$, 
by (i) and (ii), $\rho(v)=M$ is both graphic and cographic and therefore $\rho(T)$ is delta-graphic. 

If $|E(N_{v})|=2$ for some $v\in V(T)$, then $N_{v}$ is isomorphic to $U_{1,2}$ 
because $N_{v}$ is connected. Since $(T,\rho)$ is equivalent to $(T,\rho')$, 
by (iii) and (iv), $\rho(v)$ is also isomorphic to $U_{1,2}$ and by the definition of a tree decomposition, $|V(T)|=1$ and $\rho(v)=M$ is isomorphic to $U_{1,2}$, which is delta-graphic.

So we can assume that $|E(N_{v})|\geq 3$ for each $v\in V(T)$. Then $(T,\rho')$ is a canonical tree decomposition of $\rho'(T)$ because $(T,\rho')$ is equivalent to $(T,\rho)$. Obviously, $M\in\mathcal{G}\cup\mathcal{G}^{*}$ if $\rho'(T)\in\mathcal{G}\cup\mathcal{G}^{*}$. Therefore, we may assume that $\rho'(T)\notin\mathcal{G}\cup\mathcal{G}^{*}$. Then, by Proposition~\ref{prop:deltagraphictree}, $(T,\rho')$ satisfies at least one of \ref{item:t1}, \ref{item:t2}, \ref{item:t3}, or \ref{item:t4}. 

It is obvious that $(T,\rho)$ satisfies \ref{item:t1}, \ref{item:t2}, \ref{item:t3} if $(T,\rho')$ satisfies \ref{item:t1}, \ref{item:t2}, \ref{item:t3}, respectively. Suppose that $(T,\rho')$ satisfies \ref{item:t4} and let $P$ be the spine of $(T,\rho')$. We claim that $(T,\rho)$ is a fan decomposition with the spine $P$. Obviously, \ref{item:f1} and \ref{item:f2} hold. Let $v$ be an end vertex of $P$ in $T$. If $|E(\rho(v))|=|E(\rho'(v))|$, then \ref{item:f3} obviously holds for $(T,\rho)$. If not, $\rho'(v)$ is a proper minor of $\rho(v)$ and $|E(\rho(v))|>|E(\rho'(v))|\geq\deg_{T}(v)$. So \ref{item:f3} holds.
\end{proof}

Let $K_{3,3}'$ be a simple graph such that $K_{3,3}'\setminus e=K_{3,3}$ for some edge $e$.

\begin{lemma}
\label{lem:rootleafminor}
Let $G$ be a $3$-connected simple nonplanar graph and $e$ be an edge of $G$. Then, $G$ has a minor $H$ containing $e$ such that $H$ is isomorphic to $K_{5}$, $K_{3,3}$, or $K_{3,3}'$.
\end{lemma}

\begin{proof}
If $G$ is isomorphic to $K_{5}$, then the proof is done with $H=G$. So we may assume that $G$ is not isomorphic to $K_{5}$ and by Lemma~\ref{rm1}, $G$ has a minor $H$ isomorphic to $K_{3,3}$. 

We prove that if a $2$-connected graph $G$ has a minor $H$ isomorphic to $K_{3,3}$ and $e$ is an edge of~$G$, then $G$ has a minor which contains $e$ and is isomorphic to $K_{3,3}$ or $K_{3,3}'$. We proceed by induction on $|E(G)|$. If $|E(G)|\leq 10$, then $G$ is a $K_{3,3}$-subdivision or is isomorphic to $K_{3,3}'$. So  we may assume that $|E(G)|>10$. Then there is an edge $f\in E(G)- (E(H)\cup\{e\})$ and by Lemma~\ref{lem:Brylawski} and the induction hypothesis, $G/f$ or $G\setminus f$ contains a minor $H'$ which contains $e$ and is isomorphic to $K_{3,3}$ or $K_{3,3}'$. Hence, $G$ contains a minor $H'$.
\end{proof}

\begin{lemma}
\label{lem:nonplanar}
For a $3$-connected non-cographic and graphic matroid $M$ and a nonempty subset $X$ of $E(M)$, $M$ has a $3$-connected minor $N$ with $|E(N)|\leq 4|X|+6$ such that $N$ is non-cographic and contains $X$.
\end{lemma}

\begin{proof}
We use induction on $k=|X|$. We may assume $k>1$ by Lemma~\ref{lem:rootleafminor}. Let $e$ be an element of $X$. Then by the induction hypothesis, there is a non-cographic $3$-connected minor $N$ such that $X-\{e\}\subseteq E(N)$ and $|E(N)|\leq 4(|X|-1)+6$.

If $e\in E(N)$, then the proof is done. If $e\in E(M)- E(N)$, then by Lemma~\ref{lem: bixby}, there is a $3$-connected minor $N'$ of $M$ containing $e$ such that $|E(N')- E(N)|\geq 4$, and $N$ is a minor of $N'$. Then $N'$ is non-cographic and contains $X$ and $|E(N')|\leq |E(N)|+4\leq 4|X|+6$.
\end{proof}

\begin{lemma}
\label{ReducingVertex}
Let $M$ be a connected matroid with a canonical tree decomposition $(T,\rho)$. If $M$ is minor-minimally not delta-graphic and is not $3$-connected, then, for each vertex $v$ of $T$, the following hold:
\begin{enumerate}[label=\rm (\roman*)]
\item If $\rho(v)\notin\mathcal{G}\cap\mathcal{G}^{*}$, then $|E(\rho(v))|\leq 4\deg_{T}(v)+6$.

\item If $\rho(v)$ is a uniform matroid of rank $1$ or corank $1$, then $|E(\rho(v))|=\deg_{T}(v)$ or $|E(\rho(v))|=3$.
\end{enumerate}
\end{lemma}
\begin{proof}
First we prove (i). By Lemma \ref{GorCG}, $M$ has a minor isomorphic to one of $U_{2,4}$, $F_{7}$, $F_{7}^{*}$, $R_{10}$, or $R_{12}$ or $\rho(v)\in\mathcal{G}\cup\mathcal{G}^{*}$ for each $v\in V(T)$. Since $M$ is not $3$-connected and minor-minimally not delta-graphic, $M$ has no minor isomorphic to one of $U_{2,4}$, $F_{7}$, $F_{7}^{*}$, $R_{10}$, or $R_{12}$ by Lemma~\ref{lem:3c}. So $\rho(v)\in\mathcal{G}\cup\mathcal{G}^{*}$ for each $v\in V(T)$. Suppose that $\rho(v)\notin\mathcal{G}\cap\mathcal{G}^{*}$ for some $v$. Then $\rho(v)\in\mathcal{G}- \mathcal{G}^{*}$ or $\rho(v)\in\mathcal{G}^{*}- \mathcal{G}$. By duality, we may assume that $\rho(v)\in\mathcal{G}- \mathcal{G}^{*}$. Suppose that $|E(\rho(v))|>4\deg_{T}(v)+6$. Then, by Lemma~\ref{lem:nonplanar}, $\rho(v)$ has a non-cographic proper minor $H$ containing $E(\rho(v))\cap E(T)$. Let

\[
\rho_{1}(u)=
\begin{cases}
\rho(u) & \text{if $u\neq v$,} \\
H & \text{if $u=v$.}
\end{cases}
\] 
Then $(T,\rho_{1})$ is equivalent to $(T,\rho)$, and by Lemma~\ref{lem:Tiso}, $\rho_{1}(T)$ is not delta-graphic, contradicting the minimality of $M$. So $|E(\rho(v))|\leq 4\deg_{T}(v)+6$.

Now it is enough to show (ii). By duality, we may assume that $\rho(v)$ is a uniform matroid of corank~$1$. If $\deg_{T}(v)<|E(\rho(v))|$ and $|E(\rho(v))|\geq 4$, then there is an element $e\in E(\rho(v))- E(T)$. Observe that $\rho(v)/e$ is a uniform matroid of corank $1$ such that $|E(\rho(v)/e)|\geq 3$. Let
\[
\rho_{2}(u)=
\begin{cases}
\rho(u) & \text{if $u\neq v$,} \\
\rho(v)/e & \text{if $u=v$.}
\end{cases}
\] 
Since $(T,\rho_{2})$ is equivalent to $(T,\rho)$, by Lemma~\ref{lem:Tiso}, $M/e$ is not delta-graphic, contradicting the minimality of $M$. We remark that $(T,\rho_{2})$ is not equivalent to $(T,\rho)$ if $|E(\rho(v)/e)|<3$.
\end{proof}

\begin{lemma}
\label{Reducingleaf}
Let $M$ be a connected matroid with a canonical tree decomposition $(T,\rho)$. If $M$ is minor-minimally not delta-graphic and is not $3$-connected, then for each leaf $v$ of $T$, $\rho(v)$ is isomorphic to $M(K_{5})$, $M(K_{3,3})$, $M(K_{3,3}')$, $M^{*}(K_{5})$, $M^{*}(K_{3,3})$, or $M^{*}(K_{3,3}')$. 
\end{lemma}

\begin{proof}
By Lemma~\ref{planar}, $\rho(v)\notin\mathcal{G}\cap\mathcal{G}^{*}$. By duality, we may assume that $\rho(v)\in\mathcal{G}-\mathcal{G}^{*}$. Then $\rho(v)$ is isomorphic to $M(G)$ for a $3$-connected nonplanar simple graph. Let $e$ be the edge of $T$ incident with~$v$. Then by Lemma~\ref{lem:rootleafminor}, $\rho(v)$ has a minor $N$ which contains $e$ and is isomorphic to $M(K_{5})$, $M(K_{3,3})$, or $M(K_{3,3}')$. Then, for each $w\in V(T)$, let
\[
\rho'(w)=
\begin{cases}
\rho(w) & \text{if $w\neq v$,} \\
N & \text{if $w=v$.}
\end{cases}
\]
Since $(T,\rho')$ is equivalent to $(T,\rho)$, by Lemma~\ref{lem:Tiso}, $\rho'(T)$ is not delta-graphic. By the minimality of~$M$, $M=\rho'(T)$ and $\rho(v)=N$.
\end{proof}

\begin{lemma}
\label{Reducingsubtree}
Let $M$ be a connected matroid minor-minimally not delta-graphic.
Let $(T,\rho)$ be a canonical tree decomposition of $M$.
Let $e=xy$ be an edge of $T$. If $\rho(T_{e}(y))\in\mathcal{G}\cup\mathcal{G}^{*}$, then $x$ is a leaf of $T$. 
\end{lemma}

\begin{proof}
By Lemma~\ref{planar}, $\rho(T_{e}(y))\notin\mathcal{G}\cap\mathcal{G}^{*}$. So by duality, we may assume that $\rho(T_{e}(y))\in\mathcal{G}-\mathcal{G}^{*}$. If $\rho(T_{e}(x))\in\mathcal{G}\cup\mathcal{G}^{*}$, then $M\in\mathcal{G}\cup\mathcal{G}^{*}$ or $(T,\rho)$ satisfies \ref{item:t1}. So by Proposition~\ref{prop:deltagraphictree}, $M$ is delta-graphic, contradicting our assumption. So $\rho(T_{e}(x))\notin\mathcal{G}\cup\mathcal{G}^{*}$. Suppose that $x$ is not a leaf of $T$. 
Then there is a leaf $u\neq x$ in $T_{e}(y)$. Let $f$ be an edge incident with $u$ and $N$ be a matroid obtained from $\rho(u)$ by relabelling $f$ to $e$. Since $u\in V(T_{e}(y))$ is a leaf, by Lemma~\ref{planar}, $\rho(u)\notin\mathcal{G}\cap\mathcal{G}^{*}$. Since $\rho(u)$ is isomorphic to a minor of $\rho(T_{e}(y))$ by Lemma~\ref{subtreeminor}, we have $\rho(u)\in\mathcal{G}-\mathcal{G}^{*}$, and therefore $N\in\mathcal{G}-\mathcal{G}^{*}$.
Let $T'$ be a tree obtained from $T[\{u\}\cup V(T_{e}(x))]$ by adding an edge $e$ joining $u$ and $y$ and, for each $w\in V(T')$, let
\[
\rho'(w)=
\begin{cases}
\rho(w) & \text{if $w\in V(T')-\{u\}$,} \\
N & \text{if $w=u$.}
\end{cases}
\]
Then, by Lemmas~\ref{subtreeminor} and \ref{pathcontraction}, $\rho'(T')$ is isomorphic to a proper minor of $M$. So $\rho'(T')$ is delta-graphic. Since $N\notin\mathcal{G}\cap\mathcal{G}^{*}$, $\{u,y\}$ is not bad and $(T',\rho')$ is a canonical tree decomposition of $\rho'(T')$ by Lemma~\ref{pathcontraction}. Obviously, $\rho'(T'_{e}(y))=\rho'(u)\in\mathcal{G}-\mathcal{G}^{*}$.

Since $T'_{e}(u)=T_{e}(x)$ and $\rho(T_{e}(x))\notin\mathcal{G}\cup\mathcal{G}^{*}$, we have $\rho'(T')\notin\mathcal{G}\cup\mathcal{G}^{*}$. So by Proposition~\ref{prop:deltagraphictree}, $(T',\rho')$ satisfies at least one of \ref{item:t1}, \ref{item:t2}, \ref{item:t3}, and \ref{item:t4}. 

Suppose that $(T',\rho')$ satisfies \ref{item:t1}. Then there is an edge $f\in E(T')$ such that $\rho'(T_{1}')\in\mathcal{G}, \rho'(T_{2}')\in\mathcal{G}^{*}$ for two components $T_{1}'$, $T_{2}'$ of $T'\setminus f$. 

Since $\rho'(u)\in\mathcal{G}-\mathcal{G}^{*}$ and $\rho'(T'_{2})\in\mathcal{G}^{*}$, we have $u\in V(T'_{1})$. Let $T_{1}=T[V(T_{1}')\cup V(T_{e}(y))]$. Then $T_{1}$ and $T_{2}'$ are components of $T\setminus f$ such that $\rho(T_{1})\in\mathcal{G}$ and $\rho(T_{2}')\in\mathcal{G}^{*}$. So $(T,\rho)$ satisfies \ref{item:t1} and $M$ is delta-graphic by Proposition~\ref{prop:deltagraphictree}, contradicting our assumption.

Suppose that $(T',\rho')$ satisfies \ref{item:t2}. Then there is a vertex $v$ of $T'$ such that 
\begin{enumerate}[label=\rm (\roman*)]
\item $\rho'(v)$ is isomorphic to $U_{1,m}$ or $U_{m-1,m}$ for some $m\geq 3$, 
\item for each component $C$ of $T'\setminus v$, $\rho'(C)\in\mathcal{G}\cup\mathcal{G}^{*}$, and
\item there are distinct components $C_{1}$, $C_{2}$ of $T'\setminus v$ such that $\rho'(C_{1})\in\mathcal{G}$ and $\rho'(C_{2})\in\mathcal{G}^{*}$. 
\end{enumerate} 
Let $X'$ be a component of $T'\setminus v$ containing $u$. Since $\rho'(u)\notin\mathcal{G}^{*}$, we have $\rho'(X')\in\mathcal{G}$. Let $X:=T[V(X')\cup V(T_{e}(y))]$.
Then $X$ is a component of $T\setminus v$ and $\rho(X)\in\mathcal{G}$. So $(T,\rho)$ satisfies \ref{item:t2} and $M$ is delta-graphic by Proposition~\ref{prop:deltagraphictree}, contradicting our assumption.

If $(T',\rho')$ satisfies \ref{item:t3}, let $z$ be the hub. Let $Y'$ be a component of $T'\setminus z$ containing $u$. Since $\rho'(u)\notin\mathcal{G}^{*}$, we have  $\rho'(Y')\in\mathcal{G}$. Let $Y=T[V(Y')\cup V(T_{e}(y))]$. Then $Y$ is a component of $T\setminus z$ and $\rho(Y)\in\mathcal{G}$. So $(T,\rho)$ satisfies \ref{item:t3} and $M$ is delta-graphic by Proposition~\ref{prop:deltagraphictree}, contradicting our assumption.

If $(T',\rho')$ satisfies \ref{item:t4}, then let $P$ be the spine. Since $\rho(u)\notin\mathcal{G}\cap\mathcal{G}^{*}$, we have $u\notin V(P)$. Let $Z'$ be a component of $T'\setminus V(P)$ containing $u$. Since $\rho'(u)\notin\mathcal{G}^{*}$, we have $\rho'(Z')\in\mathcal{G}$. Let $Z=T[V(Z')\cup V(T_{e}(y))]$. Then $Z$ is a component of $T\setminus z$ and $\rho(Z)\in\mathcal{G}$. Therefore, $(T,\rho)$ satisfies \ref{item:t4} and $M$ is delta-graphic, contradicting our assumption. 
\end{proof}

\subsection{Excluding tripods}

A matroid $M$ is a \emph{tripod} if $M=M_{1}\oplus_{2}M_{2}\oplus_{2}M_{3}$ for $3$-connected matroids $M_{1}$, $M_{2}$, $M_{3}$ such that $|E(M_{1})\cap E(M_{2})|=|E(M_{2})\cap E(M_{3})|=1$, $E(M_{1})\cap E(M_{3})=\emptyset$, and 

\begin{enumerate}[label=\rm (\roman*)]
\item $M_{1}, M_{3}\in\mathcal{G}-\mathcal{G}^{*}$ and $M_{2}\in\mathcal{G}^{*}-\mathcal{G}$, or
\item $M_{1}, M_{3}\in\mathcal{G}^{*}-\mathcal{G}$ and $M_{2}\in\mathcal{G}-\mathcal{G}^{*}$.
\end{enumerate}
Observe that a canonical tree decomposition of a tripod $M=M_{1}\oplus_{2}M_{2}\oplus_{2}M_{3}$ is a pair $(T,\rho)$ such that $T$ is a path $v_{1}v_{2}v_{3}$ of length $2$ and $\rho(v_{i})=M_{i}$ for each $i\in\{1,2,3\}$.

\begin{lemma}
\label{tripod}
No tripod is delta-graphic.
\end{lemma}
\begin{proof}
The conclusion follows from Proposition~\ref{prop:deltagraphictree}.
\end{proof}

For a tree $T$ and vertices $x$, $y$ of $T$, let $T_{x,y}$ be the path of $T$ from $x$ to $y$.

\smallskip

For a connected matroid $M$ with a canonical tree decomposition $(T,\rho)$, a triple $(u,v,w)$ of distinct vertices of $T$ is \emph{flexible} if $v\in V(T_{u,w})$ and one of the following hold:
\begin{enumerate}[label=\rm (\roman*)]
\item $\rho(u), \rho(w)\in \mathcal{G}-\mathcal{G}^{*}$ and $\rho(v)\in \mathcal{G}^{*}-\mathcal{G}$.
\item $\rho(u), \rho(w)\in \mathcal{G}^{*}-\mathcal{G}$ and $\rho(v)\in \mathcal{G}-\mathcal{G}^{*}$. 
\end{enumerate}

\begin{lemma}
\label{lem:flexible}
Let $M$ be a connected matroid  minor-minimally not delta-graphic.
Let $(T,\rho)$ be a canonical tree decomposition of $M$.
If $M$ has a flexible triple, then $|E(M)|\leq 30$.
\end{lemma}

\begin{proof}
Let $(u,v,w)$ be a flexible triple of $M$.
By taking dual, we can assume that $\rho(u), \rho(w)\in \mathcal{G}-\mathcal{G}^{*}$ and $\rho(v)\in \mathcal{G}^{*}-\mathcal{G}$. Let  $P:=T_{u,w}$. By Lemma \ref{subtreeminor}, $\rho(P)$ is isomorphic to a minor of $M$ with a canonical tree decomposition $(P,\rho|_{V(P)})$. Let $e_{u}$, $e_{w}$ be edges of $P$ incident with $u$, $w$, respectively. Let $f_{1}$, $f_{2}$ be edges of $P$ which are incident with $v$ such that $f_{1}\in E(T_{u,v})$ and $f_{2}\in E(T_{v,w})$. Let $N$ be a matroid obtained from $\rho(v)$ by relabelling $f_{1}$ by $e_{u}$ and $f_{2}$ by $e_{w}$. Let $P'=uvw$ be a path graph and 
\[
\rho'(x)=
\begin{cases}
\rho(x) & \text{if $x\in\{u,w\}$,} \\
N & \text{if $x=v$.}
\end{cases}
\]
By Lemma~\ref{pathcontraction}, $\rho(P)$ has a minor $\rho'(P')$ with a canonical tree decomposition $(P',\rho')$. By Lemma~\ref{tripod}, $\rho'(P')$ is a tripod and is not delta-graphic. Since $M$ is minor-minimally not delta-graphic, $M=\rho'(P')=\rho(u)\oplus_{2}N\oplus_{2}\rho(w)$. Observe that $\rho(u),\rho(w),N\notin\mathcal{G}\cap\mathcal{G}^{*}$, $\deg_{P'}(u)=\deg_{P'}(w)=1$, and $\deg_{P'}(v)=2$. So by Lemmas~\ref{ReducingVertex} and \ref{Reducingleaf}, $|E(\rho(u))|, |E(\rho(w))|\leq 10$ and $|E(\rho(v))|\leq 14$. 
Therefore, $|E(M)|=|E(\rho(u))|+|E(\rho(w))|+|E(\rho(v))|-4\leq 30$.
\end{proof} 

\begin{lemma}
\label{notriaxialstrong}
Let $M$ be a connected matroid  minor-minimally not delta-graphic.
Let $(T,\rho)$ be a canonical tree decomposition of $M$.
If there is no flexible triple, then $|E(M)|\leq 12$ or $\rho(v)\in\mathcal{G}\cap\mathcal{G}^{*}$ for every internal vertex $v$ of $T$.
\end{lemma}

\begin{proof}
If $M$ has a minor isomorphic to one of $U_{2,4}$, $F_{7}$, $F_{7}^{*}$, $R_{10}$, and $R_{12}$, then $|E(M)|\leq 12$ by minimality of $M$. So we may assume that $M$ has no minor isomorphic to one of $U_{2,4}$, $F_{7}$, $F_{7}^{*}$, $R_{10}$, and $R_{12}$.

Suppose that there is an internal vertex $v$ of $T$ such that $\rho(v)\notin\mathcal{G}\cap\mathcal{G}^{*}$. By Lemma~\ref{GorCG}, $\rho(v)\in\mathcal{G}\cup\mathcal{G}^{*}$ and, by duality, we may assume that $\rho(v)\in\mathcal{G}-\mathcal{G}^{*}$. If there are distinct components $X_{1}, X_{2}$ of $T\setminus v$ such that $\rho(X_{1}),\rho(X_{2})\notin\mathcal{G}$, then there are $u\in V(X_{1})$, $w\in V(X_{2})$ such that $\rho(u), \rho(w)\notin\mathcal{G}$. Then by Lemma~\ref{GorCG}, $\rho(u), \rho(w)\in\mathcal{G}^{*}-\mathcal{G}$. Hence, $(u,v,w)$ is a flexible triple, contradicting our assumption. Since $M\notin\mathcal{G}$, there is a unique component $X$ of $T\setminus v$ such that $\rho(X)\notin\mathcal{G}$. Let $e$ be an edge of $T$ joining $v$ and a vertex $v'$ of $X$. Since $\rho(T_{e}(v'))\in\mathcal{G}$, by Lemma~\ref{Reducingsubtree}, $v$ is a leaf of $T$, contradicting our assumption.
\end{proof}

\subsection{Excluding $H$-matroids and $H'$-matroids}

\begin{figure}
\centering
\tikzstyle{v}=[circle, draw,fill=black,inner sep=0pt, minimum width=3pt]
  \begin{tikzpicture}
    \draw (0,1) node [label=left:$x_1$,v] (x1){}
    -- (0,0) node [label=left:$v_1$,v] (v1){}
    -- (0,-1)node[label=left:$y_1$,v] (y1){};
    \draw (v1)-- (2,0) node [label=right:$v_2$,v](v2){}
    --(2,1) node[label=right:$x_2$,v](x2){};
    \draw (v2)--(2,-1)node[label=right:$y_2$,v]{};
    \draw (1,-1.5) node[rectangle]{(i)};
  \end{tikzpicture}
  $\qquad$
  \begin{tikzpicture}
    \draw (0,1) node [label=left:$x_1$,v] (x1){}
    -- (0,0) node [label=left:$v_1$,v] (v1){}
    -- (0,-1)node[label=left:$y_1$,v] (y1){};
    \draw (v1)-- (1,0) node[v,label=$v$]{}--
    (2,0) node [label=right:$v_2$,v](v2){}
    --(2,1) node[label=right:$x_2$,v](x2){};
    \draw (v2)--(2,-1)node[label=right:$y_2$,v]{};
    \draw (1,-1.5) node[rectangle]{(ii)};
  \end{tikzpicture}
\caption{(i)~ $H$-graph. (ii)~ $H'$-graph.}
\label{fig:Hgraph}
\end{figure}
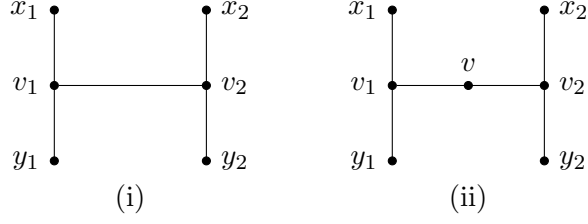
A connected matroid with a canonical tree decomposition $(T,\rho)$ is an \emph{$H$-matroid} if $T$ is an $H$-graph in Figure~\ref{fig:Hgraph} and $(T,\rho)$ satisfies the following:
\begin{enumerate}[label=\rm (h\arabic*)]
\item\label{item:h1} $\rho(v_{1})$ or $\rho(v_{2})$ is neither a uniform matroid of rank $1$ nor a uniform matroid of corank $1$.
\item\label{item:h2} $\rho(x_{1}), \rho(x_{2})\in \mathcal{G}-\mathcal{G}^{*}$ and $\rho(y_{1}), \rho(y_{2})\in \mathcal{G}^{*}-\mathcal{G}$.
\end{enumerate}

A connected matroid with a canonical tree decomposition $(T,\rho)$ is an $H'$-matroid if $T$ is an $H'$-graph in Figure~\ref{fig:Hgraph} and $(T,\rho)$ satisfies the following:
\begin{enumerate}[label=\rm (H\arabic*)]
\item\label{item:H1} $\rho(v)$ is neither a uniform matroid of rank $1$ nor a uniform matroid of corank $1$.
\item\label{item:H2} $\rho(x_{1}), \rho(x_{2})\in \mathcal{G}-\mathcal{G}^{*}$ and $\rho(y_{1}), \rho(y_{2})\in \mathcal{G}^{*}-\mathcal{G}$.
\end{enumerate}

\begin{lemma}
\label{lem:Hgraph}
No $H$-matroid is delta-graphic and no $H'$-matroid is delta-graphic.
\end{lemma}

\begin{proof}
Let $M$ be an $H$-matroid or an $H'$-matroid with a canonical tree decomposition $(T,\rho)$. Suppose that $M$ is delta-graphic. Since $M\notin\mathcal{G}\cup\mathcal{G}^{*}$, by Proposition~\ref{prop:deltagraphictree}, $(T,\rho)$ satisfies at least one of \ref{item:t1}, \ref{item:t2}, \ref{item:t3}, and \ref{item:t4}. By \ref{item:h1} and \ref{item:h2} or \ref{item:H1} and \ref{item:H2}, $T$ has no adjacent vertices $u$, $v$ such that $\rho(u)$ and $\rho(v)$ are uniform matroids of rank $1$ or corank $1$. So $(T,\rho)$ does not satisfy \ref{item:t4}. For each $e\in E(T)$, there is a component $C'$ of $T\setminus e$ such that $\{x_{1},y_{1}\}\subseteq V(C')$ or $\{x_{2},y_{2}\}\subseteq V(C')$. Hence, $\rho(C')\notin\mathcal{G}\cup\mathcal{G}^{*}$ and $(T,\rho)$ does not satisfy \ref{item:t1}, \ref{item:t2}, and \ref{item:t3}, contradicting our assumption.
\end{proof}

Let $M$ be a connected matroid with a canonical tree decomposition $(T,\rho)$. Then a vertex $v$ of $T$ is a \emph{core} if $\rho(v)\in\mathcal{G}\cap\mathcal{G}^{*}$, $\rho(v)$ is neither a uniform matroid of rank $1$ nor a uniform matroid of corank~$1$, and it satisfies one of the following:
\begin{enumerate}[label=\rm (R\arabic*)]
\item\label{item:R1} There are distinct components $X_{1}$, $X_{2}$ of $T\setminus v$ such that $\rho(X_{1}), \rho(X_{2})\notin\mathcal{G}\cup\mathcal{G}^{*}$.
\item\label{item:R2} There is exactly one component $X$ of $T\setminus v$ such that $\rho(X)\notin\mathcal{G}\cup\mathcal{G}^{*}$ and there are components $Y_{1}, Y_{2}$ of $T\setminus v$ such that $\rho(Y_{1})\in\mathcal{G}-\mathcal{G}^{*}$ and $\rho(Y_{2})\in\mathcal{G}^{*}-\mathcal{G}$.
\end{enumerate}

\begin{lemma}
\label{U13minor}
Let $M$ be a connected matroid with $|E(M)|\geq 3$. Then, for each set $X\subseteq E(M)$ with $|X|=3$, $M$ has a minor on $X$ isomorphic to $U_{1,3}$ or $U_{2,3}$.
\end{lemma}

\begin{proof}
By Lemma~\ref{3circuit}, there is a subset $C$ of $E(M)$ which is a circuit or a cocircuit of $M$ containing~$X$. If $C$ is a circuit, then $M\setminus(E(M)- C)/(C- X)$ is a minor on $X$ which is isomorphic to $U_{2,3}$. If $C$ is a cocircuit, then $M/(E(M)- C)\setminus(C- X)$ is a minor on $X$ which is isomorphic to $U_{1,3}$. 
\end{proof}

\begin{lemma}
\label{lem:rk3}
Let $M$ be a $3$-connected binary matroid with $|E(M)|\geq 3$. Then $M$ is isomorphic to $U_{1,3}$ or $U_{2,3}$, or $r(M),r^{*}(M)\geq 3$.
\end{lemma}
\begin{proof}
Since $M$ is connected, $M$ has no loops and no coloops. 
So if $|E(M)|=3$, then $M$ is isomorphic to $U_{1,3}$ or $U_{2,3}$. So we may assume that $|E(M)|\geq 4$. 

Since $M$ is $3$-connected, $M$ has no circuit of size $2$. Hence, every subset of $E(M)$ with size at most~$2$ is independent. If $r(M)=2$, then $M$ is isomorphic to $U_{2,n}$ for $n\geq 4$, contradicting our assumption that $M$ is binary. So $r(M)\geq 3$. By duality, we also have $r^{*}(M)\geq 3$.
\end{proof}

\begin{lemma}
\label{lem:core}
Let $M$ be a connected matroid minor-minimally not delta-graphic. Let $(T,\rho)$ be a canonical tree decomposition of $M$. If $(T,\rho)$ has a core $v$, then $|E(M)|\leq 40$.
\end{lemma}

\begin{proof}
If $(T,\rho)$ has a flexible triple, then $|E(M)|\leq 30$ by Lemma~\ref{lem:flexible}. So we may assume that $(T,\rho)$ has no flexible triples.
By Lemma~\ref{notriaxialstrong}, we may assume that $\rho(x)\in\mathcal{G}\cap\mathcal{G}^{*}$ for each internal vertex $x$. Suppose that a core $v$ satisfies \ref{item:R1} and let $X$, $Y$ be components of $T\setminus v$ such that $\rho(X),\rho(Y)\notin\mathcal{G}\cup\mathcal{G}^{*}$. Then $T$ has leaves $x_{1},y_{1}\in V(X)$ and $x_{2},y_{2}\in V(Y)$ such that $\rho(x_{1}),\rho(x_{2})\in\mathcal{G}-\mathcal{G}^{*}$ and $\rho(y_{1}),\rho(y_{2})\in\mathcal{G}^{*}-\mathcal{G}$ by Lemmas~\ref{GorCG} and \ref{planar}. Let $P_{1}:=T_{x_{1},y_{1}}$ and $P_{2}:=T_{x_{2},y_{2}}$. For each $i\in\{1,2\}$, let $Q_{i}$ be a shortest path from $v$ to $P_{i}$ and let $v_{i}\in V(P_{i})$ be an end of $E(Q_{i})$. Then $v_{i}\notin\{x_{i},y_{i}\}$ because $x_{i},y_{i}$ are leaves of $T$.

For each $i\in\{1,2\}$, let $e_{i}$, $f_{i}$ be edges of $P_{i}$ incident with $x_{i}$, $y_{i}$ respectively and let $g_{i}$, $g_{i}'$, $g_{i}''$ be edges incident with $v_{i}$ such that $g_{i}\in E(T_{v_{i},x_{i}})$, $g_{i}'\in E(T_{v_{i},y_{i}})$, and $g_{i}''\in E(T_{v,v_{i}})$. Let $h_{1}\in E(Q_{1})$, $h_{2}\in E(Q_{2})$ be edges incident with $v$.

For each $i\in\{1,2\}$, let $N_{i}$ be a matroid obtained from $\rho(v_{i})$ by relabelling $g_{i}$, $g_{i}'$, $g_{i}''$ to $e_{i}$, $f_{i}$, $h_{i}$ respectively. Let $T'$ be the $H'$-graph in Figure~\ref{fig:Hgraph} and, for each $w\in V(T')$, let 
\[
\rho'(w)=
\begin{cases}
\rho(w) & \text{if $w\in V(T')-\{v_{1},v_{2}\}$,} \\
N_{i} & \text{if $w=v_{i}$ for some $i\in\{1,2\}$.}
\end{cases}
\]
By Lemmas~\ref{subtreeminor} and \ref{pathcontraction}, $\rho'(T')$ is isomorphic to a minor of $\rho(T)$. For each $i\in\{1,2\}$, since $N_{i}$ is connected, by Lemma~\ref{U13minor}, $N_{i}$ has a minor $N_{i}'$ on $\{e_{i},f_{i},h_{i}\}$ which is isomorphic to $U_{1,3}$ or $U_{2,3}$. Moreover, since $\rho(v)$ is neither a uniform matroid of rank $1$ nor a uniform matroid of corank $1$, $r(\rho(v)),r^{*}(\rho(v))\geq 3$ by Lemma~\ref{lem:rk3}. Hence, by Lemma~\ref{MK4minor}, $\rho(v)$ has a minor $N$ isomorphic to $M(K_{4})$ using $h_{1}$ and $h_{2}$. So, for each vertex $w\in V(T')$, let 
\[
\rho''(w)=
\begin{cases}
\rho'(w) & \text{if $w\in V(T')-\{v,v_{1},v_{2}\}$,} \\
N_{i}' & \text{if $w\in\{v_{1},v_{2}\}$,} \\
N & \text{if $w=v$.}
\end{cases}
\] Then $\rho''(T')$ is isomorphic to a minor of $M$ which is the $H'$-matroid. By Lemma~\ref{lem:Hgraph} and the minimality of $M$, $M$ is isomorphic to $\rho''(T')$. Therefore, by Lemmas~\ref{ReducingVertex} and \ref{Reducingleaf}, 
\[
|E(M)|=\sum_{w\in\{x_{1},x_{2},y_{1},y_{2}\}}|E(\rho(w))|+|E(N_{1}')|+|E(N_{2}')|+|E(N)|-12\leq 40.
\]

Suppose that a core $v$ satisfies \ref{item:R2} and let $X$ be a component of  $T\setminus v$ such that $\rho(X)\notin\mathcal{G}\cup\mathcal{G}^{*}$ and $Y_{1}$, $Y_{2}$ be components of $T\setminus v$ such that $\rho(Y_{1})\in\mathcal{G}-\mathcal{G}^{*}$ and $\rho(Y_{2})\in\mathcal{G}^{*}-\mathcal{G}$. Then, by Lemmas~\ref{GorCG} and \ref{planar}, $T$ has leaves $x_{1},y_{1}\in V(X)$, $x_{2}\in V(Y_{1})$, and $y_{2}\in V(Y_{2})$ such that $\rho(x_{1}), \rho(x_{2})\in\mathcal{G}-\mathcal{G}^{*}$ and $\rho(y_{1}), \rho(y_{2})\in\mathcal{G}^{*}-\mathcal{G}$. Let $P_{1}:=T_{x_{1},y_{1}}$, $P_{2}:=T_{v,x_{2}}$ and $P_{3}:=T_{v,y_{2}}$. 
Let $Q$ be a shortest path from $v$ to $P_{1}$ and $v_{1}\in V(P_{1})$ is an end of $Q$. Then $v_{1}\notin\{x_{1},y_{1}\}$ because $x_{1}$, $y_{1}$ are leaves of $T$. For each $i\in\{1,2\}$, let $e_{i}$, $f_{i}$ be edges of $T$ incident with $x_{i}$, $y_{i}$ respectively. Let $v_{2}=v$ and for each $i\in\{1,2\}$, let $g_{i}$, $g_{i}'$, $g_{i}''$ be edges incident with $v_{i}$ such that $g_{i}\in E(Q)$, $g_{i}'\in E(T_{v_{i},x_{i}})$, and $g_{i}''\in E(T_{v_{i},y_{i}})$ and $N_{i}$ be a matroid obtained from $\rho(v_{i})$ by relabelling $g_{i}$, $g_{i}'$, $g_{i}''$ to $e:=g_{1},e_{i},f_{i}$ respectively. Let $T'$ be the $H$-graph in Figure~\ref{fig:Hgraph} and for $w\in V(T')$, let
\[
\rho'(w)=
\begin{cases}
\rho(w) & \text{if $w\in V(T')-\{v_{1},v_{2}\}$,} \\
N_{i} & \text{if $w=v_{i}$ for some $i\in\{1,2\}$.} 
\end{cases}
\]
By Lemmas~\ref{subtreeminor} and \ref{pathcontraction}, $\rho'(T')$ is isomorphic to a minor of $\rho(T)$. Since $N_{1}$ is connected, by Lemma~\ref{U13minor}, there is a minor $N_{1}'$ on $\{e,e_{1},f_{1}\}$ which is isomorphic to $U_{1,3}$ or $U_{2,3}$.
Moreover, since $N_{2}$ is neither a uniform matroid of rank $1$ nor a uniform matroid of corank $1$, $r(N_{2}),r^{*}(N_{2})\geq 3$ by Lemma~\ref{lem:rk3}. Hence, by Lemma~\ref{MK4minor}, $N_{2}$ has a minor $N_{2}'$ isomorphic to $M(K_{4})$ using $e$, $e_{2}$, and $f_{2}$. So, for each vertex $w\in V(T')$, let 
\[
\rho''(w)=
\begin{cases}
\rho'(w) & \text{if $w\in V(T')-\{v_{1},v_{2}\}$,} \\
N_{i}' & \text{if $w=v_{i}$ for some $i\in\{1,2\}$,} \\
\end{cases}
\] 
Then $\rho''(T')$ is isomorphic to a minor of $M$ which is the $H$-matroid. By Lemma~\ref{lem:Hgraph} and the minimality of $M$, $M$ is isomorphic to $\rho''(T')$. Therefore, by Lemmas~\ref{ReducingVertex} and \ref{Reducingleaf},
\[
|E(M)|=\sum_{w\in\{x_{1},x_{2},y_{1},y_{2}\}}|E(\rho(w))|+|E(N_{1}')|+|E(N_{2}')|-10\leq 40-1=39. \qedhere
\]
\end{proof}

\subsection{Excluding $(m,k)$-benches}

\begin{figure}[t]
\centering
\tikzstyle{v}=[circle, draw,fill=black,inner sep=0pt, minimum width=3pt]
  \begin{tikzpicture}
    \draw (0,1) node [label=left:$x_1$,v] (x1){}
    -- (0,0) node [label=left:$v_1$,v] (v1){}
    -- (0,-1)node[label=left:$y_1$,v] (y1){};
    \draw (4,0) node [label=right:$v_m$,v](v2){}
    --(4,1) node[label=right:$x_m$,v](x2){};
    \draw (v2)--(4,-1)node[label=right:$y_m$,v]{};
    \draw (v1) [dotted] to (2,0) node [v,label=$v_k$](vk) {} [dotted] to (v2);
    \draw (vk)--+(0,-1)node[v,label=right:$w$]{};
  \end{tikzpicture}
\caption{$(m,k)$-bench}
\label{fig:Bench}
\end{figure}
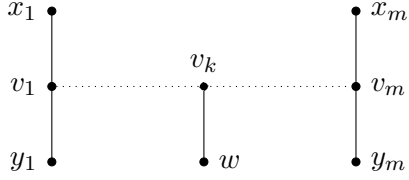

For $m\geq 3$ and $1<k<m$, an $(m,k)$-\emph{bench} is a tree obtained from a path $v_{1}v_{2}\cdots v_{m}$ of length $m-1$ by adding vertices $x_{1}$, $y_{1}$ adjacent to $v_{1}$, adding vertices $x_{m}$, $y_{m}$ adjacent to $v_{m}$, and adding a vertex~$w$ adjacent to $v_{k}$. See Figure~\ref{fig:Bench}.

A connected matroid $M$ with a canonical tree decomposition $(T,\rho)$ is an \emph{$(m,k)$-bench} for $m\geq 3$ and $1<k<m$ if $T$ is an $(m,k)$-bench and $(T,\rho)$ satisfies the following:

\begin{enumerate}[label=\rm(L\arabic*)]
\item\label{item:L1} For each $i\in\{1,\ldots,m\}$, $\rho(v_{i})$ is a uniform matroid of rank $1$ or corank $1$.
\item\label{item:L2} $\rho(x_{1}),\rho(x_{m})\in\mathcal{G}-\mathcal{G}^{*}$ and $\rho(y_{1}),\rho(y_{m})\in\mathcal{G}^{*}-\mathcal{G}$.
\item\label{item:L3} $\rho(w)\in\mathcal{G}-\mathcal{G}^{*}$ if $r(\rho(v_{k}))=1$ and $\rho(w)\in\mathcal{G}^{*}-\mathcal{G}$ if $r^{*}(\rho(v_{k}))=1$.
\end{enumerate}

\begin{lemma}
\label{ladder}
No $(m,k)$-bench is delta-graphic.
\end{lemma}

\begin{proof}
Let $M$ be an $(m,k)$-bench with a canonical tree decomposition $(T,\rho)$. Suppose that $M$ is delta-graphic. By \ref{item:L2}, $M\notin\mathcal{G}\cup\mathcal{G}^{*}$ and so by Proposition~\ref{prop:deltagraphictree}, $(T,\rho)$ satisfies at least one of \ref{item:t1}, \ref{item:t2}, \ref{item:t3}, and \ref{item:t4}. For each $e\in E(T)$, there is a component $C$ of $T\setminus e$ such that $\{x_{1},y_{1}\}\subseteq V(C)$ or $\{x_{m},y_{m}\}\subseteq V(C)$. So $\rho(C)\notin\mathcal{G}\cup\mathcal{G}^{*}$ and $(T,\rho)$ does not satisfy \ref{item:t1}, \ref{item:t2}, or \ref{item:t3}. Suppose that $(T,\rho)$ satisfies \ref{item:t4} with a spine $P$. Then by \ref{item:f1} and \ref{item:f2}, $P=v_{1}v_{2}\cdots v_{m}$ and $\rho(w)\in\mathcal{G}$ if $r^{*}(\rho(v_{k}))=1$ and $\rho(w)\in\mathcal{G}^{*}$ if $r(\rho(v_{k}))=1$, contradicting \ref{item:L3}. So $(T,\rho)$ does not satisfy \ref{item:t4}, contradicting our assumption. 
\end{proof}

\begin{lemma}
\label{lem:bench}
Let $M$ be a connected matroid  minor-minimally not delta-graphic.
Let $(T,\rho)$ be a canonical tree decomposition of $M$.
If $\rho(v)$ is a uniform matroid of rank $1$ or corank $1$ for each internal vertex $v$, then $|E(M)|\leq 47$.
\end{lemma}

\begin{proof}
If $M$ is $3$-connected, then $M$ has a minor isomorphic to one of $U_{2,4}$, $F_{7}$, $F_{7}^{*}$, $R_{10}$, $R_{12}$ by Lemma~\ref{3connected} and $|E(M)|\leq 12$ by minimality of $M$. So we may assume that $M$ is not $3$-connected.

If $M$ has a flexible triple, then $|E(M)|\leq 30$ by Lemma~\ref{lem:flexible}. So we may assume that $M$ has no flexible triples.
By Lemma~\ref{GorCG}, $\rho(v)\in\mathcal{G}\cup\mathcal{G}^{*}$ for each $v\in V(T)$. By Lemma~\ref{planar}, $\rho(v)\notin\mathcal{G}\cap\mathcal{G}^{*}$ for each leaf $v$ of $T$.
Let $X$ be the set of internal vertices of $T$ and let $P=v_{1}v_{2}\cdots v_{m}$ be a longest path of $T[X]$. Let $e_{i}=v_{i}v_{i+1}$ for $i\in\{1,\ldots,m-1\}$. Then $m\geq 3$ because otherwise $M\in\mathcal{G}\cup\mathcal{G}^{*}$ or $(T,\rho)$ satisfies \ref{item:t2} or \ref{item:t4}, implying that $M$ is delta-graphic, contradicting our assumption. 

By taking dual, we may assume that $r(\rho(v_{1}))=1$. Since $P$ is a longest path, each vertex not in $P$ but adjacent to an end of $P$ is a leaf of $T$. There are leaves $x_{1},y_{1}\in N_{T}(v_{1})$ such that $\rho(x_{1})\in\mathcal{G}-\mathcal{G}^{*}$ and $\rho(y_{1})\in\mathcal{G}^{*}-\mathcal{G}$ because otherwise by Lemma~\ref{GorCG}, $T_{e_{1}}(v_{2})\in\mathcal{G}\cup\mathcal{G}^{*}$ and $v_{1}$ is a leaf of $T$ by Lemma~\ref{Reducingsubtree}, contradicting our assumption. Similarly, there are leaves $x_{m},y_{m}\in N_{T}(v_{m})$ such that $\rho(x_{m})\in\mathcal{G}-\mathcal{G}^{*}$ and $\rho(y_{m})\in\mathcal{G}^{*}-\mathcal{G}$. 

Since $(T,\rho)$ does not satisfy \ref{item:t4}, there exist $k\in\{2,\ldots,m-1\}$ and an edge $e\notin E(P)$ incident with $v_{k}$ such that $T_{e}(v_{k})\notin\mathcal{G}^{*}$, $r(\rho(v_{k}))=1$, or $T_{e}(v_{k})\notin\mathcal{G}$, $r^{*}(\rho(v_{k}))=1$. By Lemmas~\ref{GorCG} and \ref{planar}, there is a leaf $w\in V(T_{e}(v_{k}))$ of $T$ such that $\rho(w)\in\mathcal{G}-\mathcal{G}^{*}$ if $T_{e}(v_{k})\notin\mathcal{G}^{*}$ and $\rho(w)\in\mathcal{G}^{*}-\mathcal{G}$ if $T_{e}(v_{k})\notin\mathcal{G}$. Let $P_{1}:=T_{v_{k},w}$ and $e$, $f$ be edges of $P_{1}$ incident with $w$, $v_{k}$ respectively. Let 
\[
k^{-}=
\begin{cases}
k-1 & \text{if $k\equiv 1 \pmod{2}$,} \\
k & \text{otherwise} 
\end{cases}
\] and
\[
k^{+}=
\begin{cases}
k+1 & \text{if $k\equiv m \pmod{2}$,} \\
k & \text{otherwise.} 
\end{cases}
\]
Observe that $k^{-}\neq 1$ and $k^{+}\neq m$.
Let $e'=e_{1}$, $f'=e_{k^{-}-1}$, $f''=e_{k^{+}}$, $e''=e_{m-1}$.
Let $T_{1}$ be a minimal subtree of $T$ containing $\{w,x_{1},x_{m},y_{1},y_{m}\}$ and let $T'$ be a tree obtained from $T_{1}$ by contracting edges in $\{e_{i}:\text{$2\leq i< k^{-}-1$ or $k^{+}\leq i<m-1$}\}\cup (E(P_{1})-\{e\})$.

If $k^{-}=k-1$, then let $N_{k-1}$ be a matroid obtained from $\rho(v_{k-1})$ by relabelling $f'$ to $e'$. If $k^{+}=k+1$, then let $N_{k+1}$ be a matroid obtained from $\rho(v_{k+1})$ by relabelling $f''$ to $e''$. Let $N_{k}$ be a matroid obtained from $\rho(v_{k})$ by relabelling $f'$ to $e'$ if $f'\in E(\rho(v_{k}))$, relabelling $f''$ to $e''$ if $f''\in E(\rho(v_{k}))$, and relabelling $f$ to $e$. For each $z\in V(T')$, let
\[
\rho'(z)=
\begin{cases}
\rho(z) & \text{if $z\in V(T')-\{v_{k^{-}},v_{k},v_{k^{+}}\}$,} \\
N_{i} & \text{if $z=v_{i}$ and $k^{-}\leq i\leq k^{+}$.}
\end{cases}
\]

By Lemmas~\ref{subtreeminor} and \ref{pathcontraction}, $\rho'(T')$ is isomorphic to a minor of $\rho(T)$ and $\rho'(T')$ is isomorphic to a $(3+(k^{+}-k^{-}),2+(k-k^{-}))$-bench which is not delta-graphic by Lemma~\ref{ladder}. By the minimality of $M$, $\rho'(T')$ is isomorphic to $M$. By Lemma~\ref{ReducingVertex}, for each $z\in Z:=\{v_{1},v_{m}\}\cup\{v_{i}:k^{-}\leq i\leq k^{+}\}$, $|E(\rho(z))|=3$ because $\deg_{T'}(z)\leq 3$. Therefore, by Lemma~\ref{ReducingVertex} and \ref{Reducingleaf},
\begin{align*}
|E(M)|&=\sum_{z\in\{x_{1},y_{1},x_{m},y_{m},w\}}|E(\rho(z))|+\sum_{z\in Z}|E(\rho(z))|-(14+2(k^{+}-k^{-})) \\
&\leq 50+3(3+(k^{+}-k^{-}))-(14+2(k^{+}-k^{-}))=50-5+(k^{+}-k^{-})\leq 47. \qedhere
\end{align*}
\end{proof}

A connected matroid $M$ is \emph{starlike} if $M$ has a canonical tree decomposition $(T,\rho)$ such that 
\begin{itemize}
\item $T$ is isomorphic to $K_{1,4}$ with an internal vertex $v$ and leaves $x_{1}$, $x_{2}$, $x_{3}$, $x_{4}$,
\item $\rho(v)\in\mathcal{G}\cap\mathcal{G}^{*}$ and $\rho(v)$ is neither a uniform matroid of rank $1$ nor a uniform matroid of corank $1$, and
\item $\rho(x_{1}),\rho(x_{2})\in\mathcal{G}-\mathcal{G}^{*}$ and $\rho(x_{3}),\rho(x_{4})\in\mathcal{G}^{*}-\mathcal{G}$.
\end{itemize}

\begin{lemma}
\label{starlike}
Let $M$ be a connected matroid which is minor-minimally not delta-graphic. Then, $M$ is starlike or $|E(M)|\leq 47$.
\end{lemma}

\begin{proof}
Let $(T,\rho)$ be a canonical tree decomposition of $M$.
If $(T,\rho)$ has a flexible triple or a core, then $|E(M)|\leq 40$ by Lemmas~\ref{lem:flexible} and \ref{lem:core}. So we may assume that $M$ has no flexible triples and no cores. So by Lemma~\ref{notriaxialstrong}, $\rho(v)\in\mathcal{G}\cap\mathcal{G}^{*}$ for every internal vertex $v$. If $\rho(v)$ is a uniform matroid of rank $1$ or corank $1$ for each internal vertex $v$ of $T$, then by Lemma~\ref{lem:bench}, $|E(M)|\leq 47$. Therefore, we can assume that $T$ has an internal vertex $v$ such that $\rho(v)$ is neither a uniform matroid of rank $1$ nor a uniform matroid of corank $1$. 

Let $X_{1},\ldots, X_{m}$ be the components of $T\setminus v$ and let $C_{1}=\{X_{i}:\rho(X_{i})\notin\mathcal{G}^{*}\}$ and $C_{2}=\{X_{i}:\rho(X_{i})\notin\mathcal{G}\}$. Since $v$ is not a leaf, by Lemma~\ref{Reducingsubtree}, $|C_{1}|,|C_{2}|\geq 2$. We have $|C_{1}\cap C_{2}|\leq 1$ because otherwise \ref{item:R1} holds and $v$ is a core. Since \ref{item:R2} does not hold, we have $C_{1}\cap C_{2}=\emptyset$ and therefore $\rho(X)\in\mathcal{G}\cup\mathcal{G}^{*}$ for each component $X$ of $T\setminus v$. 

By Lemma~\ref{Reducingsubtree}, every component $X_{i}$ of $T\setminus v$ has only one vertex $x_{i}$. Therefore, $T$ is isomorphic to $K_{1,m}$. For each $i\in\{1,\ldots,m\}$, let $e_{i}$ be an edge $vx_{i}$ in $T$.

Suppose that $|C_{1}|\geq 3$. By symmetry, we may assume that $X_{1},X_{2}\in C_{1}$. Let $T_{1}=T\setminus x_{1}$ and $\rho_{1}:=\rho|_{V(T_{1})}$. By Lemma~\ref{subtreeminor}, $\rho(T_{1})$ is isomorphic to a minor of $M$ with a canonical tree decomposition $(T_{1},\rho_{1})$. So $\rho(T_{1})$ is delta-graphic. Since $|C_{1}-\{X_{1}\}|,|C_{2}|\geq 2$, $\rho(T_{1})\notin\mathcal{G}\cup\mathcal{G}^{*}$ and $(T_{1},\rho_{1})$ does not satisfy \ref{item:t1}. Since there is no vertex $u$ such that $\rho(u)$ is a uniform matroid of rank $1$ or corank $1$, $(T_{1},\rho_{1})$ satisfies neither \ref{item:t2} nor \ref{item:t4}. Therefore, $(T_{1},\rho_{1})$ is a wheel decomposition with the hub~$v$. So $\rho(v)$ is isomorphic to $M(W_{k})$ for some $k\geq 3$ and there is a circuit-hyperplane $C$ of $\rho(v)$ such that for $i\in\{2,\ldots,m\}$, $X_{i}\in C_{1}$ if $e_{i}\in C$ and $X_{i}\in C_{2}$ if $e_{i}\notin C$. If $e_{1}\in C$, then $(T,\rho)$ is a wheel decomposition and, by Proposition~\ref{prop:deltagraphictree}, $M$ is delta-graphic, contradicting our assumption. So $e_{1}\notin C$. 

Now let $T_{2}=T\setminus x_{2}$ and $\rho_{2}=\rho|_{V(T_{2})}$. Then, by Lemma~\ref{subtreeminor}, $\rho(T_{2})$ is a delta-graphic matroid with a canonical tree decomposition $(T_{2},\rho_{2})$. 
By the same reason, $(T_{2},\rho_{2})$ is a wheel decomposition with hub~$v$
and therefore $\rho(v)$ has a circuit-hyperplane $D$ such that for $i\in\{1,2, \ldots,m\}-\{2\}$,  
$X_{i}\in C_{1}$ if $e_{i}\in D$
and 
$X_{i}\in C_{2}$ if $e_{i}\notin D$.
In particular, $e_1\in D$ and so $D\neq C$. 
Since $\rho(v)$ is isomorphic to $M(W_{k})$ and both $C$ and $D$ are circuit-hyperplanes, we deduce that $k=3$.
Observe from $M(W_3)$ that $\abs{C\triangle D}>2$ and therefore there is $e_i\in C\triangle D$ with $i\neq 1,2$. This implies that $X_i\in C_1\cap C_2$, contradicting the fact that $C_1\cap C_2=\emptyset$.
Thus we conclude that $|C_{1}|=2$. Similarly, $|C_{2}|=2$. Therefore $M$ is starlike.
\end{proof}

\begin{lemma}
\label{lem:properwheel}
Let $M$ be a connected matroid which is starlike and is minor-minimally not delta-graphic. 
Let $(T,\rho)$ be a canonical tree decomposition of $M$.
If $v$ is the internal vertex of $T$ and $N$ is a $3$-connected proper minor of $\rho(v)$ containing $E(T)$, then $N$ is isomorphic to $M(W_{k})$ for some $k\geq 3$.
\end{lemma}

\begin{proof}
Let $N$ be a $3$-connected proper minor of $\rho(v)$ containing $E(T)$ and, for each vertex $w\in V(T)$, let 
\[
\rho'(w)=
\begin{cases}
\rho(w) & \text{if $w\in V(T)-\{v\}$,} \\
N & \text{if $w=v$.}
\end{cases}
\]
Then $\rho'(T)$ is a proper minor of $M$ with a canonical tree decomposition $(T,\rho')$ and therefore it is delta-graphic. Since $\rho'(T)$ is starlike, $\rho'(T)\notin\mathcal{G}\cup\mathcal{G}^{*}$ and $(T,\rho')$ does not satisfy \ref{item:t1}, \ref{item:t2}, or \ref{item:t4}. So by Proposition~\ref{prop:deltagraphictree}, $(T,\rho')$ is a wheel decomposition and therefore $N$ isomorphic to $M(W_{k})$ for some $k\geq 3$.
\end{proof}

For graphs $G$ and $G'$, $G$ is obtained from $G'$ by \emph{$1$-bridging} if, for a vertex $v$ and an edge $e=xy$ of~$G'$ which is not incident with $v$, $G$ is obtained from $G'\setminus e$ by adding a vertex $u$ joining $v$, $x$, and $y$.

Moreover, $G$ is obtained from $G'$ by \emph{$2$-bridging}, if, for two distinct edges $e_{1}=x_{1}y_{1}$, $e_{2}=x_{2}y_{2}$ of~$G'$, $G$ is obtained from $G'\setminus\{e_{1},e_{2}\}$ by adding vertices $u_{1}$, $u_{2}$ such that $N_{G}(u_{1})=\{u_{2},x_{1},y_{1}\}$ and $N_{G}(u_{2})=\{u_{1},x_{2},y_{2}\}$.

\begin{lemma}[Kelmans~\cite{Kelmans1978}] 
\label{lem:kelmans}
Let $G$ and $H$ be simple $3$-connected graphs. Then $G$ contains an $H$-subdivision as a subgraph if and only if $G$ can be obtained from $H$ by a finite sequence of adding an edge, $1$-bridging, and $2$-bridging.
\end{lemma}

\begin{lemma}
\label{lem:extwheel}
Let $G$ be a planar graph with an edge $e=ab$ such that $G/e$ is isomorphic to $W_{n}$ for some $n\geq 4$ and $\deg_{G}(a),\deg_{G}(b)\geq 3$. Let $X$ be a subset of $E(G)-\{e\}$ with $|X|=4$. Then $G$ has a minor $H$ containing $X\cup\{e\}$ such that $H/e$ is isomorphic to $W_{m}$ for some $m\leq 7$ and both ends of $e$ have degree at least $3$.
\end{lemma}

\begin{proof}
We proceed by induction on $|E(G)|$. We may assume that $n\geq 8$. 
Recall that $t_{1}$, $t_{2}$, $\ldots$, $t_{n}$, $s$ are the vertices of $W_{n}$ and $t_{n+1}:=t_{1}$ and, for each $i\in\{1,\ldots,n\}$, $e_{2i-1}=t_{i}t_{i+1}$ and $e_{2i}=st_{i+1}$. By relabelling edges, we may assume that $G/e=W_{n}$.

Since $\deg_{G}(a),\deg_{G}(b)\geq 3$, the vertex of $W_{n}$ obtained by contracting $e$ is  the center $s$. Hence, $G\setminus\{a,b\}$ is a cycle $C$ and $(N_{G}(a)\cup N_{G}(b))-\{a,b\}=V(C)$, and $N_{G}(a)\cap N_{G}(b)=\emptyset$. It is easy to check that $G$ is $3$-connected. For two vertices $x,y\in N_{G}(a)-\{b\}$, if each subpath of $C$ from $x$ to $y$ has a neighbor of $b$ in $G$, then $G$ has a minor isomorphic to $K_{3,3}$, contradicting our assumption. So $C[N_{G}(a)-\{b\}]$ and $C[N_{G}(b)-\{a\}]$ are paths. By symmetry, we may assume that $N_{G}(a)-\{b\}=\{t_{1},t_{2},\ldots,t_{k}\}$ and $N_{G}(b)-\{a\}=\{t_{k+1},\ldots,t_{n}\}$ for some $2\leq k\leq n-2$. For each $i\in\{1,2,\ldots,n\}$, let $A_{i}=\{e_{2i-1},e_{2i}\}$.

By symmetry, we may assume that $k\geq n-k$. Then $k\geq 4$. If $k\geq 6$, then there exists $i\in\{1,\ldots,k-1\}$ such that $A_{i}\cap X=\emptyset$ because $|X|=4$. Let $G'=G\setminus e_{2i}/e_{2i-1}$. Then $G'/e$ is isomorphic to $W_{n-1}$ with $\deg_{G'}(a),\deg_{G'}(b)\geq 3$. So by induction hypothesis, $G'$ has a minor $H$ containing $X\cup\{e\}$ such that $H/e$ is isomorphic to $W_{m}$ for some $m\leq 7$ and both ends of $e$ have degree at least $3$ and so does $G$. 

Hence, we may assume that $k\in\{4,5\}$ and $n-k\geq 8-5=3$. Since $|X|=4$, there exists $j\in\{1,\ldots,n\}$ such that $A_{j}\cap X=\emptyset$. Let $G'=G\setminus e_{2j}/e_{2j-1}$. Then $G'/e$ is isomorphic to $W_{n-1}$ with $\deg_{G'}(a),\deg_{G'}(b)\geq 3$. So by induction hypothesis, $G'$ has a minor $H$ containing $X\cup\{e\}$ such that $H/e$ is isomorphic to $W_{m}$ for some $m\leq 7$ and both ends of $e$ have degree at least $3$ and $H$ is also a minor of $G$. 
\end{proof}

\begin{lemma}
\label{lem:wheelkelman}
Let $G$ be a simple $3$-connected graph which is isomorphic to a graph obtained from $W_{n}$ for some $n\geq 4$ by adding an edge, $1$-bridging, or $2$-bridging. If $G$ is not isomorphic to a wheel graph and $X$ is a subset of $E(G)$ with $|X|=4$, then $G$ has a simple $3$-connected minor $H$ containing $X$ such that $H$ is not isomorphic to a wheel graph and $|E(H)|\leq 16$.
\end{lemma}

\begin{proof}
We prove the lemma by induction on $|E(G)|$. We may assume that $|E(G)|>16$.
By relabelling edges, we may assume that $G$ is a graph obtained from $W_{n}$ by adding an edge, $1$-bridging, or $2$-bridging.
 
Recall that $t_{1}$, $t_{2}$, $\ldots$, $t_{n}$, $s$ are the vertices of $W_{n}$ and $t_{n+1}:=t_{1}$. For $i\in\{1,2,\ldots,n\}$, let $A_{i}=\{t_{i}t_{i+1},st_{i+1}\}$. 

If $G$ is obtained from $W_{n}$ by adding an edge $e=xy$, then $x,y\neq s$ because $G$ is simple. By symmetry, we may assume that $x=t_{1}$ and $y=t_{k}$ for $3\leq k\leq\frac{n}{2}+1$. Then, since $|X|=4$ and $n\geq 8$, there exists $j\in\{3,\ldots,n\}$ such that $A_{j}\cap X=\emptyset$. Let $G'=G/t_{j}t_{j+1}\setminus st_{j+1}$. Then $G'$ is a simple $3$-connected graph which contains $X$ and is isomorphic to a graph obtained from $W_{n-1}$ by adding an edge. By induction hypothesis, $G'$ has a simple $3$-connected minor $H$ containing $X$ such that $H$ is not isomorphic to a wheel graph and $|E(H)|\leq 16$, and so does $G$.

If $G$ is obtained from $W_{n}$ by $1$-bridging for a vertex $v$ and an edge $e=xy$ which is not incident with $v$, then $v\neq s$ since $G$ is not isomorphic to a wheel graph. By symmetry, we may assume that $v=t_{1}$. Let $e\in A_{h}$ for some $h\in\{1,\ldots,n\}$. Then there exists $j\in\{1,\ldots,n\}-\{1,h,n\}$ such that $A_{j}\cap X=\emptyset$ because $n\geq 8$ and $|X|=4$. Let $G'=G/t_{j}t_{j+1}\setminus st_{j+1}$.  Then $G'$ is a simple $3$-connected graph which contains $X$ and is isomorphic to a graph obtained from $W_{n-1}$ by $1$-bridging. So by induction hypothesis, $G'$ has a simple $3$-connected minor $H$ containing $X$ such that $H$ is not isomorphic to a wheel graph and $|E(H)|\leq 16$, and so does $G$.

Now suppose that $G$ is obtained from $W_{n}$ by $2$-bridging for two edges $f$, $g$. 
By rotational symmetry, we may assume that $f\in A_{1}$ and $g\in A_{k}$ for some $k\in\{1,2,\ldots,n\}$. Then, there exists $j\in\{1,2,\ldots,n\}-\{1,k\}$ such that $A_{j}\cap X=\emptyset$ since $|X|=4$ and $n\geq 7$. Let $G'=G/t_{j}t_{j+1}\setminus st_{j+1}$.  Then $G'$ is a simple $3$-connected graph which contains $X$ and is isomorphic to a graph obtained from $W_{n-1}$ by $2$-bridging. So by induction hypothesis, $G'$ has a simple $3$-connected minor $H$ containing $X$ such that $H$ is not isomorphic to a wheel graph and $|E(H)|\leq 16$, and so does $G$.
\end{proof} 

\begin{lemma}
\label{lem:notwheel}
Let $G$ be a simple $3$-connected planar graph which is not isomorphic to a wheel graph and $X$ be a subset of $E(G)$ with $|X|=4$. If $G$ has a minor which contains $X$ and is isomorphic to a wheel graph, then there is a simple $3$-connected minor $H$ of $G$ containing $X$ such that $|E(H)|\leq 16$ and $H$ is not isomorphic to a wheel graph.
\end{lemma}

\begin{proof}
Let $W$ be a minor of $G$ which contains $X$ and is isomorphic to a wheel graph with maximum $|V(W)|$. By relabelling edges, we may assume that $W=W_{n}$ for some $n\geq 3$. If $n=3$, then by Lemma~\ref{lem:Truemper2}, there is a simple $3$-connected minor $H$ of $G$ such that $W_{3}$ is a minor of $H$ and $|E(H)|\leq |E(W_{3})|+3\leq 9$. By maximality of $W_{3}$, $H$ is not isomorphic to a wheel. 

So we may assume that $n\geq 4$. Let $s$ be the center of $W_{n}$. Since $W_{n}$ is a minor of $G$ containing $X$, there is a function $\alpha$ that maps $V(W_{n})$ to vertex-disjoint connected subgraphs of $G$ and maps each edge $xy$ of $W_{n}$ to an edge joining $\alpha(x)$ and $\alpha(y)$ in $G$ such that $\alpha(xy)\in X$ if and only if $xy\in X$. We choose $\alpha$ with minimum $\sum_{v\in V(W_{n})}|E(\alpha(v))|$. Then, for every vertex $v$ of $W_{n}$, $\alpha(v)$ is a tree and 
\begin{equation}
\label{equ:leaf}
\text{every leaf of $\alpha(v)$ is an end of $\alpha(e)$ for some edge $e$ of $W_{n}$.}
\end{equation}
So the number of leaves of $\alpha(v)$ is at most $\deg_{W_{n}}(v)$ for each $v\in V(W_{n})$. Therefore, for each vertex $v\neq s$ of $W_{n}$, $\alpha(v)$ is a subdivision of a star. 

Let $T$ be a connected subgraph of $G$ whose edge set is $\bigcup_{v\in V(W_{n})}E(\alpha(v))\cup\bigcup_{xy\in E(W_{n})}\alpha(xy)$.
Suppose that $\alpha(s)$ has two vertices $u_{1}$, $u_{2}$ such that $\deg_{T}(u_{1})$, $\deg_{T}(u_{2})\geq 3$. 
Let $g$ be the edge of a path of $\alpha(s)$ from $u_{1}$ to $u_{2}$. Let $H'$ be a graph obtained from $T$ by contracting $(\bigcup_{v\in V(W_{n})}E(\alpha(v)))-\{g\}$. Since $G$ is planar, $H'$ is planar. Since $H'/g=W_{n}$, there are two ends $a$, $b$ of $g$ such that $H'\setminus\{a,b\}$ is a cycle $C$, $(N_{H'}(a)\cup N_{H'}(b))-\{a,b\}=V(C)$, and $N_{H'}(a)\cap N_{H'}(b)=\emptyset$. By~\eqref{equ:leaf} and the fact that $\deg_{T}(u_{1}),\deg_{T}(u_{2})\geq 3$, we have $\deg_{H'}(a), \deg_{H'}(b)\geq 3$. Therefore, by Lemma~\ref{lem:extwheel}, $G$ has a minor~$H$ containing $X\cup\{g\}$ such that $H/g$ is isomorphic to $W_{m}$ for $m\leq 7$ and both ends of $g$ have degree at least $3$. This graph $H$ is what we are looking for because $H$ is $3$-connected and $|E(H)|\leq 15$.

So $\alpha(s)$ has at most $1$ vertex $v$ such that $\deg_{T}(v)\geq 3$. Then $T$ is a subdivision of $W_{n}$. 

Then, by Lemma~\ref{lem:kelmans}, $G$ is obtained from $W_{n}$ by a finite sequence of adding an edge, $1$-bridging, and $2$-bridging. 
Let $F$ be a simple $3$-connected graph obtained from $W_n$
by applying the first operation in the sequence.
Since $W_n$ is a largest minor containing $X$ isomorphic to a wheel graph,
$F$ is not isomorphic to a wheel graph.
So by Lemma~\ref{lem:wheelkelman}, $F$ has a simple $3$-connected minor $H$ containing $X$ such that $H$ is not isomorphic to a wheel graph and $|E(H)|\leq 16$.
\end{proof}

Now we are ready to prove Theorem~\ref{thm:exclude}.

\begin{proof}[Proof of Theorem~\ref{thm:exclude}]
Let $(T,\rho)$ be a canonical tree decomposition of $M$.
By Lemma~\ref{starlike}, $M$ is starlike or $|E(M)|\leq 47$. Suppose that $M$ is starlike. Let $v$ be the internal vertex of $T$ and $x_{1}$, $x_{2}$, $x_{3}$, $x_{4}$ be leaves of $T$. Suppose that $\rho(v)$ is isomorphic to $M(W_{k})$ for some $k\geq 3$. If $k\leq 4$, then, by Lemma~\ref{Reducingleaf}, we have
\[
|E(M)|=\sum_{i=1}^{4}|E(\rho(x_{i}))|+|E(\rho(v))|-8\leq 40+8-8=40.
\]

So we may assume that $k\geq 5$. Then there is a minor $N$ of $\rho(v)$ containing $E(T)$ such that $N$ is isomorphic to $M(W_{k-1})$. For each vertex $w$ of $T$, let
\[
\rho_{1}(w)=
\begin{cases}
\rho(w) & \text{if $w\in V(T)-\{v\}$,} \\
N & \text{if $w=v$.}
\end{cases}
\]
Then $\rho_{1}(T)$ is a proper minor of $M$ and is delta-graphic. Since $(T,\rho_{1})$ is starlike, $\rho_{1}(T)\notin\mathcal{G}\cup\mathcal{G}^{*}$ and $(T,\rho_{1})$ does not satisfy \ref{item:t1}, \ref{item:t2}, or \ref{item:t4}. So by Proposition~\ref{prop:deltagraphictree}, $(T,\rho_{1})$ is a wheel decomposition and so is $(T,\rho)$, which implies that $M$ is delta-graphic, contradicting our assumption.

So $\rho(v)$ is not isomorphic to a wheel graph. For each $i\in\{1,2,3,4\}$, let $e_{i}=vx_{i}$. Since $\rho(v)$ is not isomorphic to $U_{1,3}$ or $U_{2,3}$, we have $r(\rho(v)),r^{*}(\rho(v))\geq 3$ by Lemma~\ref{lem:rk3}. So by Lemma~\ref{MK4minor}, $\rho(v)$ has a minor $N_{1}$ containing $e_{1}$, $e_{2}$, $e_{3}$ isomorphic to $M(K_{4})$. Then, by applying Lemma~\ref{lem: bixby}, $\rho(v)$ has a $3$-connected minor $N_{2}$ containing $E(T)$ such that $N_{1}$ is a minor of $N_{2}$ and $|E(N_{2})|\leq 6+4=10$. 
If $|E(\rho(v))|\leq 16$, then we have 
\[
|E(M)|=\sum_{i=1}^{4}|E(\rho(x_{i}))|+|E(\rho(v))|-8\leq 40+16-8=48. 
\]
Suppose that $|E(\rho(v))|>16$. Then $N_{2}$ is a proper minor of $\rho(v)$. So by Lemma~\ref{lem:properwheel}, $N_{2}$ is isomorphic to a wheel matroid.

Let $G$ be a simple $3$-connected planar graph such that $M(G)=\rho(v)$ and $G'$ be a simple $3$-connected minor of $G$ with $M(G')=N_{2}$. Since $N_{2}$ is $3$-connected, $G'$ is isomorphic to a wheel graph by Lemma~\ref{lem:whitney}.
So by Lemma~\ref{lem:notwheel}, $G$ has a simple $3$-connected minor $G''$ containing $E(T)$ such that $|E(G'')|\leq 16$ and $G''$ is not isomorphic to a wheel graph. So $\rho(v)$ has a proper minor $M(G'')$ containing $E(T)$ which is not isomorphic to a wheel graph, contradicting Lemma~\ref{lem:properwheel}.
\end{proof}


\end{document}